\documentclass[reqno, a4paper, 11pt]{article}
\usepackage{amsfonts, amsmath, amssymb, amsthm, color}
\usepackage[hmargin=2.25cm, vmargin=2.25cm]{geometry}
\usepackage{txfonts}
\usepackage{graphicx, setspace}
\usepackage{epstopdf}
\usepackage[pagebackref]{hyperref}

\newtheorem{thm}{Theorem}[section]
\newtheorem{lemma}[thm]{Lemma}
\newtheorem{prop}[thm]{Proposition}

\theoremstyle{definition}
\newtheorem{dfn}[thm]{Definition}
\newtheorem{rmk}[thm]{Remark}

\newcommand{\ep}{\epsilon}
\newcommand{\vr}{\varrho}
\newcommand{\pa}{\partial}

\newcommand{\ma}{\mathcal{A}}
\newcommand{\mb}{\mathcal{B}}
\newcommand{\mc}{\mathcal{C}}
\newcommand{\md}{\mathcal{D}}
\newcommand{\me}{\mathcal{E}}
\newcommand{\mg}{\mathcal{G}}
\newcommand{\mh}{\mathcal{H}}
\newcommand{\ml}{\mathcal{L}}
\newcommand{\mq}{\mathcal{Q}}
\newcommand{\mcr}{\mathcal{R}}
\newcommand{\mw}{\mathcal{W}}
\newcommand{\mz}{\mathcal{Z}}

\newcommand{\mn}{\mathbb{N}}
\newcommand{\mr}{\mathbb{R}}
\newcommand{\omrn}{\overline{\mr^N_+}}

\newcommand{\bg}{\bar{g}}
\newcommand{\bx}{\bar{x}}
\newcommand{\by}{\bar{y}}

\newcommand{\hh}{\hat{h}}
\newcommand{\hw}{\hat{w}}
\newcommand{\whw}{\widehat{W}}

\newcommand{\tg}{\tilde{g}}
\newcommand{\tm}{\tilde{m}}
\newcommand{\tz}{\tilde{z}}

\newcommand{\wtg}{\widetilde{G}}
\newcommand{\wtp}{\widetilde{P}}
\newcommand{\wtu}{\widetilde{U}}
\newcommand{\wtw}{\widetilde{W}}
\newcommand{\wtps}{\widetilde{\Psi}}

\newcommand{\oh}{\overline{H}}

\newcommand{\drn}{D^{1,2}(\mr_+^N; x_N^{1-2\gamma})}

\newcommand{\dz}{{\delta, 0}}
\newcommand{\dt}{\delta, \tau}
\newcommand{\edt}{\ep \delta, \ep \tau}
\newcommand{\edte}{\ep \delta(\ep), \ep \tau(\ep)}
\newcommand{\ls}{{\lambda, \sigma}}

\newcommand{\la}{\left\langle}
\newcommand{\ra}{\right\rangle}

\renewcommand{\(}{\left(}
\renewcommand{\)}{\right)}

\begin{document}
\title{A non-compactness result on the fractional Yamabe problem \\ in large dimensions}

\author{Seunghyeok Kim, Monica Musso, Juncheng Wei}

\newcommand{\Addresses}{{\bigskip \footnotesize

\medskip
\noindent (Seunghyeok Kim) \textsc{Facultad de Matem\'{a}ticas, Pontificia Universidad Cat\'{o}lica de Chile, Avenida Vicu\~{n}a Mackenna 4860, Santiago, Chile}\par\nopagebreak
\noindent \textit{E-mail address}: \texttt{shkim0401@gmail.com}

\medskip
\noindent (Monica Musso) \textsc{Facultad de Matem\'{a}ticas, Pontificia Universidad Cat\'{o}lica de Chile, Avenida Vicu\~{n}a Mackenna 4860, Santiago, Chile}\par\nopagebreak
\noindent \textit{E-mail address}: \texttt{mmusso@mat.puc.cl}

\medskip
\noindent (Juncheng Wei) \textsc{Department of Mathematics, University of British Columbia, Vancouver, B.C., Canada, V6T 1Z2 and Department of Mathematics, Chinese University of Hong Kong, Shatin, NT, Hong Kong}
\par\nopagebreak
\noindent \textit{E-mail address}: \texttt{jcwei@math.ubc.ca}}}

\date{\today}
\maketitle

\begin{abstract}
Let $(X^{n+1}, g^+)$ be an $(n+1)$-dimensional asymptotically hyperbolic manifold with a conformal infinity $(M^n, [\hh])$.
The fractional Yamabe problem addresses to solve
\[P^{\gamma}[g^+,\hh] (u) = cu^{n+2\gamma \over n-2\gamma}, \quad u > 0 \quad \text{on } M\]
where $c \in \mr$ and $P^{\gamma}[g^+,\hh]$ is the fractional conformal Laplacian whose principal symbol is $(-\Delta)^{\gamma}$.
In this paper, we construct a metric on the half space $X = \mr^{n+1}_+$, which is conformally equivalent to the unit ball,
for which the solution set of the fractional Yamabe equation is non-compact
provided that $n \ge 24$ for $\gamma \in (0, \gamma^*)$ and $n \ge 25$ for $\gamma \in [\gamma^*,1)$ where $\gamma^* \in (0, 1)$ is a certain transition exponent.
The value of $\gamma^*$ turns out to be approximately 0.940197.
\end{abstract}



\allowdisplaybreaks
\numberwithin{equation}{section}
{\footnotesize \tableofcontents}

\section{Introduction}\label{sec-intro}
Given $n \in \mn$, let $(X^{n+1}, g^+)$ be an $(n+1)$-dimensional asymptotically hyperbolic manifold with a conformal infinity $(M^n, [\hh])$.
In \cite{GZ}, Graham and Zworski introduced the fractional conformal Laplacian $P^{\gamma}_{\hh} = P^{\gamma}[g^+, \hh]$ for $\gamma \in (0,n/2)$
whose principal symbol is given as $(-\Delta)^{\gamma}$ and which obeys the conformal covariance property:
\begin{equation}\label{eq-conf-cov}
P^{\gamma}\left[g^+,w^{4 \over n-2\gamma}\hh\right] = w^{-{n+2\gamma \over n-2\gamma}}P^{\gamma}\left[g^+,\hh\right](w\cdot)
\end{equation}
holds for any positive function $w$ on $M$.
If we denote by $Q^{\gamma}_{\hh} = P^{\gamma}_{\hh}(1)$ the associated fractional scalar curvature and further assume that $(X, g^+)$ is a Poincar\'{e}-Einstein manifold,
then $P^1_{\hh}$ and $Q^1_{\hh}$ become the conformal Laplacian and the scalar curvature (up to constant multiples)
\[P^1_{\hh} = -\Delta_{\hh} + {n-2 \over 4(n-1)}R_{\hh}, \quad Q^1_{\hh} = {n-2 \over 4(n-1)}R_{\hh}\]
respectively,
while $P^2_{\hh}$ and $Q^2_{\hh}$ coincide the Paneitz operator and Branson's Q-curvature
\begin{align*}
P^2_{\hh} &= \Delta_{\hh}^2 + \text{div}_{\hh}\(a_{1n}R_{\hh}\hh + a_{2n} \text{Ric}_{\hh}\)d + {n-4 \over 2} Q^2_{\hh},\\
Q^2_{\hh} &= a_{3n}\Delta_{\hh}R_{\hh} + a_{4n}R_{\hh}^2 + a_{5n}|\text{Ric}_{\hh}|^2
\end{align*}
where $a_{1n}, \cdots, a_{5n} \in \mr$ are constants depending only on $n$.
(Here $R_{\hh}$ and $\text{Ric}_{\hh}$ are the scalar curvature and the Ricci curvature tensor of the manifold $(M,\hh)$, respectively.)
Therefore, by recalling the Yamabe problem and the $Q$-curvature problem,
it is natural to ask whether there is a metric $h_0 \in [\hh]$ such that the corresponding curvature $Q^{\gamma}_{h_0}$ is a constant.
This problem is referred to as the fractional Yamabe problem and explored by Gonz\'{a}lez-Qing \cite{GQ} (non-umbilic cases) and Gonz\'{a}lez-Wang \cite{GW} (umbilic and non-locally conformally flat cases) in the case of $\gamma \in (0,1)$.
Owing to \eqref{eq-conf-cov}, it is equivalent to find a solution of
\begin{equation}\label{eq-main}
P^{\gamma}_{\hh} (u) = cu^{n+2\gamma \over n-2\gamma}, \quad u > 0 \quad \text{on } M
\end{equation}
for some constant $c \in \mr$.

\medskip
The classical Yamabe problem $(\gamma = 1)$ was completely solved by a series of works, starting from Yamabe \cite{Ya}.
Trudinger \cite{Tr} proved existence of a (least energy) solution for the Yamabe problem under the additional assumption that the metric $\hh$ has non-positive scalar curvature.
Aubin \cite{Au} obtained a solution assuming that $n \geq 6$ and that $(M,\hh)$
is not locally conformally flat.
Schoen \cite{Sc0} completed the remaining cases, using the positive mass theorem. See also Lee-Parker \cite{LP} and Bahri \cite{abbas}.
On the other hand, the variational theory for high Morse index solutions was also actively investigated
(see e.g. \cite{Sc3} for the examples such as $S^1 \times S^{n-1}$, and \cite{Po} for general manifolds with $n \ge 3$ and positive scalar curvature).
In this point of view, it is natural to take into account the full set of the solutions.

Schoen \cite{Sc1} raised the conjecture that the solution set for the classical Yamabe problem is compact in the $C^2$-topology,
unless the underlying manifold is conformally equivalent to $S^n$ with the canonical metric.
The case of the round sphere $S^n$ is exceptional since (\ref{eq-main}) is invariant under the action of the conformal group on $S^n$, which is not compact.
Then numerous progress on this direction was achieved by several researchers.
Schoen himself proved compactness of the solution set in the locally conformally flat case \cite{Sc1,Sc2}.
Li and Zhu proved it in dimension $3$ \cite{LZ}, Druet in dimensions $4$ and $5$ \cite{Dr}, see also \cite{LZh,LZh2}.
In dimension $n \geq 6$, the analysis is much more subtle and it is related to the so called Weyl Vanishing conjecture which asserts that the Weyl tensor should vanish at an order greater than $[{n-6 \over 2} ]$ at a blow-up point.
Li and Zhang in \cite{LZh,LZh2} proved the Weyl Vanishing conjecture up to dimension $11$, which in combination with the positive mass theorem allow them to show compactness of the solution set for Yamabe problem up to dimension $11$.
See also Marques \cite{Ma} which treated the dimension up to 7.
The recent work of Khuri, Marques and Schoen \cite{KMS} verified the Weyl Vanishing conjecture up to dimension $24$ and  revealed that the compactness of the solution set for the classical Yamabe problem holds when the dimension of the manifold is strictly less than 25.
Somewhat surprisingly,
the compactness conjecture is not valid in dimension $n\geq 25$: indeed, in this case it is possible to construct a Riemannian manifold $(M, [\hh ])$ such that the set of constant scalar curvature metrics in the conformal class of $\hh$ fails to be compact.  This is shown by Brendle \cite{Br}, for $n \geq 52$,  and Brendle-Marques \cite{BM}, for $n\geq 25$.
We also refer to \cite{AmM,BeM} for construction of non-smooth background metrics.

\medskip
In 1992, Escobar \cite{Es1, Es2} formulated an analogue of the Yamabe problem for manifolds with boundary, which is now called the boundary Yamabe problem.
This corresponds to the fractional Yamabe problem with $\gamma = 1/2$ as Gonz\'alez and Qing observed in  \cite{GQ}. The solvability issue was solved in most of the cases:
in \cite{Es1} solvability is proved in dimension  $2$, in dimension is $3$ or $4$ under the assumption that boundary  is umbilic, in dimension $n\geq 5$ if the manifold is locally conformally flat and the boundary is umbilic. We refer the reader for developments on this issue to \cite{Ma2, Ma3, Al0, Es2, BC} and reference therein.
The problem of compactness of the solution set for the ${1\over 2}$-fractional Yamabe problem is studied in the conformally flat case with umbilic boundary in \cite{FO},
and in the case of dimension $2$ in \cite{FO1}.
Related results on compactness were obtained by Almaraz in \cite{Al2} and by Han-Li \cite{HL}.
Notably, compactness is lost for high dimensions, but this time for dimensions $n \geq 24$.
Indeed, there are examples of metrics on the unit ball $B^{n+1}$, with $n\geq 24$, for which the set of scalar-flat metrics on $B^{n+1}$ in the same conformal class with respect to which $\partial B^{n+1}$ has  constant mean curvature, is not compact.
This construction is done in \cite{Al}. Just a remark: In the boundary Yamabe problem studied by Almaraz \cite{Al}, the author denoted by $n$ the dimension of the upper-half space.
Since in this paper we assume $n$ to be the dimension of its boundary, the critical dimension in our main theorem for $\gamma = 1/2$ reads to be 24 instead of 25 as in \cite{Al}.
Thus, when $\gamma={1\over 2}$, compactness of the set of solutions to the fractional Yamabe problem is lost  at least from $n\geq 24$.
See also Disconzi-Khuri \cite{DK}.

\medskip
Interestingly enough, also for the $\gamma=2$ case, it is again from dimension $n=25$ that compactness for the set of solutions to the $2$-fractional Yamabe Problem (namely, the $Q$-curvature problem) is lost:
in \cite{WZ}, Wei and Zhao showed the existence of a non-compact set of metrics on the sphere $S^n$  for which the curvature $Q^2_{\hh}$ is constant,
or equivalently the solution set for Problem  (\ref{eq-main}), with $\gamma=2$, is non compact.
Concerning compactness of solutions to the $Q$-curvature problem, as far as we know, the only available results are contained in \cite{HR, QR, Li, LX}, see also \cite{HH}.

\medskip
Given these results, one can expect that the starting dimension for non compactness of the $\gamma$-fractional Yamabe Problem depends on $\gamma$.

\medskip
In this paper, we explore precisely this problem.
We are interested in non-compactness property for the fractional Yamabe problem provided that $\gamma \in (0,1)$ and the background dimension is sufficiently high.
We show that there is a transition of the critical dimension at some $\gamma \in (0,1)$,
which takes into account that the smaller $\gamma$ tends to be, the stronger the nonlocal effect becomes.
Our result in particular bridges the classical Yamabe problem and the boundary Yamabe problem.

\medskip
Our result is the following

\begin{thm}\label{thm-main}
There exists a number $\gamma^* \simeq 0.940197$ such that the following properties hold:

\medskip \noindent 1. There are a $C^{\infty}$ Riemannian metric $g^+$ and a boundary defining function $\rho$ on $\mr^{n+1}_+$ such that
$(\mr^{n+1}_+, g^+)$ is an asymptotically hyperbolic manifold with the conformal infinity $(\mr^n, [\hh])$ where $\hh = \rho^2g^+|_{\mr^n}$.
They can be taken to be independent of the choice of $\gamma$.

\medskip \noindent 2. Fix any $\gamma \in (0,1)$ and suppose that $n \ge 24$ if $\gamma \in (0, \gamma^*)$ and $n \ge 25$ if $\gamma \in [\gamma^*, 1)$.
Then one has a sequence of positive solutions $\{u^{\gamma}_m\}_{m \in \mn}$ to the fractional Yamabe equation \eqref{eq-main} with the Yamabe constant $c = 1$,
satisfying $\|u^{\gamma}_m\|_{L^{\infty}(\mr^n)} \to \infty$ as $m \to \infty$.
\end{thm}

Recently numerous results on nonlocal conformal operators have been established.
This includes \cite{QR} for the higher-order fractional Yamabe problem, \cite{GMS} for the fractional singular Yamabe problem,
\cite{JX} for the fractional Yamabe flow and \cite{ACH, CLZ, JLX, JLX2, JLX3} for the fractional Nirenberg problem.
Furthermore, Druet \cite{Dr2}, Druet-Hebey \cite{DH}, Micheletti-Pistoia-V{\'e}tois \cite{MPV}, Esposito-Pistoia-V{\'e}tois \cite{EPV} (for $\gamma = 1$),
Deng-Pistoia \cite{DP}, Pistoia-Vaira \cite{PV} (for $\gamma = 2$)
and Choi-Kim \cite{CK} (for $\gamma \in (0,1)$) dealt with compactness issue of lower order perturbations of Eq. \eqref{eq-main}.

\medskip \noindent \textbf{Structure of this paper.}
In the next section, we will describe the setting of our problem.
Whilst our program is adopted from \cite{Br}, \cite{BM} and \cite{Al},
we need to recall two more ingredients to handle the nonlocal conformal operators
- the singular Yamabe problem (refer to \cite{ACF} and \cite{AM}) and the Caffarelli-Silvestre type extension result (\cite{CS}) for the fractional conformal Laplacian obtained in \cite{CG}.
To be more precise, we first define a Riemannian metric $\bg$ on the closure of the half space $\omrn$, slightly perturbing the canonical metric $g_c$.
Then we select a suitable boundary defining function $\rho$
by imposing the scalar curvature of $(\mr_+^N, g^+)$ where $g^+ = \rho^{-2}\bg$ to be $-n(n+1)$ and solving the associated singular Yamabe problem (see Appendix \ref{subsec-app-a}).
Because the precise information of $\rho$ near the origin will be required, we will also achieve it in Appendix \ref{subsec-app-b}.
Now $(\mr_+^N, g^+)$ becomes an asymptotically hyperbolic manifold, and the fractional conformal Laplacian is well defined.
Instead of treating it directly, we consider its localization due to Chang-Gonz\'alez \cite{CG}.

In Section \ref{sec-red}, the finite dimensional Lyapunov-Schmidt reduction method is applied to show that our desired solution will be attained
once we find a critical point of a certain functional $J_0^{\gamma}$ (in \eqref{eq-loc-energy}).
At this point, it is necessary to understand the global behavior of $\rho$ and the spectral property of $-\Delta_{g+}$ to establish the linear theory and to ensure the positivity of solutions.
This will be touched in Appendix \ref{subsec-app-c}.

An important property is that $J_0^{\gamma}$ can be approximated at main order by a polynomial $P$ (in \eqref{eq-poly-3}) as it will be shown in Section \ref{sec-exp}.
To do so, we have to calculate a number of integrals regarding the bubbles $W_{\ls}$ in \eqref{eq-Poisson}.
In the local case ($\gamma = 1/2, 1$ or $2$), the formulae of the bubbles are explicit, so it is relatively plain to obtain the value of the integrals (refer to \cite[Proposition 27]{Br}).
However, in the non-local case, only the representation formula is available for the bubbles.
In order to get over this difficulty, we further develop the approach of Gonz\'alez-Qing \cite{GQ} where they utilized the Fourier transform.
Finally, Section \ref{sec-conc} is devoted to search a critical point of $P$, thereby proving our main theorem.

\bigskip \noindent \textbf{Notation.}

\medskip \noindent - Throughout the paper, we use the Einstein convention. The indices $a, b, c$ and $d$ always run from 1 to $n+1$, while $i, j, k, \tilde{k}, l, p, q, s$ and $\tilde{s}$ run from 1 to $n$.

\medskip \noindent - We denote $N = n+1$. Also, for $x = (x_1, \cdots, x_N) \in \mr^N_+ = \{(x_1, \cdots, x_N) \in \mr^N: x_N > 0\}$,
we use $\bx = (x_1, \cdots, x_n, 0) \in \partial \mr^N_+ \simeq \mr^n$ and $r = |\bx| \ge 0$.

\medskip \noindent - For any $\vr > 0$, we write $B_+^N(0,\vr)$ to denote the upper-half open ball in $\mr^N_+$ centered at the origin whose radius is $\vr$.
Also, $B^n(0,\vr)$ and $S^{n-1}(0,\vr)$ are the $n$-dimensional ball and the $(n-1)$-dimensional sphere, respectively, whose centers are located at 0 and radii are $\vr$.
We use $S^{n-1} = S^{n-1}(0,1)$ for the sake of brevity.
Furthermore, $dS_\vr$ is the surface measure of the sphere $S^{n-1}(0,\vr)$ in $\mr^n$ and $dS = dS_1$.

\medskip \noindent - For a Riemannian manifold $(X,g)$, $\Delta_g$ stands for the Laplace-Beltrami operator (of negative spectrum). If $(X,g)$ is the standard Euclidean space, we denote $\Delta = \Delta_g$.

\medskip \noindent - $\chi_A$ is the characteristic function of a set $A$.

\medskip \noindent - $t_+ = \max\{t, 0\}$ and $t_- = \max\{-t, 0\}$ for any $t \in \mr$.

\medskip \noindent - For fixed $n \in \mn$ and $\gamma \in (0,1)$ such that $n > 2\gamma$,
the space $\drn$ is defined as the completion of the space $C^{\infty}_c(\omrn)$ with respect to the norm
\begin{equation}\label{eq-d12}
\|U\|_{\drn} := \(\int_{\mr_+^N} x_N^{1-2\gamma}|\nabla U|^2 dx\)^{1/2} \quad \text{for } U \in C^{\infty}_c(\omrn)
\end{equation}
(refer to Remark \ref{rmk-sobolev}).
Let also $\md^{1,2}(\vr)$ be the completion of $C^{\infty}_c\(B^N_+(0,\vr) \cup B^n(0,\vr)\)$ with respect to the norm \eqref{eq-d12}.

\medskip \noindent - For a function $f \in L^2(\mr^n)$, the Fourier transform $\hat{f}$ of $f$ is defined by
\[\hat{f}(\xi) = {1 \over (2\pi)^{n/2}} \int_{\mr^n} f(x)e^{-ix \cdot \xi}dx \quad \text{for } \xi \in \mr^n.\]
We also use $\rho = |\xi|$.

\medskip \noindent - The letters $C$ and $\widetilde{C}$ (without subscripts) denote positive numbers that may vary from line to line.

\section{Setting of the problem}\label{sec-setting}
The following setting is due to Brendle \cite{Br} and Almaraz \cite{Al}.
Fix $W: (\mr^n)^4 \to \mr$ be a multi-linear form such that its tensor norm
\[|W| = \(\sum_{i,j,k,l = 1}^n (W_{ikjl} + W_{iljk})^2\)^{1/2}\]
is positive everywhere
and it satisfies all algebraic properties the Weyl tensor has:
$W_{ijkl} = - W_{jikl} = - W_{ijlk} = W_{klij}$ (symmetry and anti-symmetry), $W_{ijkl} + W_{iklj} + W_{iljk}= 0$ (the Bianchi identity)
and any contraction of $W$ gives 0 (which is equivalent to $\sum_{i=1}^nW_{ijik} = 0$ by the symmetric property).
Then we set a tensor
\begin{equation}\label{eq-h0}
H_{ij}(x) = H_{ij}(\bx) = W_{ikjl}x^kx^l \quad \text{and} \quad H_{aN}(x) = H_{Nb}(x) = 0
\end{equation}
for any $x \in \mr^N_+$, and using this we also define a trace-free symmetric two-tensor $h$ in $\mr^N_+$ which satisfies
\begin{equation}\label{eq-h}
h_{ab}(x) = \begin{cases}
\mu \ep^{2d_0}f\(\ep^{-2}|\bx|^2\)H_{ab}(\bx) &\text{for } |x| \le \nu,\\
0 &\text{for } |x| \ge 1.
\end{cases}
\end{equation}
Here $0 < \ep \ll \nu \le 1$ (e.g., $\nu |\log \ep| \ge 1/100$ would suffice), $\mu = \ep^{1/3}$ and $f(t) = \sum_{m=0}^{d_0} a_m t^m$ is a polynomial of degree $d_0$ ($1 \le d_0 \le 4$ and $a_m \in \mr$).
Moreover we impose further conditions on the tensor $h$ that
\begin{equation}\label{eq-eta_0}
h_{aN}(x) = 0 \quad \text{and} \quad \sum_{m=0}^{2(2d_0+2)} \left|D^mh_{ab}(x)\right| \le \eta_0 \quad \text{for all } x \in \mr_+^N
\end{equation}
where $\eta_0 \gg \ep > 0$ is a small number to be determined in Section \ref{sec-red},
and that it relies only on the first $n$ variables (so that $\pa_Nh_{ab} = 0$ where $\pa_N = \pa_{x_N}$) if $0 \le x_N \le \nu$.
By virtue of our construction, it immediately follows that
\begin{equation}\label{eq-h-prop}
x^ah_{ab}(x) = \sum_{a=1}^N\pa_ah_{ab}(x) = 0 \quad \text{for any } |x| \le \nu.
\end{equation}
Now if we define $\bg = \exp{(h)}$, then $\(\omrn, \bg\)$ is a smooth Riemannian manifold with a boundary.
Moreover, it is easy to check that the submanifold $(\mr^n, \hh)$ where $\hh = \bg|_{T\mr^n}$ is totally geodesic.
This is equivalent to say that the second fundamental form $\pi_{ij}$ satisfies $\pi_{ij} = \pa_N\bg_{ij}/2 = 0$.
This fact implies in particular that the mean curvature $H = \bg^{ij}\pi_{ij}/n$ also vanishes on $\mr^n$.

Furthermore, since the trace of the tensor $h$ is zero, we have the following expansion of the scalar curvature of the manifold $(\mr_+^N, \bg)$: For some $C = C(n) > 0$,
\begin{multline}\label{eq-R-bg-est-0}
\left|R_{\bg} - \(\sum_{i,j=1}^n \pa_{ij} h_{ij} - \sum_{i,j,k=1}^n \pa_i\(h_{ij}\pa_kh_{kj}\)
+ {1 \over 2} \sum_{i,j,k=1}^n \pa_ih_{ij}\pa_kh_{kj} - {1 \over 4} \sum_{i,j,k=1}^n \(\pa_kh_{ij}\)^2\) \right| \\
\le C\(|h|^2\left|D^2h\right|+|h||Dh|^2\) \quad \text{in } C\(\omrn\).
\end{multline}
See \cite[Proposition 26]{Br} for the detailed explanation.
In particular, a further inspection with \eqref{eq-h-prop} shows that
\begin{equation}\label{eq-R-bg-est}
R_{\bg} = -{1 \over 4}\sum_{i,j,k=1}^n\(\pa_kh_{ij}\)^2 + O\(|h|^2\left|D^2h\right|+|h||Dh|^2\) \quad \text{in } C^{\infty}(\{|x| \le \nu\}).
\end{equation}

\medskip
In order to make the space $\mr_+^N$ to be asymptotically hyperbolic with conformal infinity $(\mr^n, [\hh])$,
we solve the singular Yamabe problem.
Precisely, we construct a metric $g^+ \in [\bg]$ in $\mr^N_+$
such that its scalar curvature $R_{g^+}$ is equal to $-n(n+1)$ and $\rho^2 g^+|_{T\mr^n} = \hh$ for some boundary defining function $\rho$ of $\mr^n = \partial \mr^N_+$.
By the results of Aviles-McOwen \cite{AM} and Andersson-Chru\'{s}ciel-Friedrich \cite{ACF},
it is known that this problem is solvable for $N \ge 3$ and the defining function $\rho$ has the form
\begin{equation}\label{eq-bd-dfn}
\rho = x_N\(1 + Ax_N + Bx_N^N\)^{-{2 \over N-2}}
\end{equation}
near the boundary $\mr^n$, where $A \in C^{\infty}(\omrn)$, $B \in C^{\infty}(\mr^N_+)$ and $B$ has a polyhomogeneous expansion in the $x_N$-variable near the boundary.

To obtain the existence of the metric $g^+$, one can take the following procedure:
Let us assume that $g^+ = w^{4 \over N-2}\bg$ for some positive function $w$ in $\mr^N_+$ such that $wx_N^{N-2 \over 2} \to 1$ as $x_N \to 0+$.
If we put $u = wx_N^{N-2 \over 2}$ and $\tg = x_N^{-2}\bg$, then the problem boils down to the Loewner-Nirenberg problem \cite{LN}
\begin{equation}\label{eq-LN}
-{4(N-1) \over N-2}\Delta_{\tg} u + R_{\tg} u + N(N-1)u^{N+2 \over N-2} = 0 \quad \text{in } \mr_+^N \quad \text{and} \quad u = 1 \quad \text{on } \mr^n.
\end{equation}
By employing a stereographic projection, we may assume that the domain of the equation is $B^N$ instead of $\mr_+^N$.
Then it turns out that this equation admits positive upper and lower solutions, which gives the unique positive solution $u$ continuous up to the boundary $S^n$ (or $\mr^n$ after transforming back - see Appendix \ref{subsec-app-a} for further discussion on the conformal change).
This also guarantees the existence of the defining function $\rho = u^{-{2 \over N-2}}x_N$.

Very recently, Han and Jiang \cite{HJ} established optimal asymptotic expansions of solutions to the Dirichlet problem for minimal graphs in the hyperbolic space.
As it will be discussed in Appendix \ref{subsec-app-b}, their approach also alludes that
the formal expansion of the solution $u$ to Eq. \eqref{eq-LN} in the $x_N$-variable is accurate up to $O(x_N^N \log x_N)$ order.
Because the coefficient of the $x_N$-order in the expansion of $u$ is a constant multiple of the mean curvature $H$ of $(\mr^n, \hh) \subset \(\omrn, \bg\)$ and it holds that $H = 0$ due to our construction of $\bg$,
it is expected that the asymptotic expansion of $\rho$ contains only even powers of $x_N$.
Indeed, we have the following description on $\rho$ up to the $4(d_0+1)$-th order of $x_N$.
\begin{prop}\label{prop-rho-est}
Assume that $N \ge 22$ (and $n \ge 21$) and let $x = (\bx, x_N) \in \mr^N_+$.
\begin{enumerate}
\item It holds that $C^{-1}x_N < \rho(\bx, x_N) < Cx_N$ in $\mr^N_+$ for some $C > 0$ independent of the points $x \in \mr^N_+$.
\item Denote $\oh_{ab}(x) = f(|\bx|^2)H_{ab}(x)$ and fix numbers $\nu, \eta > 0$ sufficiently small.
Then we have
\begin{equation}\label{eq-rho-est}
\rho(\ep x) = \left[1 + \mu^2\ep^{4(d_0+1)} \sum_{m=1}^{2d_0+2} C_{2m}(\bx) x_N^{2m} + O\(\mu^3 \ep^{4(d_0+1)}|x|^2\(1+|x|^{4d_0}\)x_N^2\) + O\((\ep x_N)^{4(d_0+1)+2-\eta}\) \right] \ep x_N
\end{equation}
in $C^2(B_+^N(0,\nu/\ep))$ where the function $C_{2m}$ is defined as
\begin{equation}\label{eq-C}
C_{2m}(\bx) = -{1 \over 24(N-1)(N-2)} \left[\prod_{\tm=1}^{m-1} {1 \over (2\tm+3)(N-2(\tm+1))}\right] \sum_{i,j,k=1}^n \Delta^{m-1} \(\pa_k \oh_{ij}(\bx)\)^2
\end{equation}
for all $m = 1, \cdots, 2(d_0+2)$.
The value in the bracket is understood as 1 if $m = 1$.
\end{enumerate} \end{prop}
\begin{proof}
Since $u$ is a bounded function in $\mr^N_+$ away from 0 and $\rho = u^{-{2 \over N-2}}x_N$, the first assertion is true.
The proof of \eqref{eq-rho-est} is postponed to Appendix \ref{subsec-app-b}.
\end{proof}

Our proof for Theorem \ref{thm-main} strongly relies not only on the results on the singular Yamabe problem,
but also on the following local interpretation of the conformal fractional Laplacian found by Chang and Gonz\'{a}lez.
\begin{prop}\label{prop-CG}
(\cite[Theorem 5.1]{CG}, see also \cite[Proposition 2.1]{GQ})
Let $\gamma \in (0,1)$, $\bg$ be a Riemannian metric on $\omrn$ and $\hh$ its induced metric on the boundary $\mr^n$.
Also we suppose that the mean curvature on $\mr^n$ is 0 and
the last component $x_N$ of $\mr_+^N$ serves as the boundary defining function, namely, $\bg = \bg_{x_N} + dx_N^2$
for some one parameter family of metrics $\bg_{x_N}$ on $\mr^n$.
Then one can construct an asymptotically hyperbolic metric $g^+$ in $\mr_+^N$ conformal to $\bg$
such that $R_{g^+} = - n(n+1)$ and a defining function $\rho$ satisfying $\rho^2g^+|_{T\mr^n} = \hh$ as well as \eqref{eq-bd-dfn} and \eqref{eq-rho-est}.
(This is what we explained in the previous paragraphs.)
Moreover if $U$ is a solution of the following extension problem
\[\begin{cases}
-\textnormal{div}_{\bg}\(\rho^{1-2\gamma}\nabla U\) + E(\rho)U = 0 &\text{in } \(\mr^N_+, \bg\),\\
U = f &\text{on } \mr^n
\end{cases}\]
for a given function $f$ in the Sobolev space $H^{\gamma}(\mr^n)$, where $E$ is the error term given by
\begin{equation}\label{eq-error}
E(\rho) = -\Delta_{\bg}\(\rho^{1-2\gamma \over 2}\)\rho^{1-2\gamma \over 2}+ \(\gamma^2-{1 \over 4}\) \rho^{-1-2\gamma} + {n-1 \over 4n}R_{\bg}\rho^{1-2\gamma},
\end{equation}
then
\begin{equation}\label{eq-wnormal}
P^{\gamma}[g^+,\hh] f = - \kappa_{\gamma} \(\lim_{\rho \to 0+} \rho^{1-2\gamma}{\pa U \over \pa \rho}\) = - \kappa_{\gamma} \(\lim_{x_N \to 0+} x_N^{1-2\gamma}{\pa U \over \pa x_N}\) := \pa^{\gamma}_{\nu} U.
\end{equation}
Here $\kappa_{\gamma} = 2^{2\gamma-1}\Gamma(\gamma)/\Gamma(1-\gamma)$ and $\nu$ designates the unit outer normal vector $-\pa_{x_N}$ to the boundary $\mr^n$.
\end{prop}
\noindent Therefore, in order to solve the nonlocal Eq. \eqref{eq-main} with $c = 1$,
it suffices to find a positive solution of the degenerate local problem
\begin{equation}\label{eq-main-ext}
\begin{cases}
-\text{div}_{\bg}\(\rho^{1-2\gamma}\nabla U\) + E(\rho)U = 0 &\text{in } \(\mr^N_+, \bg\),\\
U = u &\text{on } \mr^n,\\
\pa^{\gamma}_{\nu} U = U^{n+2\gamma \over n-2\gamma} &\text{on } \mr^n.
\end{cases} \end{equation}
Besides, by \eqref{eq-wnormal}, 
a critical point of the energy functional
\begin{equation}\label{eq-energy}
I^{\gamma}(U) = {\kappa_{\gamma} \over 2} \int_{\mr^N_+}\(\rho^{1-2\gamma}|\nabla U|_{\bg}^2 + E(\rho)U^2\) dv_{\bg}
- {n-2\gamma \over 2n}\int_{\mr^n} U_+^{2n \over n-2\gamma} dv_{\hh} \quad \text{for } U \in \mh_1
\end{equation}
solves \eqref{eq-main-ext}.
Here $dv_{\bg}$ and $dv_{\hh}$ represent the volume forms of $\(\omrn, \bg\)$ and $(\mr^n, \hh)$ respectively.
Because $\det \bg = \det \hh = 1$ due to our construction, it holds that $dv_{\bg} = dx$ and $dv_{\hh} = d\bx$.
Eq. \eqref{eq-lin-04} and the Sobolev trace inequality $\drn \hookrightarrow L^{2n \over n-2\gamma}(\mr^n)$
ensure that $I^{\gamma}$ is well-defined in the space $\mh_1$ defined in \eqref{eq-mh-1}.

In the special case $\bg = dx^2$, $g^+ = dx^2/x_N^2$ and $\rho = x_N$,
the fractional Paneitz operator $P^{\gamma}_{d\bx^2}$ reduces to the usual fractional Laplacian $(-\Delta)^\gamma$
and the corresponding result to Proposition \ref{prop-CG} was established by Caffarelli and Silvestre \cite{CaS}.
As it is now well-understood through a series of works conducted by many mathematicians (see for instance \cite{CS, CS2, CKL, DDW, DLS, GQ, SV, T} and references therein),
this observation allows one to apply well-known techniques such as the mountain pass theorem, blow-up analysis, the finite dimensional reduction method, the moving plane method, the Moser iteration method and so on,
for local, but degenerate, equations to analyze the corresponding nonlocal equations.
On the other hand, the results of Yang \cite{Yan} and Case-Chang \cite{CC}, which present the extension results for the higher order fractional conformal Laplacians, would allow one to apply similar approaches for the case $\gamma \in (1, n/2)$.

\medskip
Before finishing this section, we recall the bubbles $w_{\ls}$ and their $\gamma$-harmonic extensions $W_{\ls}$.
Given $\lambda > 0$ and $\sigma \in \mr^n$, the function $w_{\ls}$ is defined as
\[w_{\ls}(\bx) = c_{n,\gamma} \(\lambda \over \lambda^2 + |\bx-\sigma|^2\)^{n-2\gamma \over 2} = {1 \over \lambda^{n-2\gamma \over 2}} w_{1,0}\({\bx-\sigma \over \lambda}\) \quad \text{for any } \bx \in \mr^n\]
for some normalizing constant $c_{n,\gamma}$ whose value is presented below,
and $W_{\ls}(x) = \lambda^{-{n-2\gamma \over 2}}W_{1,0}(\lambda^{-1}(\bx-\sigma), \lambda^{-1}x_N)$ is a unique solution of the degenerate elliptic equation
\begin{equation}\label{eq-har-ext}
\begin{cases}
\text{div}\(x_N^{1-2\gamma}\nabla U\) = 0 &\text{in } \mr_+^N,\\
U(\cdot,0) = w_{\ls} &\text{on } \mr^n,\\
U \in \drn.
\end{cases} \end{equation}
Each bubble $w_{\ls}$ solves the equation
\begin{equation}\label{eq-rn}
(-\Delta)^{\gamma} u = u^{n+2\gamma \over n-2\gamma} \quad \text{in } \mr^n.
\end{equation}
Hence $\partial_{\nu}^{\gamma} W_{\ls} = w_{\ls}^{n+2\gamma \over n-2\gamma}$ in $\mr^n$ by Proposition \ref{prop-CG}.
Besides, it is possible to describe $W_{\ls}$ in terms of the Poisson kernel $K_{\gamma}$:
\begin{equation}\label{eq-Poisson}
W_{\ls}(\bx,x_N) = (K_{\gamma}(\cdot, x_N) \ast w_{\ls})(\bx) = p_{n,\gamma} \int_{\mr^n} {x_N^{2\gamma} \over (|\bx-\xi|^2+|x_N|^2)^{n+2\gamma \over 2}} w_{\ls}(\xi) d\xi.
\end{equation}
The values of constants $c_{n,\gamma}$ and $p_{n,\gamma}$ are
\begin{equation}\label{eq-cng}
c_{n,\gamma} = 2^{n-2\gamma \over 2}\({\Gamma\({n+2\gamma \over 2}\) \over \Gamma\({n-2\gamma \over 2}\)}\)^{n-2\gamma \over 4\gamma}
\quad \text{and} \quad p_{n,\gamma} = {\Gamma\({n+2\gamma \over 2}\) \over \pi^{n \over 2}\Gamma(\gamma)}.
\end{equation}

The nondegeneracy result of \cite{DDS} tells us that the set of bounded solutions for the linearized problem to \eqref{eq-rn}
\[(-\Delta)^{\gamma} u = \({n+2\gamma \over n-2\gamma}\) w_{\ls}^{4\gamma \over n-2\gamma}u \quad \text{in } \mr^n\]
is spanned by
\[z_{\ls}^1 := {\pa w_{\ls} \over \pa \sigma_1} = -{\pa w_{\ls} \over \pa x_1}
, \cdots, z_{\ls}^n := {\pa w_{\ls} \over \pa \sigma_n} = -{\pa w_{\ls} \over \pa x_n}\]
and
\[z_{\ls}^0 := {\pa w_{\ls} \over \pa \lambda} = - \({n-2\gamma \over 2}\) w_{\ls} - x \cdot \nabla w_{\ls}\]
where $\lambda > 0$ and $\sigma = (\sigma_1, \cdots, \sigma_n)$.
Also, if we let $Z_{\ls}^m$ be the $\gamma$-harmonic extension of $z_{\ls}^m$ for $m = 0, \cdots, n$, i.e., the solution of \eqref{eq-har-ext} whose second equality is replaced by $U(\cdot,0) = z_{\ls}^m$ on $\mr^n$,
then the following decay properties can be checked.

\begin{lemma}\label{lemma-W-decay}
There exists a constant $C = C(n,\gamma) > 0$ such that
\begin{equation}\label{eq-W-decay-1}
\left|\nabla^m_{\bx} W_{\ls}(x)\right| \le {C \lambda^{n-2\gamma \over 2} \over \lambda^{n-2\gamma+m} +|x-(\sigma,0)|^{n-2\gamma+m}} \quad (m = 0, 1, 2)
\end{equation}
and
\[\left|\nabla^m_{\bx} \pa_N W_{\ls}(x)\right| \le C \lambda^{n-2\gamma \over 2} \({x_N^{2\gamma-1} \over \lambda^{n+m} + |x-(\sigma,0)|^{n+m}}
+ {1 \over \lambda^{n-2\gamma+1+m} + |x-(\sigma,0)|^{n-2\gamma+1+m}}\) \quad (m = 0, 1)\]
for all $(\ls) \in (0,\infty) \times \mr^n$ and $x \in \mr^N_+$.
Here $\nabla^m_{\bx}$ means the $m$-th derivative with respect to the $\bx$-variable.
\end{lemma}
\begin{proof}
In \cite[Lemma A.2]{CK} the authors proved the assertion under the assumption that $(\ls) = (1,0)$ by treating Green's representation formula \eqref{eq-Green} (cf. Lemma \ref{lemma-lin-3} below).
The estimate for general $(\ls)$ is achieved by rescaling of the variables.
\end{proof}

\begin{lemma}\label{lemma-W-decay-2}
Fix any $\gamma \in (0,1)$, $\alpha \ge 1$ and $\beta \ge 0$. Given $t_0 > 0$ fixed, we have that
\begin{align*}
\int_{B^N_+(0,t_0)^c} x_N^{\alpha-2\gamma} |\bx|^{\beta} W_{1,0}^2(\bx,x_N) d\bx dx_N &\le C t_0^{-n+1+2\gamma+\alpha+\beta}
\intertext{and}
\int_{B^N_+(0,t_0)^c} x_N^{\alpha-2\gamma} |\bx|^{\beta} |\nabla W_{1,0}(\bx,x_N)|^2 d\bx dx_N &\le C t_0^{-n-1+2\gamma+\alpha+\beta}
\end{align*}
where $C = C(n, \gamma, \alpha, \beta)$ is a positive constant relying only on $n$, $\gamma$, $\alpha$ and $\beta$.
\end{lemma}
\begin{proof}
Integrating in the polar coordinate and taking advantage of \eqref{eq-W-decay-1}, we have
\begin{align*}
&\ \int_{B^N_+(0,t_0)^c} x_N^{\alpha-2\gamma} |\bx|^{\beta} W_{1,0}^2(\bx,x_N) d\bx dx_N \\
&\le C \int_{\left\{(r,x_N) \in \mr^2:\ r^2+x_N^2 \ge t_0^2,\ r, x_N > 0\right\}} {x_N^{\alpha-2\gamma} r^{n-1+\beta} \over 1+(r^2+x_N^2)^{n-2\gamma}} drdx_N \quad (r = |\bx|)\\
&\le \int_0^{\pi \over 2} \int_{t_0}^{\infty} {t^{\alpha-2\gamma} (\sin\theta)^{\alpha-2\gamma} t^{n-1+\beta}(\cos\theta)^{n-1+\beta} \over t^{2(n-2\gamma)}} tdtd\theta \quad (r = t\cos\theta,\ x_N = t\sin\theta)\\
&= C \int_{t_0}^{\infty} t^{-n+2\gamma+\alpha+\beta} dt = C t_0^{-n+1+2\gamma+\alpha+\beta}.
\end{align*}
In the above formula, that $\int_0^{\pi \over 2} (\sin\theta)^{\alpha-2\gamma} (\cos\theta)^{n-1+\beta} d\theta < \infty$ is guaranteed by the assumption that $\alpha-2\gamma > -1$ and $\beta \ge 0$.

The other equation can be derived in similar reasoning.
Therefore the conclusion of Lemma \ref{lemma-W-decay-2} follows.
\end{proof}
\noindent For the remaining part of the paper, we write $W_\dz = W_{\delta}$ and $w_\dz = w_{\delta}$ for simplicity.

\section{Reduction process} \label{sec-red}
Recall the parameter $\ep \in (0,1)$ in the definition of the tensor $h$ (refer to \eqref{eq-h}).
From now on, for each sufficiently small fixed $\ep > 0$, we look for a positive solution to \eqref{eq-main-ext} of the form $W_{\edte} + \Psi_{\edte}$
where $W_{\edte}$ is the $\gamma$-harmonic extension of the bubble $w_{\edte}$
and $\Psi_{\edte}$ is a remainder term which is small in a suitable sense,
by choosing the constant $\delta(\ep) > 0$ and the point $\tau(\ep) \in \mr^n$ appropriately.

Let us consider the admissible set $\ma := \(1-\varepsilon_0, 1+\varepsilon_0\) \times B^n(0,\varepsilon_0)$ where $\varepsilon_0 \in (0,1)$ is some small number.
In this section, given any $(\dt) \in \ma$, we shall choose a function $\Psi_{\edt}$ for each $W_{\edt}$ so that $W_{\edt} + \Psi_{\edt}$ solves an auxiliary equation to \eqref{eq-main-ext}.
The selection of the special pairs $(\edte)$, which gives a desired solution of \eqref{eq-main-ext} for each $\ep > 0$, will be performed in the subsequent sections.
Throughout this section, it is assumed that $(\ls) = (\edt)$.

\subsection{Weighted Sobolev inequality and regularity results for degenerate elliptic equations}
In this subsection, we derive Sobolev inequalities for the spaces $\drn$ and $\md^{1,2}(\vr)$.
After proving them, we also examine regularity of solutions to degenerate elliptic equations.

\begin{lemma}\label{lemma-sobolev}
The followings hold.
\begin{enumerate}
\item Fix any $\vr_1, \vr_2 > 0$. Then we have
\[\(\int_0^{\vr_1} \int_{B^n(0, \vr_2)} x_N^{1-2\gamma} |U|^2 dx\)^{1/2} \le C \|U\|_{\drn}
\quad \text{for } U \in \drn\]
where $C > 0$ depends only on $N$, $\gamma$, $\vr_1$ and $\vr_2$.
\item Given $\vr > 0$ fixed, there exist positive constants $C$ and $\eta$ depending only on $N$, $\gamma$ and $\vr$ such that for all $1 \le m \le {N \over N-1} + \eta$
\begin{equation}\label{eq-sobolev-2}
\(\int_{B^N_+(0,\vr)} x_N^{1-2\gamma} |U|^{2m} dx\)^{1/2m} \le C \(\int_{B^N_+(0,\vr)} x_N^{1-2\gamma} |\nabla U|^2 dx\)^{1/2}, \quad U \in \md^{1,2}(\vr).
\end{equation}
\end{enumerate}
\end{lemma}
\begin{proof}
1. By density, we may assume that $U \in C_c^{\infty}(\mr^N_+)$.
By the proof of Lemma 3.1 of \cite{CK}, we obtain
\[\int_0^{\vr_1} \int_{B^n(0, \vr_2)} x_N^{1-2\gamma} |U(\bx,x_N)|^2 d\bx dx_N
\le C \(\int_{B^n(0, \vr_2)} |U(\bx,0)|^2 d\bx + \int_0^{\vr_1} \int_{B^n(0, \vr_2)} x_N^{1-2\gamma} |\nabla U(\bx,x_N)|^2 d\bx dx_N\).\]
Thus the Sobolev trace inequality $\drn \hookrightarrow L^{2n \over n-2\gamma}(\mr^n)$ (proved in \cite{Xi}) gives the desired inequality.

\medskip \noindent
2. Since $x_N^{1-2\gamma}$ is an element of the class of Muckenhoupt weights $A_2$, \cite[Lemma 2.2]{TX} (cf. \cite[Theorem 1.2]{FKS}) implies that
\eqref{eq-sobolev-2} holds for arbitrary $U \in C_c^{\infty}\(B^N_+(0,\vr) \cup B^n(0,\vr)\)$.
By the standard density argument, \eqref{eq-sobolev-2} can be extended to all $U \in \md^{1,2}(\vr)$.
\end{proof}
\begin{rmk}\label{rmk-sobolev}
Since $x_N^{1-2\gamma}, x_N^{2\gamma-1} \in L^1_{\text{loc}}(\mr^N_+)$,
Lemma \ref{lemma-sobolev} (1) shows that the gradient $\nabla U$ is well-defined for any $U \in \drn$. See \cite[Subsection 2.1]{FKS}.
\end{rmk}

\medskip Consider the equation
\begin{equation} \label{eq-lin-12}
\begin{cases}
-\textnormal{div}_{\bg}\(\rho^{1-2\gamma}\nabla U\) = x_N^{1-2\gamma} \Phi &\text{in } \(\mr^N_+, \bg\),\\
\pa^{\gamma}_{\nu} U = \zeta &\text{on } \mr^n
\end{cases} \end{equation}
and its related inequalities for given $\Phi \in L^1_{\text{loc}}(\mr^N_+)$ and $\zeta \in L^1_{\text{loc}}(\mr^n)$.

\begin{dfn}\label{dfn-w-sol}
We say that a function $U \in \drn$ is a weak solution of \eqref{eq-lin-12} if
\[\kappa_{\gamma} \int_{\mr^N_+} \rho^{1-2\gamma} \la \nabla U, \nabla V \ra_{\bg} dx
= \kappa_{\gamma} \int_{\mr^N_+} x_N^{1-2\gamma} \Phi V dx + \int_{\mr^n} \zeta v d\bx \quad \text{for any } V \in C^{\infty}_c(\omrn)\]
where $V = v$ on $\mr^n$.
\end{dfn}
\noindent Because the norm $U \mapsto \(\int_{\mr^N_+} \rho^{1-2\gamma} |\nabla U|_{\bg} dx\)^{1/2}$ is equivalent to $\|\cdot\|_{\drn}$,
the space $\drn$ is suitable to deal with \eqref{eq-lin-12}.
Also, we can immediately extend the space $C_c^{\infty}(\omrn)$ of test functions
in Definition \ref{dfn-w-sol} to $C_c(\omrn) \cap \drn$ by the method of mollifiers. 

\medskip The following local regularity result for a weak solution to \eqref{eq-lin-12} can be proved.
\begin{lemma}\label{lemma-reg}
Suppose that $U \in L^{\infty}(\mr^N_+)$ is a weak solution of \eqref{eq-lin-12}. Fix any $\varrho > 0$.
If $\Phi \in L^{\infty}(B^N_+(0,\varrho))$ and $\zeta \in L^{\infty}(B^n(0,\varrho))$,
then $U \in C^{\vartheta}(B^N_+(0,\varrho/2))$ for some $\vartheta \in (0,1)$ and
\begin{equation}\label{eq-reg-1}
\|U\|_{C^{\vartheta}(B^N_+(0,\varrho/2))} \le C\(\|U\|_{L^2(B^N_+(0,\varrho))}
+ \|\Phi\|_{L^{\infty}(B^N_+(0,\varrho))} + \|\zeta\|_{L^{\infty}(B^n(0,\varrho))}\).
\end{equation}
The constant $C > 0$ depends only on $N$, $\gamma$ and $\varrho$.
\end{lemma}
\begin{rmk}
We may relax the integrability condition of $U$, $\Phi$ and $\zeta$ to get more general results.
However, the current setting is sufficient for our purpose, so we do not pursue in this direction.
\end{rmk}
\begin{proof}[Proof of Lemma \ref{lemma-reg}]
By applying the standard Moser iteration technique with the John-Nirenberg inequality for $BMO(B^N_+(0,\vr); x_N^{1-2\gamma} dx)$, we obtain Moser's Harnack inequality:
If $U \in L^{\infty}(\mr^N_+)$ is a nonnegative weak solution to \eqref{eq-lin-12}, then there exists $C > 0$ depending only on $N, \gamma$ and $\varrho$ such that for any $0 < \varrho' < \varrho/2$
\[\sup_{x \in B^N_+(0,\varrho')} U \le C \(\inf_{x \in B^N_+(0,\varrho')} U
+ (\varrho')^2 \|\Phi\|_{L^{\infty}(B^N_+(0,\varrho))} + (\varrho')^{2\gamma} \|\zeta\|_{L^{\infty}(B^n(0,\varrho))}\).\]
Inequality \eqref{eq-reg-1} is its consequence.
For a proof in a similar setting, refer to Proposition 2.6 of \cite{JLX} and Propositions 3.1, 3.2 of \cite{TX}.
In fact, our case is simpler because we assumed that $U \in L^{\infty}(\mr^N_+)$ so that we do not need to trim it.
\end{proof}

\medskip In the remaining part of this subsection, we are concerned about the weak maximum principles for weighted Neumann problems.
\begin{lemma}\label{lemma-lin-1}
Suppose that $U \in \drn$ satisfies the inequality 
\[\begin{cases}
-\textnormal{div}_{\bg}\(\rho^{1-2\gamma}\nabla U\) \ge 0 &\text{in } \(\mr^N_+, \bg\),\\
\pa^{\gamma}_{\nu} U \ge 0 &\text{on } \mr^n
\end{cases}\]
weakly. Then $U \ge 0$ in $\mr^N_+$.
\end{lemma}
\begin{proof}
It holds that
\[\kappa_{\gamma} \int_{\mr^N_+} \rho^{1-2\gamma} \la \nabla U, \nabla V \ra_{\bg} dx \ge \int_{\mr^n} \pa^{\gamma}_{\nu} U \cdot V dx \ge 0 \quad \text{for any nonnegative } V \in C^{\infty}_c(\omrn).\]
Since the space $C^{\infty}_c(\omrn)$ is dense in $\drn$, we can insert $V = U_- \in \drn$ in the above inequality to get $V = 0$.
Hence $U \ge 0$ in $\mr^N_+$.
\end{proof}

\noindent The following generalized maximum principle will be used in Lemma \ref{lemma-red-2}.
\begin{lemma}\label{lemma-gen-max}
Suppose that $U \in L^{\infty}(\mr^N_+) \cap \drn$ satisfies
\begin{equation}\label{eq-gen-max-1}
\begin{cases}
LU := -\textnormal{div}_{\bg}\(\rho^{1-2\gamma}\nabla U\) + E(\rho)U \ge 0 &\text{in } \(\mr^N_+, \bg\),\\
\pa^{\gamma}_{\nu} U \ge 0 &\text{on } \mr^n
\end{cases} \end{equation}
weakly (in the sense of the adequate modification of Definition \ref{dfn-w-sol}).
Assume also that there exists a function $W \in \drn$ such that $W \in C(\omrn)$ , $\nabla W \in L^{\infty}(\omrn)$,
\begin{equation}\label{eq-gen-max-3}
\begin{cases}
LW = -\textnormal{div}_{\bg}\(\rho^{1-2\gamma}\nabla W\) + E(\rho)W \ge 0 &\text{in } \(\mr^N_+, \bg\),\\
W > 0 &\text{on } \omrn,\\
\pa^{\gamma}_{\nu} W \ge 0 &\text{on } \mr^n,
\end{cases} \end{equation}
and $|U(x)|/W(x) \to 0$ uniformly as $|x| \to \infty$, then $U \ge 0$ in $\mr^N_+$.
\end{lemma}
\begin{proof}
The proof is in the spirit of that of \cite[Lemma A.3]{JLX}.
By testing \eqref{eq-gen-max-1} with $W^{-1}\Phi \in C_c(\omrn) \cap \drn$,
we observe that the function $V := W^{-1}U$ satisfies
\begin{equation}\label{eq-gen-max-2}
\int_{\mr^N_+} \left[ \rho^{1-2\gamma} \( \la \nabla V, \nabla \Phi \ra_{\bg} - 2 \la \nabla V, \nabla W \ra_{\bg} W^{-1}\Phi
+ \la \nabla W, \nabla \(VW^{-1}\Phi\)\ra_{\bg} \) + E(\rho) V\Phi \right] dx \ge 0
\end{equation}
for all nonnegative function $\Phi \in C_c^{\infty}(\mr^N_+)$.
By density and \eqref{eq-sobolev-2}, it is also allowed to take any nonnegative $\Phi \in \drn$ with compact support into \eqref{eq-gen-max-2}.

To the contrary, suppose that $\inf\limits_{x \in \mr^N_+} V(x) < -m < 0$.
If we define a function $V_m = V + m$, then $(V_m)_- \ge 0$ has a compact support since $|V(x)| \to 0$ uniformly in $x$ as $|x| \to \infty$ by the hypothesis.
Therefore putting $\Phi = (V_m)_-$ in \eqref{eq-gen-max-2} and employing \eqref{eq-gen-max-3} with the test function $W^{-1}V(V_m)_-$, we obtain
\[\int_{\mr^N_+} \rho^{1-2\gamma} |\nabla (V_m)_-|^2_{\bg} dx \le 2 \int_{\mr^N_+} \rho^{1-2\gamma} \la \nabla W, \nabla (V_m)_- \ra_{\bg} W^{-1} (V_m)_- dx.\]
Now, by applying H\"older's inequality and the boundedness of $W^{-1}$ and $\nabla W$ on the previous inequality,
and then utilizing the weighted Sobolev inequality \eqref{eq-sobolev-2}, we find that
\[\int_{\text{supp}(V_m)_-} x_N^{1-2\gamma} dx \ge C > 0\]
where supp$(V_m)_-$ is the support of $(V_m)_-$.
However since $x_N^{1-2\gamma} \in L^1_{\text{loc}}(\omrn)$ and $|\text{supp}(V_m)_-| \to 0$ as $m \to -\inf\limits_{x \in \mr^N_+} V(x)$,
the left-hand side should go to 0 as well, which is absurd.
We have reached a contradiction, and so $V$ (or $U$) $\ge 0$ in $\mr^N_+$.
\end{proof}

\subsection{Existence and decay estimate for solutions to degenerate elliptic equations}
This part is devoted to study existence and decay property of solutions to degenerate elliptic equations.

\medskip
Assuming that $n > 2\gamma + 4(d_0+1) + 2/3$, let us set three weighted norms
\begin{align*}
\|U\|_* &= \sup_{x \in \mr^N_+} \left[\chi_{\{|x - (\sigma,0)| \le \nu/2\}} \cdot \(\frac{\mu \ep^{\kappa-{n+2\gamma \over 2}}} {\ep^{\kappa-2\gamma-(2d_0+2)} + |x-(\sigma,0)|^{\kappa-2\gamma-(2d_0+2)}}
+ \eta_0 {\ep^{\kappa - {n+2\gamma \over 2}} \over \nu^{\kappa}}\)^{-1} \right.\\
&\hspace{245pt} \left. + \chi_{\{|x - (\sigma,0)| \ge \nu/2\}} \cdot \frac{|x - (\sigma,0)|^{\kappa-2\gamma}}{\eta_0 \ep^{\kappa - {n+2\gamma \over 2}}}\right] \cdot |U(x)|, \\
\|U\|_{**} &= \sup_{x \in \mr^N_+} \left[\chi_{\{|x - (\sigma,0)| \le \nu/2\}} \cdot \frac{\ep^{\kappa-2\gamma-2d_0} + |x-(\sigma,0)|^{\kappa-2\gamma-2d_0}} {\mu \ep^{\kappa-{n+2\gamma \over 2}}}
+ \chi_{\{|x - (\sigma,0)| \ge \nu/2\}} \cdot \frac{|x - (\sigma,0)|^{\kappa-2\gamma+2}}{\eta_0 \ep^{\kappa - {n+2\gamma \over 2}}} \right] \cdot |U(x)|
\end{align*}
and
\[\|v\|_{***} = \sup_{\bx \in \mr^n} \left[\chi_{\{|\bx - \sigma| \le \nu/2\}} \cdot \(\frac{\mu \ep^{\kappa-{n+2\gamma \over 2}}} {\ep^{\kappa-(2d_0+2)} + |\bx-\sigma|^{\kappa-(2d_0+2)}}
+ \eta_0 {\ep^{\kappa - {n+2\gamma \over 2}} \over \nu^{\kappa}}\)^{-1}
+ \chi_{\{|\bx - \sigma| \ge \nu/2\}} \cdot \frac{|\bx - \sigma|^{\kappa}}{\eta_0 \ep^{\kappa - {n+2\gamma \over 2}}} \right] \cdot |v(\bx)|\]
for any fixed number
\begin{equation}\label{eq-kappa}
\kappa \in \(\max\left\{ {n+2\gamma \over 2} + 2(d_0+1) + {1 \over 3}, n-2\gamma \right\}, n\),
\end{equation}
small parameters $\nu, \eta_0 \gg \ep > 0$, points $(\dt) \in \ma$,
and functions $U = U(\bx, x_N)$ in $\mr_+^N$ and $v = v(\bx)$ on $\mr^n$.
(Here the dimension assumption implies that
$\mu \ep^{-{n-2\gamma \over 2}+2d_0+2} \to \infty$ and $\mu^{-2} \ep^{\kappa - {n+2\gamma \over 2}} \nu^{-\kappa} \to 0$ 
as $\ep \to 0$.)
Then we define the Banach spaces
\begin{equation}\label{eq-mh-1}
\mh_1 = \left\{U \in \drn \cap C(\omrn): \|U\|_* < \infty\right\}, \quad \mh_2 = \left\{U \in L^{\infty}(\mr^N_+): \|U\|_{**} < \infty \right\}
\end{equation}
and
\[\mh_3 = \left\{v \in L^{\infty}(\mr^n): \|v\|_{***} < \infty \right\}\]
where the space $\mh_1$ is endowed with the norm $\|\cdot\|_{\drn} + \|\cdot\|_*$.

We solve an inhomogeneous degenerate equation with homogeneous weighted Neumann condition and obtain an estimate for the solution.
\begin{lemma}\label{lemma-lin-2}
Let $\ep$ and $\eta_0$ be the small positive numbers chosen in \eqref{eq-h} and \eqref{eq-eta_0}.
For any fixed point $(\dt) \in \ma$ and a function $\Phi \in \mh_2$, the equation
\begin{equation}\label{eq-lin-21}
\begin{cases}
-\textnormal{div}_{\bg}\(\rho^{1-2\gamma}\nabla U\) = x_N^{1-2\gamma} \Phi &\text{in } \(\mr^N_+, \bg\),\\
\pa^{\gamma}_{\nu} U = 0 &\text{on } \mr^n
\end{cases} \end{equation}
has a unique solution $U_0 \in \mh_1$ satisfying
\begin{equation}\label{eq-lin-22}
\|U_0\|_* \le C \|\Phi\|_{**}.
\end{equation}
Here the constant $C > 0$ relies only on $n, \gamma$ and $\kappa$.
\end{lemma}
\begin{proof}
\noindent \medskip \textsc{Step 1 (A priori estimate).}
Suppose that $U_0 \in \drn$ is a solution of \eqref{eq-lin-21} for a given $\Phi \in \mh_2$.
It holds 
\begin{equation}\label{eq-lin-23}
-\text{div}_{\bg}\(\rho^{1-2\gamma}\nabla U\) = -\text{div}_{g_c}\(x_N^{1-2\gamma}\nabla U\) + \me(U) \quad \text{for any } U \in \mh_1
\end{equation}
where
\begin{equation}\label{eq-lin-24}
\begin{aligned}
\me(U) &:= \(1 - u^{-{2 \over N-2}(1-2\gamma)}\) \text{div}_{g_c}\(x_N^{1-2\gamma}\nabla U\)
+ u^{-{2 \over N-2}(1-2\gamma)} \(\delta^{ij} - \bg^{ij}\) x_N^{1-2\gamma} \pa_{ij}U \\
&\ - \pa_a u^{-{2 \over N-2}(1-2\gamma)} \bg^{ab} x_N^{1-2\gamma} \pa_b U
- u^{-{2 \over N-2}(1-2\gamma)} \pa_i \bg^{ij} x_N^{1-2\gamma} \pa_j U
\end{aligned}
\quad \text{in } \mr^N_+
\end{equation}
and $g_c$ is the standard metric in $\mr^N_+$.
Therefore the function $U_1 \in C(\omrn) \cap \drn$ defined to be
\begin{align*}
C_{11}^{-1}\|\Phi\|_{**}^{-1}U_1(x) = \begin{cases}
\mu \ep^{-{n-2\gamma \over 2}+2d_0+2} \left[2 - \(\dfrac{|x-(\sigma,0)|}{\ep}\)^2\right] + C_{12} &\text{for } |x-(\sigma,0)| \le \ep,\\
\dfrac{\mu \ep^{\kappa - {n+2\gamma \over 2}}} {|x-(\sigma,0)|^{\kappa-2\gamma-(2d_0+2)}}
+ C_{12} &\text{for } \ep \le |x-(\sigma,0)| \le \nu/2,\\
\dfrac{\eta_0 \ep^{\kappa - {n+2\gamma \over 2}}} {|x-(\sigma,0)|^{\kappa-2\gamma}} &\text{for } |x-(\sigma,0)| \ge \nu/2
\end{cases} \end{align*}
with a large constant $C_{11} > 0$ depending only on $n, \gamma$ and $\kappa$ and
\[C_{12} := \ep^{\kappa - {n+2\gamma \over 2}} \({\nu \over 2}\)^{-\kappa+2\gamma} \left[\eta_0 - \mu \({\nu \over 2}\)^{2d_0+2} \right],\]
satisfies
\begin{equation}\label{eq-lin-27}
\begin{cases}
-\text{div}_{\bg}\(\rho^{1-2\gamma}\nabla \(U_1 \pm U_0\)\) \ge 0 &\text{in } \(\mr^N_+, \bg\), \\
\pa^{\gamma}_{\nu} \(U_1 \pm U_0\) = 0 &\text{on } \mr^n.
\end{cases} \end{equation}
See below for the details. Then Lemma \ref{lemma-lin-1} will assert that $|U_0| \le U_1$, and hence \eqref{eq-lin-22} will be valid.

\medskip \noindent \textbf{Derivation of \eqref{eq-lin-27}.}
If $|x - (\sigma,0)| \le \ep$, then we have the inequality
\[- \text{div}_{\bg} \(\rho^{1-2\gamma} \nabla U_1\) \ge {3 \over 2} (n+2-2\gamma) C_{11}\|\Phi\|_{**} x_N^{1-2\gamma} \mu \ep^{-{n-2\gamma \over 2}+2d_0}\]
by \eqref{eq-lin-23}, \eqref{eq-lin-24} and Lemma \ref{lemma-z-reg}. Also since $|x| \le |x-(\sigma,0)| + |\sigma| \le (1+\varepsilon_0) \ep$,
\[\left|x_N^{1-2\gamma}\Phi\right| \le x_N^{1-2\gamma} \|\Phi\|_{L^{\infty}(\{|x-(\sigma,0)| \le \ep\})}
\le \|\Phi\|_{**} x_N^{1-2\gamma} \mu \ep^{-{n-2\gamma \over 2}+2d_0}.\]
In the meantime, we get for $\ep \le |x-(\sigma,0)| \le \nu/2$ that
\begin{align*}
- \text{div}_{\bg} \(\rho^{1-2\gamma} \nabla U_1\)
&\ge {1 \over 2} (\kappa-2\gamma-2(d_0+1))(n-\kappa+2(d_0+1)) C_{11}\|\Phi\|_{**} x_N^{1-2\gamma}
\mu \ep^{\kappa - {n+2\gamma \over 2}} |x - (\sigma,0)|^{-\kappa+2\gamma+2d_0}\\
&\ge x_N^{1-2\gamma} \|\Phi\|_{L^{\infty}\(\left\{\ep \le |x-(\sigma,0)| \le \nu/2\right\}\)} \ge \left|x_N^{1-2\gamma} \Phi\right|.
\end{align*}
Thus $- \text{div}_{\bg} \(\rho^{1-2\gamma} \nabla U_1\) \ge \left|\text{div}_{\bg} \(\rho^{1-2\gamma} \nabla U_0\)\right|$ in both cases.
Moreover a similar estimate can be performed to show that this inequality still holds when $|x-(\sigma,0)| \ge \nu/2$.
The identity $\pa^{\gamma}_{\nu} \(U_1 \pm U_0\) = \pa^{\gamma}_{\nu} U_1 = 0$ is readily checkable from the definition of $U_1$.

\medskip \noindent \textsc{Step 2 (Existence and Uniqueness).}
For each $\ell \in \mn$, we consider the mixed boundary value problem
\begin{equation}\label{eq-lin-26}
\begin{cases}
-\textnormal{div}_{\bg}\(\rho^{1-2\gamma}\nabla U\) = x_N^{1-2\gamma} \Phi &\text{in } B^N_+(0,\ell),\\
U = 0 & \text{on } \pa B^N_+(0,\ell) \cap \{x_N > 0\},\\
\pa^{\gamma}_{\nu} U = 0 &\text{on } B^n(0,\ell) \subset \mr^n.
\end{cases} \end{equation}
Then \eqref{eq-sobolev-2} and the Riesz representation (or the Lax-Milgram) theorem are applied to derive the unique solution $U_{0\ell} \in \md^{1,2}(\ell)$.

Also, by changing the argument in Step 1 a bit, we can obtain that $|U_{0\ell}| \le U_1$ in $B_+^N(0,\ell)$ for all $\ell \in \mn$.
Therefore
\[\kappa_{\gamma} \int_{B^N_+(0,\ell)} \rho^{1-2\gamma} |\nabla U_{0\ell}|_{\bg}^2 dx = \int_{B^N_+(0,\ell)} x_N^{1-2\gamma} \Phi U_{0\ell} dx \le C \|\Phi\|_{**}^2,\]
which implies the existence of the $\drn$-weak limit $U_0$.
It is easy to check with Lemma \ref{lemma-reg} that $U_0$ belongs to $\mh_1$ and satisfies both \eqref{eq-lin-21} and \eqref{eq-lin-22}, so the proof is finished.
\end{proof}

The next lemma provides decay property of a solution to the equation with a nonzero weighted Neumann boundary condition
\begin{equation}\label{eq-Neumann}
\begin{cases}
-\text{div}_{g_c}\(x_N^{1-2\gamma}\nabla U\) = 0 &\text{in } \mr^N_+,\\
\pa^{\gamma}_{\nu} U = \zeta &\text{on } \mr^n
\end{cases}
\end{equation}
for a given function $\zeta$ on $\mr^n$.
\begin{lemma}\label{lemma-lin-3}
Suppose that a function $\zeta$ on $\mr^n$ satisfies $\|\zeta\|_{***}' := \|(1+|x|^{\kappa})\zeta\|_{L^{\infty}(\mr^n)} < \infty$ for any fixed $\kappa \in (2\gamma, n)$.
Then there exists a constant $C > 0$ depending only on $N$, $\gamma$ and $\kappa$ such that
\[|\nabla_{\bx}^m U(x)| \le {C\|\zeta\|_{***}' \over 1+|x|^{\kappa-2\gamma+m}} \quad (m = 0, 1, 2)\]
for the solution $U \in \mh_1$ to problem \eqref{eq-Neumann} and all $x \in \mr^N_+$.
Moreover it holds that
\[|\pa_{x_N} U(x)| \le C\|\zeta\|_{***}' \({1 \over 1+|x|^{\kappa-2\gamma+1}} + {x_N^{2\gamma-1} \over 1+|\bx|^{\kappa}}\)\]
for every $x \in \mr^N_+$.
\end{lemma}
\begin{proof}
We borrow the idea of the proof of \cite[Lemma A.2]{CK}.
Note that the solution $U \in \mh_1$ can be expressed as
\begin{equation}\label{eq-Green}
U(\bx, x_N) = {1 \over \left|S^{n-1}\right|} \cdot {2^{1-2\gamma}\Gamma({n-2\gamma \over 2}) \over \Gamma({n \over 2})\Gamma(\gamma)} \int_{\mr^n} {1 \over |(\bx-\by,x_N)|^{n-2\gamma}} \zeta(\by) d\by
\end{equation}
(see e.g. \cite{CaS, CS, CKL}).

\medskip \noindent \textsc{Step 1 (Estimate for $U$).}
Without loss of generality, we may assume that $|x| \ge \vr$ for some fixed $\vr > 1$ large enough.
For $|\bx| \ge x_N$, by suitably modifying the proof of \cite[Lemma B.2]{WY}, we find
\begin{equation}\label{eq-Green-1}
\int_{\mr^n} {1 \over |(\bx-\by,x_N)|^{n-2\gamma}} {d\by \over 1+|\by|^{\kappa}} \le \int_{\mr^n} {1 \over |\bx-\by|^{n-2\gamma}} {d\by \over 1+|\by|^{\kappa}}
\le {C \over |\bx|^{\kappa-2\gamma}} \le {C \over |x|^{\kappa-2\gamma}}.
\end{equation}
If $|\bx| \le x_N$, then we immediately get that $|x| \le \sqrt{2}x_N$. This allows us to discover
\begin{equation}\label{eq-Green-2}
\int_{\left\{|\by| \le 2|\bx| \right\}} {1 \over |(\bx-\by,x_N)|^{n-2\gamma}} {d\by \over 1+|\by|^{\kappa}}
\le {1 \over x_N^{n-2\gamma}}\int_{\left\{|\by| \le 2|\bx| \right\}} {d\by \over 1+|\by|^{\kappa}}
\le {C \(1+|\bx|^{n-\kappa}\) \over x_N^{n-2\gamma}} \le {C \over |x|^{\kappa-2\gamma}}
\end{equation}
and
\begin{equation}\label{eq-Green-3}
\begin{aligned}
\int_{\left\{|\by| \ge 2|\bx| \right\}} {1 \over |(\bx-\by,x_N)|^{n-2\gamma}} {d\by \over 1+|\by|^{\kappa}}
&\le C \int_{\left\{|\by| \ge 2|\bx| \right\}} {1 \over (|\by|^2+x_N^2)^{n-2\gamma \over 2}} {d\by \over |\by|^{\kappa}} \\
&\le {C \over x_N^{\kappa-2\gamma}} \int_0^{\infty} {1 \over (t^2+1)^{n-2\gamma \over 2}}{dt \over t^{\kappa}} \le {C \over |x|^{\kappa-2\gamma}}.
\end{aligned} \end{equation}
By combining \eqref{eq-Green}-\eqref{eq-Green-3}, we realize that \eqref{eq-lin-11} is true.

\medskip \noindent \textsc{Step 2 (Estimate for $\nabla_{\bx} U$, $\nabla_{\bx}^2 U$ and $\pa_{x_N} U$).}
We can handle the situation $|\bx| \le x_N$ as in \eqref{eq-Green-2} and \eqref{eq-Green-3}, so assume $|\bx| \ge x_N$.

Consider the function $\nabla_{\bx} U$ first.
By differentiating \eqref{eq-Green} in $\bx$ and applying integration by parts, one sees that
\begin{align*}
|\nabla_{\bx} U(\bx, x_N)| &\le C\|\zeta\|_{***}' \left[\(\int_{\left\{|\by-\bx| \ge {|\bx| \over 2}\right\} \cap \left\{|\by| \ge {|\bx| \over 2}\right\}} + \int_{\left\{ |\by| \le {|\bx| \over 2}\right\}} \)
\nabla_{\by}\({1 \over |\bx-\by|^{n-2\gamma}}\) {d\by \over 1+|\by|^{\kappa}} \right.\\
&\left. \hspace{40pt} + \int_{\mr^n} {1 \over |\bx-\by|^{n-2\gamma}} {d\by \over 1+|\by|^{\kappa+1}}
+ \int_{\left\{|\by-\bx| = {|\bx| \over 2}\right\}} {1 \over |\bx-\by|^{n-2\gamma}} {dS_{\by} \over 1+|\by|^{\kappa}} \right]
\le {C\|\zeta\|_{***}' \over |x|^{\kappa-2\gamma+1}}
\end{align*}
where $dS_{\by}$ is the surface measure on the sphere $|\by-\bx| = {|\bx| \over 2}$.
Also we confirm
\begin{align*}
|\pa_{x_N} U(\bx, x_N)| &\le C\|\zeta\|_{***}' \( \int_{\left\{|\by-\bx| \ge {|\bx| \over 2}\right\}} + \int_{\left\{|\by-\bx| \le {|\bx| \over 2}\right\}} \) {x_N \over |(\by, x_N)|^{n-2\gamma+2}} {d\by \over |\bx-\by|^{\kappa}} \\
&\le C\|\zeta\|_{***}' \left[{x_N^{2\gamma-1} \over |\bx|^{\kappa}} \int_{\mr^n} {d\by \over |(\by,1)|^{n-2\gamma+2}}
+ {x_N \over |x|^{n-2\gamma+2}} \int_{\left\{|\by-\bx| \le {|\bx| \over 2}\right\}} {d\by \over |\bx-\by|^{\kappa}} \right] \\
&\le C\|\zeta\|_{***}' \left[{x_N^{2\gamma-1} \over |\bx|^{\kappa}} + {1 \over |x|^{\kappa-2\gamma+1}} \right].
\end{align*}
The estimate of the function $\nabla_{\bx}^2 U$ is similar to that of $\nabla_{\bx} U$. This establishes the proof.
\end{proof}

\subsection{Linear theory}
The goal of this subsection is to find a function $\Psi \in \mh_1$ and numbers $(c_0, \cdots, c_n) \in \mr^{n+1}$ which solve the linear problem
\begin{equation}\label{eq-lin-0}
\begin{cases}
-\text{div}_{\bg}\(\rho^{1-2\gamma}\nabla \Psi\) + E(\rho)\Psi = x_N^{1-2\gamma}\Phi &\text{in } \(\mr^N_+, \bg\),\\
\Psi = \psi &\text{on } \mr^n,\\
\pa^{\gamma}_{\nu} \Psi - \({n+2\gamma \over n-2\gamma}\) w_{\ls}^{4\gamma \over n-2\gamma}\psi
= \zeta + \sum\limits_{m=0}^n c_mw_{\ls}^{4\gamma \over n-2\gamma}z_{\ls}^m &\text{on } \mr^n,\\
\int_{\mr^n} w_{\ls}^{4\gamma \over n-2\gamma}z_{\ls}^0 \psi d\bx = \int_{\mr^n} w_{\ls}^{4\gamma \over n-2\gamma}z_{\ls}^1 \psi d\bx = \cdots = \int_{\mr^n} w_{\ls}^{4\gamma \over n-2\gamma}z_{\ls}^n \psi d\bx = 0
\end{cases} \end{equation}
for given functions $\Phi \in \mh_2$ and $\zeta \in \mh_3$.

\begin{prop}\label{prop-lin}
Suppose that $n > 2\gamma + 4(d_0+1) + 2/3$.
Then, for all sufficiently small parameters $0 < \ep \ll \nu, \eta_0$ satisfying $\nu |\log \ep| \ge 1/100$, points $(\dt) \in \ma$ and functions $\Phi \in \mh_2,\ \zeta \in \mh_3$, problem \eqref{eq-lin-0} admits
a unique solution $\Psi \in \mh_1$ and $\mathbf{c} = (c_0, \cdots, c_n) \in \mr^{n+1}$.
Moreover, there exists $C > 0$ depending only on $n, \gamma$ and $\kappa$ such that
\begin{equation}\label{eq-lin-11}
\|\Psi\|_* \le C \(\|\Phi\|_{**} + \|\zeta\|_{***}\).
\end{equation}
\end{prop}
\begin{proof}
The proof of this result is divided into two steps.

\medskip \noindent \textsc{Step 1 (A priori estimate).}
In this step, we first show \eqref{eq-lin-11} assuming that $\Psi \in \mh_1$ is a solution of \eqref{eq-lin-0}. For this aim, we argue by contradiction.

To emphasize that the metric $\bg = \exp(h)$ and the defining function $\rho$ depend on the choice of $\ep$ (see \eqref{eq-h}),
we will write $\bg_{\ep} = \bg$ and $\rho_{\ep} = \rho$ throughout the proof.

Suppose that there exists no constant $C > 0$ such that \eqref{eq-lin-11} holds uniformly for any choice of $\ep> 0$ and $\zeta \in \mh_3$.
Then there are sequences of numbers $\ep_{\ell} > 0$ and $\textbf{c}_{\ell} = (c_{0\ell}, \cdots, c_{n\ell}) \in \mr^{n+1}$,
points $\(\delta_{\ell}, \tau_{\ell}\) \in \ma$,
and functions $\Psi_{\ell} \in \mh_1$, $\Phi_{\ell} \in \mh_2$ and $\zeta_{\ell} \in \mh_3$
such that they satisfy \eqref{eq-lin-0} with $\bg = \bg_{\ep_{\ell}}$ and $\rho = \rho_{\ep_{\ell}}$ for each $\ell \in \mn$, as well as
\begin{equation}\label{eq-lin-05}
\|\Psi_{\ell}\|_* = 1, \quad \|\Phi_{\ell}\|_{**} + \|\zeta_{\ell}\|_{***} \to 0,
\end{equation}
$\(\delta_{\ell}, \tau_{\ell}\) \to \(\delta_0, \tau_0\) \in \overline{\ma}$ and $\ep_{\ell} \to 0$ as $\ell \to \infty$.
By \eqref{eq-error}, \eqref{eq-E-tilde}, \eqref{eq-E-tilde-2} and \eqref{eq-z-ann}, we have
\begin{equation}\label{eq-lin-04}
x_N^{2\gamma-1} E(\rho)(x) = \begin{cases}
O\(\mu^2 \(\ep^{4d_0+2} + |x-(\sigma,0)|^{4d_0+2}\)\) &\text{for } |x-(\sigma,0)| \le \nu/2,\\
O\(\eta_0 \(1+|x-(\sigma,0)|^4\)^{-1}\) &\text{for } |x-(\sigma,0)| \ge \nu/2.
\end{cases} \end{equation}
Thus from Lemma \ref{lemma-lin-2} we get a solution $\Psi_{1\ell} \in \mh_1$ to the equation
\[\begin{cases}
-\text{div}_{\bg_{\ell}}\(\rho_{\ell}^{1-2\gamma}\nabla \Psi_{1\ell}\) = - E(\rho_{\ell}) \Psi_{\ell} + x_N^{1-2\gamma} \Phi_{\ell} &\text{in } \(\mr^N_+, \bg_{\ell}\),\\
\pa^{\gamma}_{\nu} \Psi_{1\ell} = 0 &\text{on } \mr^n
\end{cases}\]
such that
\[\|\Psi_{1\ell}\|_* \le C \(\left\|x_N^{2\gamma-1} E(\rho_{\ell}) \Psi_{\ell}\right\|_{**} + \|\Phi_{\ell}\|_{**}\)
\le C\(\eta_0 \|\Psi_{\ell}\|_* + \|\Phi_{\ell}\|_{**}\)\]
where $\bg_{\ell} := \bg_{\ep_\ell}$ and $\rho_{\ell} := \rho_{\ep_\ell}$.
Moreover, by arguing as Step 1 in the proof of Lemma \ref{lemma-lin-2}, we deduce
\begin{equation}\label{eq-lin-08}
\|\Psi_{1\ell}\|_{*,\text{inner}} \le C\(\mu_{\ell} \left|\log \ep_{\ell}\right|^{2d_0 + 2\gamma} \cdot \|\Psi_{\ell}\|_* + \|\Phi_{\ell}\|_{**}\) \to 0 \quad \text{as } \ell \to \infty
\end{equation}
for $\mu_{\ell} := \ep_{\ell}^{1/3}$ and
\[\|U\|_{*,\text{inner}} := \sup_{\left\{x \in \mr^N_+: |x-(\sigma,0)| \le \nu/2 \right\}} \(\frac{\mu \ep^{\kappa-{n+2\gamma \over 2}}} {\ep^{\kappa-2\gamma-(2d_0+2)} + |x-(\sigma,0)|^{\kappa-2\gamma-(2d_0+2)}}
+ \eta_0 {\ep^{\kappa - {n+2\gamma \over 2}} \over \mu^2 \nu^{\kappa}}\)^{-1} \cdot |U(x)|.\]

Now the function $\Psi_{2\ell} := \Psi_{\ell} - \Psi_{1\ell} \in \mh_1$ satisfies
\begin{equation}\label{eq-lin-01}
\begin{cases}
-\text{div}_{\bg_{\ell}}\(\rho_{\ell}^{1-2\gamma}\nabla \Psi_{2\ell}\) = 0 \hspace{235pt} \text{in } \(\mr^N_+, \bg_{\ell}\),\\
\Psi_{2\ell} = \psi_{2\ell} \text{ and } \Psi_{1\ell} = \psi_{1\ell} \hspace{229pt} \text{on } \mr^n,\\
\pa^{\gamma}_{\nu} \Psi_{2\ell} - \({n+2\gamma \over n-2\gamma}\) w_{\ell}^{4\gamma \over n-2\gamma}\psi_{2\ell}
= \left[\zeta_{\ell} + \({n+2\gamma \over n-2\gamma}\) w_{\ell}^{4\gamma \over n-2\gamma}\psi_{1\ell}\right]
+ \sum\limits_{m=0}^n (c_m)_{\ell} w_{\ell}^{4\gamma \over n-2\gamma}z_{\ell}^m \hspace{41pt} \text{on } \mr^n,\\
\int_{\mr^n} w_{\ell}^{4\gamma \over n-2\gamma}z_{\ell}^0 \psi_{2\ell} d\bx = -\int_{\mr^n} w_{\ell}^{4\gamma \over n-2\gamma}z_{\ell}^0 \psi_{1\ell} d\bx, \cdots,
\int_{\mr^n} w_{\ell}^{4\gamma \over n-2\gamma}z_{\ell}^n \psi_{2\ell} d\bx = - \int_{\mr^n} w_{\ell}^{4\gamma \over n-2\gamma}z_{\ell}^n \psi_{1\ell} d\bx,
\end{cases} \end{equation}
and
\begin{equation}\label{eq-lin-02}
1 - C\eta_0 \le \liminf_{\ell \to \infty} \|\Psi_{2\ell}\|_* \le \limsup_{\ell \to \infty} \|\Psi_{2\ell}\|_* \le 1 + C\eta_0
\end{equation}
where $w_{\ell} := w_{\ep_{\ell}\delta_{\ell}, \ep_{\ell}\tau_{\ell}}$, $z^0_{\ell} := z^0_{\ep_{\ell}\delta_{\ell}, \ep_{\ell}\tau_{\ell}},
\cdots, z^n_{\ell} := z^n_{\ep_{\ell}\delta_{\ell}, \ep_{\ell}\tau_{\ell}}$.
Testing \eqref{eq-lin-01} with each $Z^m_{\ell} := Z^m_{\ep_{\ell}\delta_{\ell}, \ep_{\ell}\tau_{\ell}}$ for $m = 0, \cdots, n$ and $\Psi_{2\ell}$, we find with Lemma \ref{lemma-W-decay} and the assumption $\kappa > (n+2\gamma)/2 + 2(d_0+1) + 1/3$ that
\begin{align*}
&\ (c_m)_{\ell} (\ep_{\ell}\delta_{\ell})^{-2} \int_{\mr^n}w_1^{4\gamma \over n-2\gamma} \(z_1^m\)^2 d\bx
= \sum_{\tm=0}^n (c_{\tm})_{\ell} \int_{\mr^n} w_{\ell}^{4\gamma \over n-2\gamma}z_{\ell}^{\tm} z_{\ell}^m d\bx \\
&= \kappa_{\gamma} \int_{\mr^N_+} \rho_{\ell}^{1-2\gamma} \la \nabla \Psi_{2\ell}, \nabla Z^m_{\ell} \ra_{\bg} dx - \int_{\mr^n} \zeta_{\ell} z_{\ell}^m d\bx \\
&= - \({n+2\gamma \over n-2\gamma}\) \int_{\mr^n} w_{\ell}^{4\gamma \over n-2\gamma} z_{\ell}^m \psi_{1\ell} d\bx - \int_{\mr^n} \zeta_{\ell} z_{\ell}^m d\bx + O\(\mu_{\ell}\ep_{\ell}^{2d_0+1} \cdot \(\int_{\mr^N_+} x_N^{1-2\gamma} |\nabla \Psi_{2\ell}|^2 dx\)^{1/2}\) \\
&= O\(\mu_{\ell} \ep_{\ell}^{2d_0+1} \|\Psi_{1\ell}\|_{*,\text{inner}} + \ep_{\ell}^{\kappa-1} \left|\log \ep_{\ell}\right|^{\kappa} + \mu_{\ell} \ep_{\ell}^{2d_0+1}\|\zeta_{\ell}\|_{***} \) 
+ O\(\mu_{\ell}\ep_{\ell}^{2d_0+1} \cdot \(\int_{\mr^N_+} x_N^{1-2\gamma} |\nabla \Psi_{2\ell}|^2 dx\)^{1/2}\)
\end{align*}
and
\begin{align*}
&\ \kappa_{\gamma} \int_{\mr^N_+} x_N^{1-2\gamma} |\nabla \Psi_{2\ell}|^2 dx \le C \kappa_{\gamma} \int_{\mr^N_+} \rho_{\ell}^{1-2\gamma} |\nabla \Psi_{2\ell}|^2_{\bg} dx \\
&= C\left[ \({n+2\gamma \over n-2\gamma}\) \int_{\mr^n} w_{\ell}^{4\gamma \over n-2\gamma}\psi_{2\ell}^2 d\bx
+ \int_{\mr^n} \left[\zeta_{\ell} + \({n+2\gamma \over n-2\gamma}\) w_{\ell}^{4\gamma \over n-2\gamma}\psi_{1\ell}\right] \psi_{2\ell} d\bx
- \sum_{m=0}^n (c_m)_{\ell} \int_{\mr^n} w_{\ell}^{4\gamma \over n-2\gamma}z_{\ell}^m \psi_{1\ell} d\bx \right] \\
&= O\(\mu_{\ell}^2 \ep_{\ell}^{4d_0+4}\) + o\(\mu_{\ell} \ep^{2d_0+1} \sum_{\tm=0}^n |(c_{\tm})_{\ell}|\).
\end{align*}
As a result, it holds that
\begin{equation}\label{eq-lin-06}
\int_{\mr^N_+} x_N^{1-2\gamma} |\nabla \Psi_{2\ell}|^2 dx = O\(\mu_{\ell}^2 \ep_{\ell}^{4d_0+4}\)
\quad \text{and} \quad \sum_{\tm=0}^n |(c_{\tm})_{\ell}| = o\(\mu_{\ell} \ep_{\ell}^{2d_0+3}\) \quad \text{as } \ell \to \infty.
\end{equation}

The previous estimate implies that
\[\wtps_{2\ell} := \(\mu \ep^{2d_0+2}\)^{-1} \ep^{n-2\gamma \over 2} \Psi_{2\ell} (\ep \cdot + \sigma) \rightharpoonup \wtps_{20} \quad \text{weakly in } \drn\]
up to a subsequence (which is still denoted as $\wtps_{2\ell}$).
Additionally, by the Schauder estimates \cite[Proposition 3.2]{GQ} or \cite[Lemma 4.5]{CS}, it can be further assumed that $\wtps_{2\ell}$ converges to a function $\wtps_{20}$ uniformly over compact sets in $\omrn$.
With this convergence property, we observe that $\wtps_{20}$ satisfies
\[\begin{cases}
-\text{div}_{g_c}\(x_N^{1-2\gamma}\nabla \Psi\)= 0 &\text{in } \mr^N_+,\\
\Psi = \psi &\text{on } \mr^n,\\
\pa^{\gamma}_{\nu} \Psi = \({n+2\gamma \over n-2\gamma}\) w_1^{4\gamma \over n-2\gamma}\psi &\text{on } \mr^n,\\
\int_{\mr^n} w_1^{4\gamma \over n-2\gamma}z_1^0 \psi d\bx = \cdots = \int_{\mr^n} w_1^{4\gamma \over n-2\gamma}z_1^n \psi d\bx = 0
\end{cases}\]
so that $\wtps_{20} = 0$ according to \cite{DDS} (see the paragraph after \eqref{eq-cng}).
Hence if we choose any $\varepsilon > 0$, then
\begin{equation}\label{eq-lin-03}
\left\|w_\ell^{4\gamma \over n-2\gamma} \psi_{2\ell}\right\|_{***}
\le \(1+\vr^{\kappa-(2d_0+2)}\) \left\|w_1^{4\gamma \over n-2\gamma} \wtps_{2\ell} \right\|_{L^{\infty}\(B^n(0,\vr)\)}
+ C \({1 \over \vr^{\gamma}} + \ep_{\ell}^{2\gamma} \left|\log \ep_{\ell}\right|^{2\gamma}\) \|\Psi_{2\ell}\|_*
\le {\varepsilon \over 3} + {\varepsilon \over 3} + {\varepsilon \over 3} = \varepsilon
\end{equation}
for $\vr > 0$ and $\ell \in \mn$ large - namely its leftmost side goes to 0 as $\ell \to \infty$.
We also have
\begin{equation}\label{eq-lin-07}
\left\|w_\ell^{4\gamma \over n-2\gamma} \psi_{1\ell}\right\|_{***} \le C \(\|\Psi_{1\ell}\|_{*,\text{inner}} + \ep_{\ell}^{2\gamma} \left|\log \ep_{\ell}\right|^{2\gamma} \|\Psi_{1\ell}\|_{***}\)
\quad \text{and} \quad
\left\|w_{\ell}^{4\gamma \over n-2\gamma}z_{\ell}^m\right\|_{***} \le C \(\mu \ep^{2d_0+3}\)^{-1}.
\end{equation}

Let us introduce a barrier function $U_2$ defined as
\begin{multline*}
C_{21}^{-1} \( \left\|w_\ell^{4\gamma \over n-2\gamma} \psi_{2\ell}\right\|_{***} + \|\zeta_{\ell}\|_{***}
+ \left\|w_\ell^{4\gamma \over n-2\gamma} \psi_{1\ell}\right\|_{***} + \(\mu_{\ell} \ep_{\ell}^{2d_0+3}\)^{-1} \sum_{\tm=0}^n |(c_{\tm})_{\ell}| \)^{-1} \cdot U_2(x) \\
= \begin{cases}
\begin{aligned}
\mu \ep^{2d_0+2} \ep^{-{n-2\gamma \over 2}} \left[2 - \left| \dfrac{x-(\sigma,0)}{\ep} \right|^2 + U_{3;\kappa-(2d_0+2)}\(\dfrac{|x-(\sigma,0)|}{\ep}\) \right]\\
+ \eta_0 \ep^{\kappa - {n+2\gamma \over 2}} \nu^{-\kappa} \left[2-|x-(\sigma,0)|^2 + U_{3;\kappa}(|x-(\sigma,0)|)\right] + C_{22}
\end{aligned} &\text{if } |x-(\sigma,0)| \le \ep,\\
\begin{aligned}
\dfrac{\mu \ep^{\kappa - {n+2\gamma \over 2}}}{|x-(\sigma,0)|^{\kappa-2\gamma-(2d_0+2)}}
+ \mu \ep^{2d_0+2} \ep^{-{n-2\gamma \over 2}} U_{3;\kappa-(2d_0+2)}\(\dfrac{|x-(\sigma,0)|}{\ep}\) \\
+ \eta_0 \ep^{\kappa - {n+2\gamma \over 2}} \nu^{-\kappa} \left[2-|x-(\sigma,0)|^2 + U_{3;\kappa}(|x-(\sigma,0)|)\right] + C_{22}
\end{aligned}
&\text{if } \ep \le |x-(\sigma,0)| \le \nu/2,\\
\eta_0 \ep^{\kappa - {n+2\gamma \over 2}} \(\dfrac{1}{|x - (\sigma,0)|^{\kappa -2\gamma}} + U_{3;\kappa}(|x-(\sigma,0)|)\) & \text{if } |x-(\sigma,0)| \ge \nu/2
\end{cases} \end{multline*}
for constants $C_{21} > 0$ large enough (depending on $n$, $\kappa$, $\gamma$, $\nu$ and $\eta_0$) and $C_{22} \in \mr$ suitably selected so that $U_2$ is continuous in $\mr^N_+$.
Here $U_{3;\kappa}(x) = U_{3;\kappa}(|x|) \in \drn$ is a radial function that solves
\[\begin{cases}
-\text{div}_{g_c}\(x_N^{1-2\gamma}\nabla U\) = 0 &\text{in } \mr^N_+,\\
\pa^{\gamma}_{\nu} U = \(1+|\bx|^{\kappa}\)^{-1} &\text{on } \mr^n.
\end{cases}\]
Then after some calculations using \eqref{eq-lin-23}, \eqref{eq-z-ann} and Lemma \ref{lemma-lin-3}, one finds that for all $\ell \in \mn$
\[\begin{cases}
-\text{div}_{\bg_{\ell}}\(\rho_{\ell}^{1-2\gamma}\nabla \(U_2 \pm \Psi_{2\ell}\)\) \ge 0 &\text{in } \(\mr^N_+, \bg_{\ell}\),\\
\pa^{\gamma}_{\nu} \(U_2 \pm \Psi_{2\ell}\) \ge 0 &\text{on } \mr^n.
\end{cases}\]
Consequently we see from \eqref{eq-lin-05}, \eqref{eq-lin-08}, \eqref{eq-lin-06}-\eqref{eq-lin-07}, Lemma \ref{lemma-lin-1} and Lemma \ref{lemma-lin-3} 
that
\[\|\Psi_{2\ell}\|_* \le C \( \left\|w_\ell^{4\gamma \over n-2\gamma} \psi_{2\ell}\right\|_{***} + \|\zeta_{\ell}\|_{***}
+ \left\|w_\ell^{4\gamma \over n-2\gamma} \psi_{1\ell}\right\|_{***} + \(\mu_{\ell} \ep_{\ell}^{2d_0+3}\)^{-1} \sum_{\tm=0}^n |(c_{\tm})_{\ell}| \) \to 0 \quad \text{as } \ell \to \infty.\]
However it contradicts \eqref{eq-lin-02}, meaning that \eqref{eq-lin-11} should be correct.
This concludes a priori estimate part of the proof.

\medskip \noindent \textsc{Step 2 (Existence).}
Set a subspace of $\mh_1$
\begin{equation}\label{eq-mzp}
\mz^\perp = \left\{\Psi \in \mh_1: \Psi = \psi \text{ on } \mr^n \text{ and } \int_{\mr^n} w_{\ls}^{4\gamma \over n-2\gamma}z_{\ls}^m \psi d\bx = 0 \text{ for all } m = 0, \cdots, n\right\}.
\end{equation}
Then expressed in the weak form, Eq. \eqref{eq-lin-0} is reduced to a problem finding $\Psi \in \mh_1$ such that
\begin{equation}\label{eq-lin-09}
\kappa_{\gamma} \int_{\mr^N_+} \rho^{1-2\gamma} \la \nabla \Psi, \nabla V \ra_{\bg} dx + \kappa_{\gamma} \int_{\mr^N_+} E(\rho) \Psi V dx
= \kappa_{\gamma} \int_{\mr^N_+} x_N^{1-2\gamma} \Phi V dx + \({n+2\gamma \over n-2\gamma}\) \int_{\mr^n} w_{\ls}^{4\gamma \over n-2\gamma} \psi v d\bx + \int_{\mr^n} \zeta v d\bx
\end{equation}
for any $V \in \mz^\perp$ where $V = v$ on $\mr^n$. See below for more explanation.
Moreover the above equation can be rewritten in the operational form
\[(\kappa_{\gamma} \cdot \text{Id}- K)\Psi = \widetilde{\Phi} + \widetilde{\zeta}\]
where $\widetilde{\Phi},\ \widetilde{\zeta} \in \mz^{\perp}$ are defined by the relation
\[\int_{\mr^N_+} \rho^{1-2\gamma} \la \nabla \( \widetilde{\Phi} + \widetilde{\zeta} \), \nabla V \ra_{\bg} dx
= \kappa_{\gamma} \int_{\mr^N_+} x_N^{1-2\gamma} \Phi V dx + \int_{\mr^n} \zeta v d\bx\]
holding for any $V \in \mz^\perp$ and $K$ is a compact operator in $\mz^\perp$ given by
\[\int_{\mr^N_+} \rho^{1-2\gamma} \la \nabla K(\Psi), \nabla V \ra_{\bg} dx
= \({n+2\gamma \over n-2\gamma}\) \int_{\mr^n} w_{\ls}^{4\gamma \over n-2\gamma} \psi v d\bx - \kappa_{\gamma} \int_{\mr^N_+} E(\rho) \Psi V dx\]
for every $V \in \mz^\perp$.
(One can prove existence of $\widetilde{\Phi},\ \widetilde{\zeta}$ and well-definedness and compactness of $K$ by applying the truncation argument
as in the proof of Lemma \ref{lemma-lin-2} with \eqref{eq-sobolev-2}, the Sobolev trace inequality in \cite{Xi} and \eqref{eq-lin-04}.)
In light of \eqref{eq-lin-11}, the operator $\kappa_{\gamma} \cdot \text{Id}- K$ must be injective in $\mz^\perp$.
Thus the Fredholm alternative guarantees that it is also surjective, from which we deduce the unique solvability of \eqref{eq-lin-0}.
\end{proof}

\subsection{Estimate for the error}
Let $E_{\ls} := -x_N^{2\gamma-1}(\me(W_{\ls}) + E(\rho)W_{\ls})$ be the error term where the operator $\me$ is defined in \eqref{eq-lin-24}.
The next lemma contains its estimate, especially showing that it is small as an element of $\mh_2$.
\begin{lemma}\label{lemma-error}
For fixed $\nu, \eta_0 \gg \ep > 0$ small and $(\dt) \in \ma$, we have
\begin{equation}\label{eq-E-est}
\|E_{\ls}\|_{**} \le C
\end{equation}
for $C > 0$ dependent only on $n$, $\gamma$ and $\kappa$.
\end{lemma}
\begin{proof}
We observe from \eqref{eq-lin-24} that
\begin{equation}\label{eq-E-est-1}
E_{\ls} = u^{-{2 \over N-2}(1-2\gamma)} \(\bg^{ij} - \delta^{ij}\) \pa_{ij}W_{\ls} + \pa_a u^{-{2 \over N-2}(1-2\gamma)} \bg^{ab} \pa_b W_{\ls}
+ u^{-{2 \over N-2}(1-2\gamma)} \pa_i \bg^{ij} \pa_j W_{\ls} - x_N^{2\gamma-1} E(\rho)W_{\ls}.
\end{equation}
Hence an application of Lemmas \ref{lemma-W-decay} and \ref{lemma-z-reg}, \eqref{eq-lin-04} and \eqref{eq-z-ann} yields
\begin{equation}\label{eq-E-est-2}
|E_{\ls}(x)| \le \begin{cases}
\dfrac{C \mu \ep^{n-2\gamma \over 2}}{\ep^{n-2\gamma-2d_0} + |x-(\sigma,0)|^{n-2\gamma-2d_0}} &\text{for } |x-(\sigma,0)| \le \nu/2,\\
\dfrac{C \eta_0 \ep^{n-2\gamma \over 2}}{|x-(\sigma,0)|^{n-2\gamma+2}} &\text{for } |x-(\sigma,0)| \ge \nu/2.
\end{cases} \end{equation}
For instance, the second term of $E_{\ls}$ in \eqref{eq-E-est-1} can be estimated as
\begin{align*}
\left|\pa_a u^{-{2 \over N-2}(1-2\gamma)} \bg^{ab} \pa_b W_{\ls}\right|
&\le C\(|\nabla_{\bx} z| \cdot |\nabla_{\bx} W_{\ls}| + |\pa_N z| \cdot |\pa_N W_{\ls}|\) \\
&\le C \mu^2 \ep^{n-2\gamma \over 2} \(\ep^{4d_0} + |x|^{4d_0}\) \left[{|x|^3 \over \ep^{n-2\gamma+1} + |x-(\sigma,0)|^{n-2\gamma+1}}
+ {|x|^2x_N^{2\gamma} \over \ep^n + |x-(\sigma,0)|^n} \right]\\
&\le C \mu^2 \ep^{n-2\gamma \over 2} \left[{\ep^{4d_0+3} + |x-(\sigma,0)|^{4d_0+3} \over \ep^{n-2\gamma+1} + |x-(\sigma,0)|^{n-2\gamma+1}}
+ {\ep^{4d_0+2+2\gamma} + |x-(\sigma,0)|^{4d_0+2+2\gamma} \over \ep^n + |x-(\sigma,0)|^n} \right]\\
&\le {C \nu^{2d_0+2} \mu^2 \ep^{n-2\gamma \over 2} \over \ep^{n-2\gamma-2d_0} + |x-(\sigma,0)|^{n-2\gamma-2d_0}}
\end{align*}
if $|x| \le \nu$.
The norm bound \eqref{eq-E-est} is immediately deduced by \eqref{eq-E-est-2}.
\end{proof}

\subsection{Solvability of the nonlinear problem}
We now prove that an intermediate problem
\begin{equation}\label{eq-inter}
\begin{cases}
-\text{div}_{\bg}\(\rho^{1-2\gamma}\nabla \Psi\) + E(\rho)\Psi = x_N^{1-2\gamma} E_{\ls} &\text{in } \(\mr^N_+, \bg\),\\
\Psi = \psi &\text{on } \mr^n,\\
\pa^{\gamma}_{\nu} \Psi - \({n+2\gamma \over n-2\gamma}\)w_{\ls}^{4\gamma \over n-2\gamma}\psi = N_{\ls}(\psi) + \sum\limits_{m=0}^n c_m w_{\ls}^{4\gamma \over n-2\gamma} z_{\ls}^m &\text{on } \mr^n,\\
\int_{\mr^n} w_{\ls}^{4\gamma \over n-2\gamma}z_{\ls}^0 \psi d\bx = \int_{\mr^n} w_{\ls}^{4\gamma \over n-2\gamma}z_{\ls}^1 \psi d\bx = \cdots = \int_{\mr^n} w_{\ls}^{4\gamma \over n-2\gamma}z_{\ls}^n \psi d\bx = 0
\end{cases}\end{equation}
to our main Eq. \eqref{eq-main-ext}, or \eqref{eq-main}, is solvable by using the contraction mapping argument.
Here
\begin{equation}\label{eq-nls}
N_{\ls}(\psi) := (w_{\ls}+\psi)_+^{n+2\gamma \over n-2\gamma} - w_{\ls}^{n+2\gamma \over n-2\gamma} - \({n+2\gamma \over n-2\gamma}\)w_{\ls}^{4\gamma \over n-2\gamma}\psi \in \mh_3.
\end{equation}
\begin{prop}\label{prop-inter}
For $\nu, \eta_0 \gg \ep > 0$ small enough, $n > 2\gamma + 4(d_0+1) + 2/3$ and $(\dt) \in \ma$ fixed, there exists a unique solution $\Psi_{\ls} \in \mh_1$ and $\textbf{c}_{\ls} = ((c_0)_{\ls}, \cdots, (c_n)_{\ls}) \in \mr^{n+1}$ to Eq. \eqref{eq-inter} such that
\begin{equation}\label{eq-inter-1}
\|\Psi_{\ls}\|_* \le C
\end{equation}
where $C > 0$ depends only on $n$, $\gamma$ and $\kappa$.
\end{prop}
\begin{proof}
According to Proposition \ref{prop-lin}, one can define an operator $T_{\ls}: \mh_2 \times \mh_3 \to \mh_1$ to be $T_{\ls}(\Phi, \zeta) = \Psi$
where $\Psi \in \mh_1$ solves Eq. \eqref{eq-lin-0} for given pairs $(\ls) \in (0,\infty) \times \mr^n$ and $(\Phi, \zeta) \in \mh_2 \times \mh_3$.
One also has that $\|T_{\ls}(\Phi, \zeta)\|_* \le M_1(\|\Phi\|_{**} + \|\zeta\|_{***})$ for some $M_1 > 0$.
In terms of this operator $T_{\ls}$, \eqref{eq-inter} is reformulated as
\[\Psi = T_{\ls}(E_{\ls}, N_{\ls}(\psi)) =: T_{\ls}'(\Psi) \quad \text{for } \Psi \in \mz^\perp\]
where $\mz^\perp$ is the space defined in \eqref{eq-mzp}. Let us set
\[\mb = \left\{\Psi \in \mz^\perp : \|\Psi\|_* \le M_2\right\}\]
with $M_2 > 0$ a number to be determined.
By using the facts that $\kappa > n-2\gamma$,
\[|N_{\ls}(\psi)| \le C \begin{cases}
w_{\ls}^{-{n-6\gamma \over n-2\gamma}} |\psi|^2 &\text{if } |\bx-\sigma| \le \nu/2, \\
|\psi|^{n+2\gamma \over n-2\gamma} &\text{if } |\bx-\sigma| \ge \nu/2,
\end{cases}\]
and
\[|N_{\ls}(\psi_1) - N_{\ls}(\psi_2)| \le C \begin{cases}
w_{\ls}^{-{n-6\gamma \over n-2\gamma}} (|\psi_1| + |\psi_2|) |\psi_1 - \psi_2| &\text{if } |\bx-\sigma| \le \nu/2, \\
\( |\psi_1|^{4\gamma \over n-2\gamma} + |\psi_2|^{4\gamma \over n-2\gamma} \) |\psi_1 - \psi_2| &\text{if } |\bx-\sigma| \ge \nu/2
\end{cases}\]
(which follows from the mean value theorem),
we easily get that $\|N_{\ls}(\psi)\|_{***} = o(1) \|\Psi\|_*$ and $\|N_{\ls}(\psi_1) - N_{\ls}(\psi_2)\|_{***} \le o(1) \|\Psi_1 - \Psi_2\|_*$.
Then, by \eqref{eq-E-est} also, we see that there exists a constant $M_3 > 0$ such that
\[\|T_{\ls}'(\Psi)\|_* \le M_1 \(\|E_{\ls}\|_{**} + \|N_{\ls}(\psi)\|_{***}\) \le M_1\(M_3 + o(1) \|\Psi\|_*\) \le 2M_1M_3\]
for all $\Psi \in \mb$ and
\[\|T_{\ls}'(\Psi_1) - T_{\ls}'(\Psi_2)\|_* = \|T_{\ls}\(0, N_{\ls}(\psi_1) - N_{\ls}(\psi_2)\)\|_* \le M_1 \|N_{\ls}(\psi_1) - N_{\ls}(\psi_2)\|_{***} \le o(1) \|\Psi_1 - \Psi_2\|_*\]
for any $\Psi_1$ and $\Psi_2 \in \mb$.
Therefore $T_{\ls}'$ is a contraction map on the set $\mb$ with the choice $M_2 = 2M_1M_3$.
The result follows from the contraction mapping theorem.
\end{proof}

One can also analyze the differentiability of the function $\Psi_{\edt}$ with respect to its parameter $(\dt)$.
\begin{lemma}\label{lemma-inter}
Given $n > 2\gamma + 4(d_0+1) + 2/3$ and small fixed numbers $\nu, \eta_0 \gg \ep > 0$, the map $(\dt) \in \ma \mapsto \Psi_{\edt} \in \mh_1$ is of class $C^1$.
Furthermore, there exists $C > 0$ depending only on $n$, $\gamma$ and $\kappa$ such that
\begin{equation}\label{eq-inter-2}
\left\|\nabla_{(\dt)} \Psi_{\edt}\right\|_* \le C \ep^{-1}.
\end{equation}
\end{lemma}
\begin{proof}
The proof is similar to that of \cite[Proposition 6.2]{WZ}.
\end{proof}

\subsection{Variational reduction}
Provided that the assumptions of Proposition \ref{prop-inter} are fulfilled and in particular a small $\ep > 0$ is fixed, let $J_0^{\gamma}$ be a localized energy functional given by
\begin{equation}\label{eq-loc-energy}
J_0^{\gamma}(\dt) = I^{\gamma}(W_{\edt} + \Psi_{\edt}) \quad \text{for } (\dt) \in \ma \subset (0,\infty) \times \mr^n
\end{equation}
where $I^{\gamma}$ is the functional defined in \eqref{eq-energy}.

\begin{lemma}\label{lemma-red}
The followings are valid provided that $\eta_0, \ep > 0$ small fixed and $n > 2\gamma + 4(d_0+1) + 2/3$.
\begin{enumerate}
\item The functional $J_0^{\gamma}$ is continuously differentiable.
\item If $(\delta(\ep), \tau(\ep)) > 0$ is a critical point of $J_0^{\gamma}$, then $(I^{\gamma})'(W_{\edte} + \Psi_{\edte}) = 0$.
\end{enumerate}
\end{lemma}
\begin{proof}
1. Since the functional $I^{\gamma}: \mh_1 \to \mr$ is a $C^1$-map, the assertion follows from Lemma \ref{lemma-inter}.

\medskip \noindent
2. Suppose that $\(J_0^{\gamma}\)'(\delta(\ep), \tau(\ep)) = 0$.
If we write $(\dt) = (\tau_0, \tau_1, \cdots, \tau_n)$, then
\[0 = \pa_{\tau_m} J_0^{\gamma}(\delta(\ep), \tau(\ep)) = \sum_{\tm=0}^n c_{\tm} \int_{\mr^n} w_{\edte}^{4\gamma \over n-2\gamma} z_{\edte}^{\tm}
\pa_{\tau_m}\(w_{\edte}^{4\gamma \over n-2\gamma} + \psi_{\edte}\) =: \sum_{\tm=0}^n c_{\tm} \hat{c}_{m \tm} \]
for $m = 0, \cdots, n$, where $\Psi_{\edte} = \psi_{\edte}$ on $\mr^n$.
According to \eqref{eq-inter-2}, the matrix $(\hat{c}_{m \tm})_{m,\tm = 0, \cdots, n}$ is diagonal dominant.
Thus $c_0 = \cdots = c_n = 0$ and so $(I^{\gamma})'(W_{\edte} + \Psi_{\edte}) = 0$.
\end{proof}

The next lemma implies that the solution $W_{\edte} + \Psi_{\edte}$ to problem \eqref{eq-main-ext} (or \eqref{eq-main}) has desired properties described in Theorem \ref{thm-main}.
Consequently, in view of the previous lemma, it suffices to find a critical point of $J_0^{\gamma}$ whose domain $\ma$ is finite dimensional.

\begin{lemma}\label{lemma-red-2}
The critical point $W_{\edte} + \Psi_{\edte} \in \mh_1$ of $I^{\gamma}$ is positive in $\omrn$ and of $C^{\vartheta}(\omrn)$ for some $\vartheta \in (0,1)$.
Also, if parameters $\nu, \eta_0 \gg \ep > 0$ are small enough, then there exists a constant $C > 0$ depending only on $n$ and $\gamma$ such that
\begin{equation}\label{eq-lower-bd}
\left\|W_{\edte} + \Psi_{\edte}\right\|_{L^{\infty}(\mr^n)} \ge C \ep^{-{n-2\gamma \over 2}}.
\end{equation}
\end{lemma}
\begin{proof}
\textsc{Step 1 (Positivity of $W_{\edte} + \Psi_{\edte}$).}
For the brevity, we write $U = W_{\edte} + \Psi_{\edte}$ for a fixed $\ep > 0$.
Fixing any $\kappa' < \kappa$ which satisfy \eqref{eq-kappa}, let us define $W$ (which should not be confused with the bubbles $W_{\ls}$) by
\[W(x) = C_{31} \begin{cases}
\mu \ep^{-{n-2\gamma \over 2}+2d_0+2} \left[2 - \(\dfrac{|x-(\sigma,0)|}{\ep}\)^2\right] + C_{32} &\text{for } |x-(\sigma,0)| \le \ep,\\
\dfrac{\mu \ep^{\kappa' - {n+2\gamma \over 2}}} {|x-(\sigma,0)|^{\kappa'-2\gamma-(2d_0+2)}}
+ C_{32} &\text{for } \ep \le |x-(\sigma,0)| \le \nu/2,\\
\dfrac{\eta_0 \ep^{\kappa' - {n+2\gamma \over 2}}} {|x-(\sigma,0)|^{\kappa'-2\gamma}} &\text{for } |x-(\sigma,0)| \ge \nu/2
\end{cases}\]
where $C_{31} > 0$ large and $C_{32} \in \mr$ chosen so that $W \in C(\omrn)$.
Then $W$ is a suitable barrier which makes it possible to apply Lemma \ref{lemma-gen-max}.
This leads us to deduce that $U$ is nonnegative in $\omrn$.

For the moment, we admit
\begin{equation}\label{eq-eig}
\lambda_1(-\Delta_{g^+}) > {n^2 \over 4} - \gamma^2 = {(N-1)^2 \over 4} - \gamma^2
\end{equation}
where $\lambda_1(-\Delta_{g^+})$ is the first eigenvalue (or the infimum of the spectra) of the operator $-\Delta_{g^+}$ acting on the space $L^2(\mr^N_+, g^+)$.
Its validity will be proved in the end of Appendix \ref{subsec-app-c}.
Then by Lemma 4.5-Theorem 4.7 and the discussion in Section 5 of Chang-Gonz\'alez \cite{GZ} (or \cite[Lemma 6.1]{CC}),
we realize that there is a special boundary defining function $\rho^*$ in $\mr^N_+$ such that $E(\rho^*) = 0$
and $\wtu := (\rho/\rho^*)^{(n-2\gamma)/2} U$ satisfies a degenerate elliptic equation of pure divergent form
\begin{equation}\label{eq-main-ext-2}
\begin{cases}
-\text{div}\(\(\rho^*\)^{1-2\gamma}\nabla \wtu\) = 0 &\text{in } \(\mr^N_+, \bg\),\\
\pa^{\gamma}_{\nu} \wtu = \wtu^{n+2\gamma \over n-2\gamma} - Q^{\gamma}_{\hh} \wtu &\text{on } \mr^n,
\end{cases} \end{equation}
where $Q^{\gamma}_{\hh}$ is the fractional scalar curvature.
By the strong maximum principle for uniformly elliptic operators, it is immediately obtained that $\wtu > 0$ in $\mr^N_+$.
On the other hand, if there is a point $x_0 \in \mr^n$ such that $\wtu(x_0) = 0$,
then the Hopf lemma \cite[Theorem 3.5]{GQ} for \eqref{eq-main-ext-2} implies that
\[0 > \pa^{\gamma}_{\nu} \wtu(x_0) = \wtu^{n+2\gamma \over n-2\gamma}(x_0) - Q^{\gamma}_{\hh}(x_0) \wtu(x_0) = 0,\]
a contradiction.
Therefore the function $\wtu$, or equivalently, $U$ must be positive in $\omrn$.

\medskip \noindent \textsc{Step 2 (Regularity property and Estimate of the lower bound).}
Because of \eqref{eq-inter-1}, our solution $U$ is essentially bounded in $\mr^N_+$.
Hence it is in $C^{\vartheta}(\omrn)$ for some $\vartheta \in (0,1)$ by Lemma \ref{lemma-reg} (1). Moreover we have
\[\(W_{\edte} + \Psi_{\edte}\)(\sigma,0) \ge w_{\edte}(\sigma) - \|\Psi_{\edte}\|_* \(\mu \ep^{2d_0+2} \ep^{-{n-2\gamma \over 2}} + \eta_0 \ep^{\kappa-{n+2\gamma \over 2}}\nu^{-\kappa}\)
\ge C \ep^{-{n-2\gamma \over 2}}.\]
Therefore \eqref{eq-lower-bd} is obtained.
\end{proof}

\section{Energy expansion} \label{sec-exp}
This section is devoted to compute the localized energy $J_0^{\gamma}$.
We initiate it by getting a further estimation of the term $\Psi_{\ls} = \Psi_{\edt}$.
Recall that $\oh_{ab}(x) = f(|\bx|^2)H_{ab}(x)$ for $x = (\bx, x_N) \in \mr_+^N$.

\subsection{Refined estimation of the term $\Psi_{\ls}$} \label{subsec-ref-est}
Suppose that $\ep > 0$ is small and $(\dt) = (\ep^{-1}\lambda, \ep^{-1}\sigma) \in \ma$.
By applying Proposition \ref{prop-lin} with $h = 0$, one can deduce that there exists a solution $\Psi_{\ls}^A$ of
\begin{equation}\label{eq-aux}
\begin{cases}
-\text{div}_{g_c}\(x_N^{1-2\gamma}\nabla \Psi\) = - \sum\limits_{i,j=1}^n x_N^{1-2\gamma} \mu \ep^{2d_0}f\(\ep^{-2}|\bx|^2\)H_{ij}(\bx) \pa_{ij}W_{\ls} &\text{in } \mr^N_+,\\
\Psi = \psi &\text{on } \mr^n,\\
\pa^{\gamma}_{\nu} \Psi - \({n+2\gamma \over n-2\gamma}\) w_{\ls}^{4\gamma \over n-2\gamma}\psi
= \sum\limits_{m=0}^n c_mw_{\ls}^{4\gamma \over n-2\gamma}z_{\ls}^m &\text{on } \mr^n,\\
\int_{\mr^n} w_{\ls}^{4\gamma \over n-2\gamma}z_{\ls}^0 \psi d\bx = \int_{\mr^n} w_{\ls}^{4\gamma \over n-2\gamma}z_{\ls}^1 \psi d\bx = \cdots = \int_{\mr^n} w_{\ls}^{4\gamma \over n-2\gamma}z_{\ls}^n \psi d\bx = 0.
\end{cases} \end{equation}
In fact, \eqref{eq-aux} has a scaling invariance:
If we put $\Psi_{\dt}(x) = (\mu_0 \ep_0^{2d_0+2})^{-1} \ep_0^{n-2\gamma \over 2} \Psi_{\ls}(\ep x)$ for any fixed $\ep_0 > 0$ (and $\mu_0 = \ep_0^{1/3}$) small,
then it solves the equation with $\ep = 1$.
This implies that \eqref{eq-aux} admits a solution for any $\ep > 0$.

Let us introduce norms
\begin{align*}
\|U\|_*' &= \sup_{x \in \mr^N_+} \left[\chi_{\{|x - (\sigma,0)| \le \nu/2\}} \cdot \(\frac{\mu^2 \ep^{\kappa-{n+2\gamma \over 2}}} {\ep^{\kappa-2\gamma-(4d_0+4)} + |x-(\sigma,0)|^{\kappa-2\gamma-(4d_0+4)}}
+ \eta_0 {\ep^{\kappa - {n+2\gamma \over 2}} \over \nu^{\kappa}}\)^{-1} \right.\\
&\hspace{245pt} \left. + \chi_{\{|x - (\sigma,0)| \ge \nu/2\}} \cdot \frac{|x - (\sigma,0)|^{\kappa-2\gamma}}{\eta_0 \ep^{\kappa - {n+2\gamma \over 2}}}\right] \cdot |U(x)|,\\
\|U\|_{**}' &= \sup_{x \in \mr^N_+} \left[\chi_{\{|x - (\sigma,0)| \le \nu/2\}} \cdot
\(\frac{\mu^2 \ep^{\kappa-{n+2\gamma \over 2}}} {\ep^{\kappa-2\gamma-(4d_0+2)} + |x-(\sigma,0)|^{\kappa-2\gamma-(4d_0+2)}} + \eta_0 {\ep^{\kappa - {n+2\gamma \over 2}} \over \nu^{\kappa}}\)^{-1} \right. \\
&\hspace{235pt} \left. + \chi_{\{|x - (\sigma,0)| \ge \nu/2\}} \cdot \frac{|x - (\sigma,0)|^{\kappa-2\gamma+2}}{\eta_0 \ep^{\kappa - {n+2\gamma \over 2}}} \right] \cdot |U(x)|,\\
\|v\|_{***}' &= \sup_{\bx \in \mr^n} \left[\chi_{\{|\bx - \sigma| \le \nu/2\}} \cdot \(\frac{\mu^2 \ep^{\kappa-{n+2\gamma \over 2}}} {\ep^{\kappa-(4d_0+4)} + |\bx-\sigma|^{\kappa-(4d_0+4)}}
+ \eta_0 {\ep^{\kappa - {n+2\gamma \over 2}} \over \nu^{\kappa}}\)^{-1}
+ \chi_{\{|\bx - \sigma| \ge \nu/2\}} \cdot \frac{|\bx - \sigma|^{\kappa}}{\eta_0 \ep^{\kappa - {n+2\gamma \over 2}}} \right] \cdot |v(\bx)|
\end{align*}
for $U = U(\bx, x_N)$ in $\mr_+^N$ and $v = v(\bx)$ on $\mr^n$, and set a function $\Psi_{\ls}^B := \Psi_{\ls} - \Psi_{\ls}^A$.
Then it can be estimated as in the following lemma.
\begin{lemma}\label{lemma-Psi-B}
Suppose that $n > 2\gamma + 4(d_0+1) + 2/3$. Then we have
\[\left\| \Psi_{\ls}^A \right\|_* \le C, \quad \left\| \Psi_{\ls}^B \right\|_*' \le C\]
and
\[\left\| \Psi_{\ls}^A \right\|_{\drn} \le C \mu \ep^{2(d_0+1)},
\quad \left\| \Psi_{\ls}^B \right\|_{\drn} = o\(\mu \ep^{2(d_0+1)}\)\]
for some $C > 0$ independent of $\ep > 0$ and $(\dt) \in \ma$.
\end{lemma}
\begin{proof}
We find easily that
\begin{equation}\label{eq-aux-2}
\begin{cases}
-\text{div}_{\bg}\(\rho^{1-2\gamma} \nabla \Psi_{\ls}^B\)
= x_N^{1-2\gamma} E_{\ls}' &\text{in } \mr^N_+,\\
\Psi_{\ls}^B = \psi_{\ls}^B \text{ and } \Psi_{\ls} = \psi_{\ls} &\text{on } \mr^n,\\
\pa^{\gamma}_{\nu} \Psi_{\ls}^B - \({n+2\gamma \over n-2\gamma}\) w_{\ls}^{4\gamma \over n-2\gamma}\psi_{\ls}^B
= N_{\ls}(\psi_{\ls}) + \sum\limits_{m=0}^n c_mw_{\ls}^{4\gamma \over n-2\gamma}z_{\ls}^m &\text{on } \mr^n,\\
\int_{\mr^n} w_{\ls}^{4\gamma \over n-2\gamma}z_{\ls}^0 \psi_{\ls}^B d\bx = \cdots = \int_{\mr^n} w_{\ls}^{4\gamma \over n-2\gamma}z_{\ls}^n \psi_{\ls}^B d\bx = 0,
\end{cases} \end{equation}
where the nonlinear operator $N_{\ls}$ is given in \eqref{eq-nls} and
\begin{align*}
E_{\ls}' &:= \(u^{-{2 \over N-2}(1-2\gamma)} - 1\) \(\bg^{ij} - \delta^{ij}\) \pa_{ij}W_{\ls}
+ \left[\bg^{ij} - \delta^{ij} + \mu \ep^{2d_0}f\(\ep^{-2}|\bx|^2\)H_{ij} \right] \pa_{ij}W_{\ls} \\
&\ + \pa_a u^{-{2 \over N-2}(1-2\gamma)} \bg^{ab} \pa_b W_{\ls} + u^{-{2 \over N-2}(1-2\gamma)} \pa_i \bg^{ij} \pa_j W_{\ls} - x_N^{2\gamma-1} E(\rho) \(W_{\ls} + \Psi_{\ls}\).
\end{align*}
Computing similarly to the proof of Lemma \ref{lemma-error}, we obtain
\[\|E_{\ls}'\|_{**}' \le C.\]
Moreover we have
\[\|N_{\ls}(\psi_{\ls})\|_{***}' \le C\]
under the assumption $\kappa > n - 2\gamma$. 
Hence, following the argument in Step 1 of the proof of Proposition \ref{prop-lin}, we infer that
\[\left\| \Psi_{\ls}^B \right\|_*' \le C \(\|E_{\ls}'\|_{**}' + \|N_{\ls}(\psi)\|_{***}'\) \le C.\]
The second inequality is now verified. The first inequality is direct consequence of \eqref{eq-inter-1}.

In order to derive the third and fourth estimates, one can test $\Psi_{\ls}^A$ and $\Psi_{\ls}^B$ in \eqref{eq-aux} and \eqref{eq-aux-2}, respectively,
and then use their $L^{\infty}$-bounds and \eqref{eq-kappa}.
The details are omitted.
\end{proof}

\begin{lemma}\label{lemma-rem}
It holds that
\begin{equation}\label{eq-J0-exp}
J_0^{\gamma}(\dt) = I^{\gamma}(W_{\edt}) + \mu^2\ep^{4(d_0+1)} \cdot {\kappa_{\gamma} \over 2} J_1^{\gamma}(\dt) + o\(\mu^2\ep^{4(d_0+1)}\)
\end{equation}
uniformly in the admissible set $\ma = (1-\varepsilon_0, 1+\varepsilon_0) \times B^n(0,\varepsilon_0)$ where $J_1^{\gamma}(\dt)$ is a $C^2$-function defined by
\begin{equation}\label{eq-J1}
J_1^{\gamma}(\dt) := \sum\limits_{i,j=1}^n \int_{\mr^N_+} x_N^{1-2\gamma} \oh_{ij}(\bx) \pa_{ij}W_{\dt} \Psi_{\dt}^A dx.
\end{equation}
\end{lemma}
\begin{proof}
Since $(I^{\gamma})'(W_{\edt} + \Psi_{\edt})\Psi_{\edt} = 0$, we get by Lemma \ref{lemma-Psi-B}, \eqref{eq-lin-04} and \eqref{eq-inter-1} that
\[J_0^{\gamma}(\dt) = I^{\gamma}(W_{\edt}) - {\kappa_{\gamma} \over 2} \int_{\mr^N_+} \rho^{1-2\gamma} \left|\nabla \Psi_{\edt}^A\right|_{\bg}^2 dx
+ {1 \over 2} \({n+2\gamma \over n-2\gamma}\) \int_{\mr^n} w_{\edt}^{4\gamma \over n-2\gamma} \(\psi_{\edt}^A\)^2 d\bx + o\(\mu^2\ep^{4(d_0+1)}\).\]
Here $\Psi_{\edt}^A = \psi_{\edt}^A$ on $\mr^n$ and the inequality
\[\left|\(w_{\edt}+\psi_{\edt}\)_+^{2n \over n-2\gamma} - w_{\edt}^{2n \over n-2\gamma}
- \({2n \over n-2\gamma}\) \(w_{\edt}+\psi_{\edt}\)_+^{n+2\gamma \over n-2\gamma} \psi_{\edt}
+ {(n+2\gamma)n \over (n-2\gamma)^2} w_{\edt}^{4\gamma \over n-2\gamma} \psi_{\edt} \right|
\le C \left|\psi_{\edt}\right|^{2n \over n-2\gamma}\]
is applied to control the nonlinear term.
Besides, by making use of \eqref{eq-aux}, we discover
\begin{align*}
&\ \kappa_{\gamma} \int_{\mr^N_+} \rho^{1-2\gamma} \left|\nabla \Psi_{\edt}^A\right|_{\bg}^2 dx \\
&= - \kappa_{\gamma} \sum\limits_{i,j=1}^n \int_{\mr^N_+} \mu \ep^{2d_0} x_N^{1-2\gamma} f\(\ep^{-2}|\bx|^2\)H_{ij}(\bx) \pa_{ij}W_{\edt} \Psi_{\edt}^A dx
+ \({n+2\gamma \over n-2\gamma}\) \int_{\mr^n} w_{\edt}^{4\gamma \over n-2\gamma} \(\psi_{\edt}^A\)^2 d\bx.
\end{align*}
Putting these facts together, we obtain \eqref{eq-J0-exp}.

On the other hand, by \eqref{eq-Poisson}, we have that $\pa_{\delta} W_{\dt} = K_{\gamma}(\cdot, x_N) \ast \pa_{\delta} w_{\dt}$, etc.
Thus we can employ the standard difference quotient argument to verify that the first and second order derivatives of $\Psi^A_{\dt}$ with respect to $(\dt)$ are continuous.
(Check \cite[Propostion 2.13]{JLX}.)
The $C^2$-differentiability of $J_1^{\gamma}$ follows from it.
\end{proof}
\noindent The previous lemma ensures that if there exists a minimizer of the function $(\dt) \mapsto I^{\gamma}(W_{\edt}) + \mu^2\ep^{4(d_0+1)} \cdot \kappa_{\gamma} J_1^{\gamma}(\dt)/2$ in the set $\ma$,
then $J_0^{\gamma}$ also has a minimizer in $\ma$ provided that $\ep > 0$ is sufficiently small.

\subsection{Expansion of the localized energy}
We derive an expansion of the map $(\dt) \mapsto I^{\gamma}(W_{\edt})$.
\begin{prop}\label{prop-expand}
Suppose that $n > 2\gamma + 4(d_0+1)$. 
If we choose $\eta = 1$ in the statement of Proposition \ref{prop-rho-est}, then the following estimation holds.
\begin{align*}
I^{\gamma}(W_{\edt}) &= {\kappa_{\gamma} \over 2} \left[\int_{\mr_+^N} x_N^{1-2\gamma}|\nabla W_1|^2 dx +
\mu^2 \ep^{4(d_0+1)} J_2^{\gamma}(\dt) \right] - {n-2\gamma \over 2n} \int_{\mr^n} w_1^{2n \over n-2\gamma} d\bx \\
&\ + O\(\mu^3 \ep^{4(d_0+1)} + ({\ep/\nu})^{n-2\gamma}\)
\end{align*}
where 
\begin{align*}
J_2^{\gamma}(\dt) &:= {1 \over 2} \sum_{i,j,l=1}^n \int_{\mr_+^N} x_N^{1-2\gamma} \(\oh_{il}\oh_{jl}\)(\bx) \pa_i W_{\dt}\pa_j W_{\dt} dx \\
&\ + {3 \over 2} \left\{(N-2)^2 - (1-2\gamma)^2\right\} \int_{\mr_+^N} x_N^{1-2\gamma} C_2 W_{\dt}^2 dx
\stepcounter{equation}\tag{\theequation} \label{eq-wjg} \\
&\ + (1-2\gamma) \sum_{m=1}^{2d_0+2} \int_{\mr_+^N} x_N^{1+2m-2\gamma} C_{2m} |\nabla W_{\dt}|^2 dx \\
&\ + \left\{N\(\gamma-{1 \over 2}\) - 2\(\gamma^2-{1 \over 4}\)\right\} \sum_{m=1}^{2d_0+1} \int_{\mr_+^N} x_N^{1+2m-2\gamma} (2m+3) C_{2(m+1)}W_{\dt}^2 dx.
\end{align*}
The functions $C_m \in C^{\infty}(\mr^n)$ for $m = 1, \cdots, 2d_0+2$ are defined in \eqref{eq-C}.
\end{prop}
\begin{proof}
We start the proof by calculating $\int_{\mr^N_+} \rho^{1-2\gamma}|\nabla W_{\edt}|_{\bg}^2 dv_{\bg}$.
From an estimate
\begin{equation}\label{eq-hh-inv-est}
\left|\bg^{ij} - \(\delta_{ij} - h_{ij} + {1 \over 2} \sum_{l=1}^nh_{il}h_{jl}\)\right| \le C|h|^3 \quad \text{in } C^1(\{0 \le x_N \le \nu\}),
\end{equation}
Proposition \ref{prop-rho-est} and Lemma \ref{lemma-W-decay-2}, and the facts that $\pa_Nh_{ab} = 0$ in $\{0 \le x_N \le \nu\}$ and $\det\bg = 1$, we see that
\begin{align*}
&\ \int_{\mr^N_+} \rho^{1-2\gamma}\bg^{ab} \pa_a W_{\edt}\pa_b W_{\edt} dv_{\bg}\\
& = \int_{B_+^N(0,\nu)} \rho^{1-2\gamma} \left[\bg^{ij} \pa_i W_{\edt} \pa_j W_{\edt} + (\pa_N W_{\edt})^2 \right] dx + O\(\({\ep \over \nu}\)^{n-2\gamma}\)\\
&= \int_{B_+^N(0,\nu)} x_N^{1-2\gamma} \left[1 + (1-2\gamma) \mu^2\ep^{4(d_0+1)-2m} \sum_{m=1}^{2d_0+2} C_{2m}\(\ep^{-1}\bx\) x_N^{2m} \right]\\
&\quad \times \left[|\nabla W_{\edt}|^2 + \sum_{i,j=1}^n\(- h_{ij}(x) + \sum_{l=1}^n{h_{il}h_{jl}(x) \over 2}\) \pa_i W_{\edt} \pa_j W_{\edt} \right]dx
+ O\(\mu^3 \ep^{4(d_0+1)} + \({\ep \over \nu}\)^{n-2\gamma}\)\\
&= \int_{\mr_+^N} x_N^{1-2\gamma} |\nabla W_1|^2 dx
+ \sum_{i,j=1}^n\int_{B_+^N(0,\nu/\ep)} x_N^{1-2\gamma} \(- h_{ij}(\ep \bx)dx + {1 \over 2} \sum_{l=1}^nh_{il}h_{jl}(\ep \bx)\) \pa_i W_{\dt}\pa_j W_{\dt}dx \\
&\quad + (1-2\gamma) \mu^2\ep^{4(d_0+1)} \sum_{m=1}^{2d_0+2} \int_{B_+^N(0,\nu/\ep)} x_N^{1+2m-2\gamma} C_{2m}(\bx) |\nabla W_{\dt}|^2dx + O\(\mu^3 \ep^{4(d_0+1)} + \({\ep \over \nu}\)^{n-2\gamma}\).
\end{align*}
Furthermore, the algebraic properties of the tensor $W$ give 
\begin{equation}\label{eq-energy-exp-1}
\sum_{i,j=1}^n \int_{B_+^N(0,\nu/\ep)} x_N^{1-2\gamma} h_{ij}(\ep \bx)\pa_i W_{\dt}\pa_j W_{\dt}dx = O\(\mu\nu^{2(d_0+1)} \cdot \({\ep \over \nu}\)^{n-2\gamma}\)
\end{equation}
whose proof is deferred to the end of the proof, whereas we immediately obtain from the definition of the tensor $\oh_{ij}$ that
\begin{multline*}
\int_{B_+^N(0,\nu/\ep)} x_N^{1-2\gamma} \(h_{il}h_{jl}\)(\ep \bx) \pa_i W_{\dt}\pa_j W_{\dt}dx \\
= \mu^2 \ep^{4(d_0+1)} \int_{\mr_+^N} x_N^{1-2\gamma} \(\oh_{il}\oh_{jl}\)(\bx) \pa_i W_{\dt}\pa_j W_{\dt}dx + O\(\mu^2\nu^{4(d_0+1)} \cdot \({\ep \over \nu}\)^{n-2\gamma}\).
\end{multline*}
This completes the estimation on the gradient part of the energy $I^{\gamma}$.

Next, we compute $\int_{\mr^N_+} E(\rho)W_{\edt}^2 dx$.
Let
\begin{equation}\label{eq-E-tilde}
\begin{aligned}
\widetilde{E}(\rho) &= - \Delta_{\bg}\(\rho^{1-2\gamma \over 2}\)\rho^{1-2\gamma \over 2} + \(\gamma^2-{1 \over 4}\) \rho^{-1-2\gamma} \\
&= \(\gamma-{1 \over 2}\) \rho^{-2\gamma} \left[ \pa_i\bg^{ij} \pa_j\rho + \bg^{ij} \pa_{ij}\rho + \pa_{NN}\rho\right]
- \(\gamma^2 - {1 \over 4}\) \rho^{-1-2\gamma}
\left[\bg^{ij} \pa_i\rho \pa_j\rho + \((\pa_N \rho)^2-1\)\right].
\end{aligned}
\end{equation}
Then putting Proposition \ref{prop-rho-est} and \eqref{eq-hh-inv-est} together leads us to deduce
\begin{align*}
&\ \widetilde{E}(\rho)(\ep x) \\
&= \(\gamma-{1 \over 2}\) (\ep x_N)^{-2\gamma} \left[\mu^4 \ep^{8d_0+7} \sum_{i,j,l=1}^n \sum_{m=1}^{2d_0+2} \oh_{il} \pa_i\oh_{jl} \pa_jC_{2m} x_N^{2m+1}
+ \mu^2 \ep^{4d_0+3} \sum_{i,j=1}^n \sum_{m=1}^{2d_0+2} \delta_{ij} \pa_{ij}C_{2m}x_N^{2m+1} \right.\\
&\hspace{85pt} \left. + \mu^2 \ep^{4d_0+3} \sum_{m=0}^{2d_0+1} 2(m+1)(2m+3) C_{2(m+1)}x_N^{2m+1} \right]
\stepcounter{equation}\tag{\theequation} \label{eq-E-tilde-2} \\
&\ - \(\gamma^2 - {1 \over 4}\) (\ep x_N)^{-1-2\gamma}
\left[ \mu^4\ep^{8(d_0+1)} \sum_{i,j=1}^n \sum_{m=1}^{2d_0+2} \delta_{ij} \pa_iC_{2m}\pa_jC_{2m}x_N^{2(m+1)} \right.\\
&\hspace{55pt} \left. + \mu^2\ep^{4(d_0+1)} \sum_{m=0}^{2d_0+1} 2(2m+3) C_{2(m+1)}x_N^{2(m+1)} \right] + O\(\mu \cdot \mu^2 \ep^{4d_0+3-2\gamma} x_N^{1-2\gamma}|x|^2\(1+|x|^{4d_0}\)\)
\end{align*}
in $B^N_+(0,\nu/\ep)$. Therefore, recalling the definition \eqref{eq-error} of $E(\rho)$ and the expansion \eqref{eq-R-bg-est} (or \eqref{eq-R-bg-est-2}) of the scalar curvature $R_{\bg}$, we find that
\begin{align*}
&\ \int_{\mr^N_+} E(\rho)W_{\edt}^2 dv_{\bg}\\
&= \int_{\mr^N_+} \(\widetilde{E}(\rho) + {N-2 \over 4(N-1)}R_{\bg}\rho^{1-2\gamma} \) W_{\edt}^2 dx\\
&= \mu^2\ep^{4(d_0+1)} \int_{B^N_+(0,\nu/\ep)} x_N^{1-2\gamma} \left[ \(\gamma-{1 \over 2}\) \(\sum_{m=1}^{2d_0+2} \Delta C_{2m}x_N^{2m} + \sum_{m=0}^{2d_0+1} 2(m+1)(2m+3) C_{2(m+1)}x_N^{2m} \) \right.\\
&\hspace{120pt} \left. - \(\gamma^2 - {1 \over 4}\) \sum_{m=0}^{2d_0+1} 2(2m+3) C_{2(m+1)}x_N^{2m} + {3(N-2)^2 \over 2} C_2 \right] W_{\dt}^2 dx \\
&\ + O\(\mu^3\ep^{4(d_0+1)} + (\ep/\nu)^{n-2\gamma}\).
\end{align*}
Here we utilized the fact that $|R_{\bg}|$ and $x_N^{-1+2\gamma}|\widetilde{E}(\rho)|$ are bounded in $\mr^N_+$ (refer to \eqref{eq-lin-04}) to compute the remainder term.
We further note that $\Delta C_{2m} = (2m+3)(N-2(m+1))C_{2(m+1)}$ is satisfied for all $m = 1, \cdots, 2d_0+2$ by \eqref{eq-C}.

Finally, it holds that $\int_{\mr^n} w_{\edt}^{2n \over n-2\gamma} dv_{\hh} = \int_{\mr^n} w_1^{2n \over n-2\gamma} d\bx$ by scaling invariance and the observation that $\det{\hh} = 1$.
Thus collecting all the computations made here, we can conclude the proof.

\medskip
\noindent \textbf{Derivation of \eqref{eq-energy-exp-1}.}
Unlike the local cases where pointwise relations of the bubbles were used (see \cite[Proposition 13]{Br} for $\gamma =1$ or \cite[Proposition 3.2]{Al} for $\gamma = 1/2$),
our proof heavily relies on the algebraic properties of the tensor $W_{ijkl}$ instead.
Write $\tau = (\tau^1, \cdots, \tau^n) \in \mr^n$ and \[f\(|\bx+\tau|^2\) = \sum_{m=0}^{d_0} a_m \(|\bx|^2+|\tau|^2+2\bx \cdot \tau\)^m
= \sum_{m=0}^{d_0}\sum_{t=0}^m a_mb_{m,t} \(|\bx|^2+|\tau|^2\)^{m-t}(\bx \cdot \tau)^{t}\]
where $a_m,\ b_{m,t} \in \mr$. By the definition of $h_{ij}$ in \eqref{eq-h}, then we have
\begin{align*}
&\ \sum_{i,j=1}^n\int_{B_+^N(0,\nu/\ep)} x_N^{1-2\gamma} h_{ij}(\ep \bx)\pa_i W_{\dt}\pa_j W_{\dt} dx \\
&= \mu \ep^{2d_0} \sum_{i,j=1}^n \int_{\mr_+^N} x_N^{1-2\gamma} H_{ij}(\ep \bx) f\(|\bx|^2\) \pa_i W_{\dt}\pa_j W_{\dt} dx + O\(\mu\nu^{2(d_0+1)} \cdot \({\ep \over \nu}\)^{n-2\gamma}\) \\
&= \mu \ep^{2(d_0+1)} \int_{\mr_+^N} x_N^{1-2\gamma} W_{ikjl} \(x^k+\tau^k\)\(x^l+\tau^l\)
f\(|\bx+\tau|^2\) x^ix^j|\bx|^{-2} (\pa_rW_{\delta} \pa_r W_{\delta})(|\bx|,x_N) dx \\
&\quad + O\(\mu\nu^{2(d_0+1)} \cdot (\ep/\nu)^{n-2\gamma}\) \stepcounter{equation}\tag{\theequation}\label{eq-energy-exp-a} \\
&= \sum_{m=0}^{d_0}\sum_{t=0}^m a_mb_{m,t} \int_0^{\infty} x_N^{1-2\gamma} \int_0^{\infty} \(|\bx|^2+|\tau|^2\)^{m-t} |\bx|^{-2} (\pa_rW_{\delta} \pa_r W_{\delta})(|\bx|,x_N) \\
&\hspace{30pt} \times W_{ikjl} \int_{S^{n-1}(0,r)} x^ix^j\(x^k+\tau^k\)\(x^l+\tau^l\)(\bx \cdot \tau)^t dS_rdrdx_N + O\(\mu\nu^{2(d_0+1)} \cdot \({\ep \over \nu}\)^{n-2\gamma}\).
\end{align*}
However, since $W_{iijk} = 0$ and $\sum_{i=1}^nW_{ijik} = 0$ hold, we have the validity of
\begin{equation}\label{eq-energy-exp-b}
\begin{aligned}
W_{ikjl} \int_{S^{n-1}(0,1)} x^ix^j\(x^k+\tau^k\)\(x^l+\tau^l\)\(x^1\tau^1 + \cdots + x^n\tau^n\)^t dS = 0
\end{aligned}
\end{equation}
for any given $\tau \in \mr^n$ and $t = 0, 1, \cdots, d_0$ (where $d_0$ will be chosen to be $d_0 \le 4$).
Plugging \eqref{eq-energy-exp-b} into \eqref{eq-energy-exp-a}, we obtain \eqref{eq-energy-exp-1}.
\end{proof}
\noindent The previous proposition implies that searching a critical point of the function $J_0^{\gamma}$ can be reduced to looking for that of $J_1^{\gamma} + J_2^{\gamma}$.
Since the tensor $H_{ab}$ in \eqref{eq-h0} satisfies the symmetric condition $H_{ab}(\bx, x_N) = H_{ab}(-\bx, x_N)$ for all $(\bx, x_N) \in \mr^N_+$, so does the tensor $\oh$.
Also, it is a simple task to check that $\Psi^A_{\delta, -\tau}(\bx, x_N) = \Psi^A_{\dt}(-\bx, x_N)$. 
Therefore $(J_1^{\gamma} + J_2^{\gamma})(\dt) = (J_1^{\gamma} + J_2^{\gamma})(\delta, -\tau)$ for any $(\dt) \in \ma$.
As an immediate consequence, we have
\begin{equation}\label{eq-J-diff}
{\pa \(J_1^{\gamma} + J_2^{\gamma}\) \over \pa \tau} (\dz) = {\pa^2 \(J_1^{\gamma} + J_2^{\gamma}\) \over \pa \tau \pa \delta} (\dz) = 0 \quad \text{for every } \delta > 0.
\end{equation}
In the next subsection, we carry out some computations necessary to find a critical point (specifically, a local minimizer) of $J_1^{\gamma} + J_2^{\gamma}$.
Actually, with the aid of these computations, we are able to deduce that $(J_1^{\gamma} + J_2^{\gamma})(\dz)$ can be expressed with a polynomial $P = P(\delta)$ (see Subsection \ref{subsec-exp-3}).
As a result, our problem is translated into obtaining a suitable critical point of the polynomial $P$ which we shall take care of in Section \ref{sec-conc}.
It will turn out that for sufficiently large dimensions (for instance $n \ge 52$ if $\gamma = 1/2$),
an appropriate choice of a linear function $f$ in the definition of the metric $\bg$ (see \eqref{eq-h}) gives a desirable critical point of the polynomial $P$.
However, it is inevitable to introduce a polynomial $f$ of degree $d_0 = 4$ in the metric $\bg$ instead
so as to enable to find a necessary critical point of $P$ in lower dimensions (e.g. $24 \le n \le 51$ for $\gamma = 1/2$).
Since the computation is extremely complicated in the case that $d_0 = 4$,
we will take into account only when the dimension $n$ is large enough (so that $d_0 = 1$) in most part of the paper to clarify the exposition.
Changes required to consider lower dimensions will be described in Subsection \ref{subsec-low}.

\subsection{Preparation for an expansion of $J_1^{\gamma} + J_2^{\gamma}$}\label{subsec-exp-2}
Let us introduce some functions.

\medskip \noindent
- Denote the Bessel function of the first kind and the modified Bessel function of the second kind of order $\gamma$ by $J_{\gamma}$ and $K_{\gamma}$, respectively.
Their definitions and properties can be found in \cite{AS}.

\medskip \noindent - Set $\varphi$ by the solution of the ordinary differential equation in the variable $t > 0$:
\begin{equation}\label{eq-ode-1}
\phi''(t) + {1-2\gamma \over t} \phi'(t) - \phi(t) = 0,\quad \phi(0) = 1 \quad \text{and } \phi(\infty) = 0.
\end{equation}
In particular, $\varphi(t) = d_1 t^{\gamma}K_{\gamma}(t)$ for $t > 0$ where $d_1 = 2^{1-\gamma}/\Gamma(\gamma)$.

\medskip \noindent - Notice that the Fourier transform of $w_1$ is a radially symmetric function.
We shall denote by $\hw_1(\xi) = \hw_1(\rho)$ with a slight abuse of the notation.

\medskip \noindent - Let $A_{\alpha}$ and $B_{\alpha}$ be numbers defined to be
\[A_{\alpha} = \int_0^{\infty} t^{\alpha-2\gamma}\varphi^2(t)dt \quad \text{and} \quad B_{\alpha} = \int_0^{\infty} \rho^{-\alpha+2\gamma}\hw_1^2(\rho)\rho^{n-1}d\rho\]
for ${\alpha} \in \mn \cup \{0\}$. Also, we set functions
\begin{equation}\label{eq-F}
\begin{aligned}
F_{1,n,\gamma}(\alpha, \beta) &= \int_0^{\infty} \int_{\mr^n} x_N^{\alpha-2\gamma} |\bx|^{\beta} W_1^2(\bx, x_N) d\bx dx_N,\\
F_{2,n,\gamma}(\alpha, \beta) &= \int_0^{\infty} \int_{\mr^n} x_N^{\alpha-2\gamma} |\bx|^{\beta} |\nabla_{\bx} W_1|^2(\bx, x_N) d\bx dx_N,\\
F_{3,n,\gamma}(\alpha, \beta) &= \int_0^{\infty} \int_{\mr^n} x_N^{\alpha-2\gamma} |\bx|^{\beta} |\pa_{x_N} W_1|^2(\bx, x_N) d\bx dx_N
\end{aligned}
\end{equation}
for $\alpha \in 2\mn+1$ and $\beta \in 2\mn$ as far as they are finite.

\medskip The main objective of this subsection is to depict how to express the values of functions in \eqref{eq-F} in terms of numbers $A_1$ and $B_2$.
Especially, the following lemma will be established as one of the consequences.
See also Appendix \ref{sec-app-b} below and \cite{KMW}.

\begin{lemma}\label{lemma-W-rel}
Suppose that $d_0 = 1$ and $n > 2\gamma + 8$. Then it holds
\begin{align}
F_{1,n,\gamma}(1,2) &= \left|S^{n-1}\right| \left[{n \(3(n-3)^2+\(1-4\gamma^2\)\) \over 3(n-4)(n-2\gamma-4)(n+2\gamma-4)}\right]A_1B_2, \label{eq-W-rel-0}\\
F_{1,n,\gamma}(1,4) &= \left|S^{n-1}\right| \left[{n(n+2) \(15(n-3)^2(n-5)^2+R_{1,n,\gamma}(1,4)\)
\over 15(n-4)(n-6) (n-2\gamma-4)(n+2\gamma-4) (n-2\gamma-6)(n+2\gamma-6)}\right] A_1B_2, \label{eq-W-rel-01}\\
F_{1,n,\gamma}(1,6) &= \left|S^{n-1}\right| \left[ \tfrac{n(n+2)(n+4) \(35(n-3)^2(n-5)^2(n-7)^2+R_{1,n,\gamma}(1,6)\)}
{35(n-4)(n-6)(n-8) (n-2\gamma-4)(n+2\gamma-4) (n-2\gamma-6)(n+2\gamma-6) (n-2\gamma-8)(n+2\gamma-8)}\right] A_1B_2. \label{eq-W-rel-02}
\end{align}
Moreover we have
\begin{align*}
F_{2,n,\gamma}(3,2) &= \left|S^{n-1}\right| \left[{\(1-\gamma^2\)2(n+2)\(5(n-1)(n-3) + \(1-4\gamma^2\)\) \over 15(n-4)(n-2\gamma-4)(n+2\gamma-4)}\right] A_1B_2,\\
F_{3,n,\gamma}(3,2) &= \left|S^{n-1}\right| \left[{2(1-\gamma)(2-\gamma)\(5(n-1)(n-2)(n-3) - R_{3,n,\gamma}(3,2)\) \over 15(n-4)(n-2\gamma-4)(n+2\gamma-4)}\right] A_1B_2
\end{align*}
where
\begin{align*}
R_{1,n,\gamma}(1,4) &= \(1-4\gamma^2\)\left[10n^2 - 80n + 177 - 12\gamma^2\right],\\
R_{1,n,\gamma}(1,6) &= \(1-4\gamma^2\)\left[35n^4 - 700n^3 + 5299n^2 - 17990n + 23469 \right. \\
&\hspace{150pt} \left. + 80\gamma^4 - 4\gamma^2\(21n^2-210n+611\)\right],\\
R_{3,n,\gamma}(3,2) &= (1-2\gamma) \left[3n - 14 - 2\gamma(n+2)\right].
\end{align*}
\end{lemma}
\noindent In order to verify the lemma, the following observation of Gonz\'alez \cite[Lemma 14]{Go} and Gonz\'alez-Qing \cite[Section 7]{GQ} is needed.
\begin{lemma}\label{lemma-ode-1}
For each $x_N > 0$ fixed, let $\whw_1(\xi,x_N)$ be the Fourier transform of $W_1(\bx,x_N)$ with respect to the variable $\bx \in \mr^n$.
Then we have that
\begin{equation}\label{eq-W-FT}
\whw_1(\xi,x_N) = \hw_1(\xi)\varphi(|\xi|x_N) \quad \text{for all } \xi \in \mr^n \text{ and } x_N > 0,
\end{equation}
where $\varphi(t) = d_1 t^{\gamma}K_{\gamma}(t)$ is the solution to \eqref{eq-ode-1}.
\end{lemma}

\noindent Next, we obtain the explicit form of the Fourier transform $\hw_1(\rho) = \hw_1(\xi)$ of the standard bubble $w_1(x)$.
\eqref{eq-hatw} is also obtained in \cite{GW} up to the constant multiple.
\begin{lemma}\label{lemma-ode-2}
If $n > 4\gamma - 1$, then it is true that
\begin{equation}\label{eq-hatw}
\hw_1(\rho) = d_2 \rho^{-\gamma}K_{\gamma}(\rho) \quad \text{for any } \rho > 0 \quad \text{where } d_2 = {2\Gamma\({n+2\gamma \over 2}\)^{n-2\gamma \over 4\gamma} \over \Gamma\({n-2\gamma \over 2}\)^{n+2\gamma \over 4\gamma}}.
\end{equation}
As a result, it is a solution of the equation in $t > 0$:
\begin{equation}\label{eq-ode-2}
\phi''(t) + {1+2\gamma \over t} \phi'(t) - \phi(t) = 0
\end{equation}
with the asymptotic behavior
\[\phi(t) = \begin{cases}
d_2 \Gamma(\gamma) 2^{\gamma-1} t^{-2\gamma} (1+o(1)) &\text{as } t \to 0+,\\
d_2 \sqrt{\pi/2} t^{-\gamma-{1 \over 2}}e^{-\gamma} (1+o(1)) &\text{as } t \to \infty.
\end{cases}\] 
\end{lemma}
\begin{proof}
Since $w$ is radial, namely, $w_1(\bx) = w_1(r)$, so is its Fourier transform and can be expressed in terms of the Bessel function $J_{n-2 \over 2}$:
\[\hw_1(\xi) = \hw_1(\rho) = {1 \over \rho^{n-2 \over 2}} \int_0^{\infty} w_1(r)J_{n-2 \over 2}(\rho r)r^{n \over 2}dr \quad \text{for } \rho = |\xi| \ge 0.\]
We observe that the integral is the $\({n-2 \over 2}\)$-th order Hankel transform of the function $r \mapsto r^{n-2 \over 2}w_1(r)$,
whose precise value can be computed under the assumption on the dimension $n > 4\gamma - 1$ as listed in \cite{Pi}.
As a result, we find from \eqref{eq-cng} that \eqref{eq-hatw} has the validity.
The fact that $\hw_1$ solves Eq. \eqref{eq-ode-2} results from a direct computation.
\end{proof}

\noindent The following lemma can be obtained by modifying the proof of \cite[Lemma 7.2]{GQ}.
\begin{lemma}\label{lemma-ode-3}
Suppose that $\phi = \varphi$ or $\hw_1$ in Lemmas \ref{lemma-ode-1} and \ref{lemma-ode-2}.
Set also $\alpha = 1-2\gamma$ if $\phi = \varphi$, or $\alpha = 1 + 2\gamma$ if $\phi = \hw_1$. Then we have
\begin{align}
\int_0^{\infty} \(\phi'(\rho)\)^2\rho^{\eta} d\rho &= \({\eta+1 \over 2}\)\({\eta+1 \over 2}-\alpha\)^{-1} \int_0^{\infty} \phi(\rho)^2 \rho^{\eta} d\rho \label{eq-ode-5}
\intertext{and}
\int_0^{\infty} \phi(\rho)^2 \rho^{\eta} d\rho &= (\eta-\alpha)\({\eta-1 \over 2}\)
\left[1 + \({\eta+1 \over 2}\)\({\eta+1 \over 2}-\alpha\)^{-1}\right]^{-1} \int_0^{\infty} \phi(\rho)^2 \rho^{\eta-2} d\rho \label{eq-ode-6}
\end{align}
provided that $\eta > 1$ for $\phi = \varphi$, and $\eta > 4\gamma + 1$ for $\phi = \hw_1$.
\end{lemma}
\begin{proof}
We only take into account the case that $\phi = \hw_1$ since the other case can be covered in the same way.
If we multiply $\rho^{\eta}\hw_1'(\rho)$ on the both sides of \eqref{eq-ode-2} and then integrate the results over $(0,\infty)$, we get
\begin{equation}\label{eq-ode-4}
\(\alpha - {\eta+1 \over 2}\) \int_0^{\infty} \(\hw_1'(\rho)\)^2\rho^{\eta} d\rho = -\({\eta+1 \over 2}\) \int_0^{\infty} \(\hw_1(\rho)\)^2 \rho^{\eta} d\rho,
\end{equation}
which is \eqref{eq-ode-5}.
Since it is known that $K_{\gamma}'(\rho)$ is of order $\rho^{-\gamma-1}$ near 0,
it holds that $\rho^{\eta}\hw_1'(\rho)^2|_{\rho = 0} = 0$ if $\eta > 4\gamma + 1$, which validates the above calculation.

On the other hand, if we test $\rho^{\eta+2}\hw_1(\rho)$ on \eqref{eq-ode-2} instead, we then discover that
\begin{equation}\label{eq-ode-7}
-\int_0^{\infty} \(\hw_1'(\rho)\)^2\rho^{\eta} d\rho + (\alpha-\eta) \int_0^{\infty} \hw_1'(\rho)\hw_1(\rho)\rho^{\eta-1} d\rho
= \int_0^{\infty} \hw_1(\rho)^2\rho^{\eta} d\rho.
\end{equation}
Since an application of integration by parts shows that
\[\int_0^{\infty} \hw_1'(\rho)\hw_1(\rho)\rho^{\eta-1} d\rho = -\({\eta-1 \over 2}\)\int_0^{\infty} \hw_1(\rho)^2\rho^{\eta-2}d\rho,\]
we conclude with \eqref{eq-ode-4} and \eqref{eq-ode-7} that \eqref{eq-ode-6} holds.
\end{proof}

\noindent With the previous lemmas, it is now possible to proceed the proof of Lemma \ref{lemma-W-rel}.
\begin{proof}[Proof of Lemma \ref{lemma-W-rel}]
We remark that the basic idea of this proof is motivated from \cite[Lemma 7.3]{GQ}.

\medskip
We first deal with $F_{1,n,\gamma}$.
By taking the Fourier transform on the variable $\bx$ and applying \eqref{eq-W-FT}, one derives
\begin{equation}\label{eq-W-rel-1}
\begin{aligned}
\int_{\mr^n} |\bx|^{2(1+m)}W_1^2(\bx, x_N) d\bx &= \left\||\cdot|^{1+m} W_1(\cdot, x_N)\right\|_{L^2(\mr^n)}^2
= \left\|\((-\Delta)^{1+m \over 2} \whw_1\)(\cdot,x_N)\right\|_{L^2(\mr^n)}^2 \\
&= \int_{\mr^n} \hw_1(|\xi|)\varphi(|\xi|x_N) \cdot (-\Delta)_{\xi}^{1+m}\(\hw_1(|\xi|)\varphi(|\xi|x_N)\) d\xi
\end{aligned} \end{equation}
for each fixed $m \in \mn \cup \{0\}$ and $x_N \in (0,\infty)$.
Assume $m = 0$ first. By \eqref{eq-ode-1} and \eqref{eq-ode-2}, we find
\begin{multline}\label{eq-W-rel-2}
\Delta_{\xi}\(\hw_1(|\xi|)\varphi(|\xi|x_N)\) = \hw_1(\rho) \varphi(\rho x_N)\(1+x_N^2\) + 2\hw_1'(\rho)\varphi'(\rho x_N)x_N \\
+ \({n-2\gamma-2 \over \rho}\) \hw_1'(\rho)\varphi(\rho x_N) + \({n+2\gamma-2 \over \rho}\) \hw_1(\rho) \varphi'(\rho x_N)x_N
\end{multline}
where $\Delta_{\xi}$ stands for the Laplacian with respect to the $\xi$-variable.
Moreover the substitution $t = \rho x_N$ enables us to get
\[\int_0^{\infty} \int_0^{\infty} x_N^{1-2\gamma} \(\hw_1(\rho)\varphi(\rho x_N)\)^2\left(1+x_N^2\) \rho^{n-1} dx_N d\rho = A_1B_2 + A_3B_4\]
and
\begin{align*}
&\ \int_0^{\infty} \int_0^{\infty} x_N^{2-2\gamma} \hw_1(\rho)\varphi(\rho x_N)\hw_1'(\rho)\varphi'(\rho x_N) \rho^{n-1} dx_Nd\rho \\
&= \(\int_0^{\infty} t^{2-2\gamma}\varphi(t)\varphi'(t)dt\)\(\int_0^{\infty} \rho^{-3+2\gamma} \hw_1(\rho)\hw_1'(\rho)\rho^{n-1}d\rho\)
= \left[{2(1-\gamma)(n-3) \over (n-4)(n-2\gamma-4)}\right] A_1B_2,
\end{align*}
which are finite for $n > 2\gamma + 4$.
Therefore, treating the other two terms of the right-hand side of \eqref{eq-W-rel-2} in this fashion, we deduce from \eqref{eq-W-rel-1} and Lemma \ref{lemma-ode-3} that
\begin{align*}
\int_0^{\infty} x_N^{1-2\gamma} \int_{\mr^n} |\bx|^2W_1^2(\bx, x_N) d\bx dx_N
&= \int_0^{\infty} x_N^{1-2\gamma} \int_{\mr^N} \hw_1(|\xi|)\varphi(|\xi|x_N) \cdot (-\Delta)_{\xi}\(\hw_1(|\xi|)\varphi(|\xi|x_N)\) d\xi dx_N \\
&= \left|S^{n-1}\right| \left[{n \(3(n-3)^2+\(1-4\gamma^2\)\) \over 12(n-4)(n-2\gamma-4)(n+2\gamma-4)}\right]A_1B_2,
\end{align*}
getting \eqref{eq-W-rel-0}.
Similar technique also can be applied for $m = 1$ and 2, which gives \eqref{eq-W-rel-01} and \eqref{eq-W-rel-02}.
To derive \eqref{eq-W-rel-02} for instance, we first observe that
\begin{multline}\label{eq-W-rel-5}
\int_0^{\infty} x_N^{1-2\gamma} \int_{\mr^n} |\bx|^6W_1^2(\bx, x_N) d\bx dx_N \\
= \int_0^{\infty} x_N^{1-2\gamma} \int_{\mr^N} (-\Delta)_{\xi}\(\hw_1(|\xi|)\varphi(|\xi|x_N)\) \cdot \Delta^2_{\xi}\(\hw_1(|\xi|)\varphi(|\xi|x_N)\) d\xi dx_N.
\end{multline}
Furthermore, one can check that
\begin{align*}
&\ \Delta^2_{\xi}\(\hw_1(|\xi|)\varphi(|\xi|x_N)\) \\
& = \hw_1(\rho) \varphi(\rho x_N)\(1+6x_N^2+x_N^4\)
+ \left[{n^2 - 2n(3+2\gamma) + 4\(2+3\gamma+\gamma^2\) \over \rho^2}\right] \hw_1(\rho) \varphi(\rho x_N) \\
& + \left[{n^2 - 2n(3-2\gamma) + 4\(2-3\gamma+\gamma^2\) \over \rho^2}\right] \hw_1(\rho) \varphi(\rho x_N) x_N^2 \\
& + 4 \hw_1'(\rho)\varphi'(\rho x_N) x_N \(1+x_N^2\) + \left[{n^2-10n+4\(5+\gamma^2\) \over \rho^2}\right] 2 \hw_1'(\rho)\varphi'(\rho x_N) x_N \\
& + \({n-2\gamma-2 \over \rho}\) 2 \hw_1'(\rho)\varphi(\rho x_N) + \({3n-2\gamma-8 \over \rho}\) 2 \hw_1'(\rho)\varphi(\rho x_N)x_N^2 \\
& - \left[{(1+\gamma)\(n^2 - 2n(3+2\gamma) + 4\(2+3\gamma+\gamma^2\)\) \over \rho^3}\right] 2 \hw_1'(\rho)\varphi(\rho x_N) \\
& + \({n+2\gamma-2 \over \rho}\) 2 \hw_1(\rho) \varphi'(\rho x_N)x_N^3 + \({3n+2\gamma-8 \over \rho}\) 2 \hw_1(\rho) \varphi'(\rho x_N)x_N \\
& + \left[{(-1+\gamma)\(n^2 - 2n(3-2\gamma) + 4\(2-3\gamma+\gamma^2\)\) \over \rho^3}\right] 2 \hw_1(\rho)\varphi'(\rho x_N)x_N.
\end{align*}
Putting this into \eqref{eq-W-rel-5} and computing term-by-term as before, we can determine \eqref{eq-W-rel-02}.

We next turn to the analysis of $F_{2,n,\gamma}$ and $F_{3,n,\gamma}$.
As in \eqref{eq-W-rel-1}, one has
\begin{equation}\label{eq-W-rel-8}
\begin{aligned}
\int_{\mr^n}|\bx|^{2(1+m)} |\nabla_{\bx} W_1(\bx,x_N)|^2 d\bx
&
= \sum_{i=1}^n \left\|(-\Delta)^{1+m \over 2}\widehat{\pa_iW_1}(\cdot,x_N)\right\|_{L^2(\mr^n)}^2 \\
&= \sum_{i=1}^n \int_{\mr^n} \xi_i \whw_1(|\xi|,x_N) \cdot (-\Delta)_{\xi}^{1+m}\(\xi_i\whw_1(|\xi|,x_N)\) d\xi
\end{aligned} \end{equation}
for any $m \in \mn \cup \{0\}$.
Therefore it is possible to perform a computation using \eqref{eq-W-rel-8} and the relation
\[(-\Delta_{\xi})^m \(\xi_i\whw_1\) = -2m \pa_i(-\Delta_{\xi})^{m-1}\whw_1 + \xi_i (-\Delta_{\xi})^m \whw_1\]
to show that
\begin{align*}
&\ - \left|S^{n-1}\right|^{-1} F_{2,n,\gamma}(3,2)\\
&= \int_0^{\infty}\int_0^{\infty} x_N^{3-2\gamma} \left[2\rho^n \whw_1(\rho,x_N) \whw_1'(\rho,x_N)
+ \rho^{n+1} \whw_1(\rho,x_N) \(\whw_1''(\rho,x_N) + {n-1 \over \rho}\whw_1'(\rho,x_N)\)\right] d\rho dx_N \\
&= - \left[{(1-\gamma^2)2(n+2)\(5(n-1)(n-3) + \(1-4\gamma^2\)\) \over 15(n-4)(n-2\gamma-4)(n+2\gamma-4)}\right] A_1B_2
\end{align*}
($\whw_1'$ signifies the derivative of $\whw_1$ in the radial variable $\rho = |\bx|$).
Likewise, one sees that
\begin{equation}\label{eq-W-rel-9}
\begin{aligned}
\int_{\mr^n}|\bx|^{2(1+m)} |\pa_N W_1(\bx,x_N)|^2 d\bx &= \left\|(-\Delta)^{1+m \over 2}\widehat{\pa_NW_1}(\cdot,x_N)\right\|_{L^2(\mr^n)}^2
= \left\|(-\Delta)^{1+m \over 2}\pa_N\whw_1(\cdot,x_N)\right\|_{L^2(\mr^n)}^2 \\
&= \int_{\mr^n} |\xi|\hw_1(\xi)\varphi'(|\xi|x_N) \cdot (-\Delta)_{\xi}^{1+m}\(|\xi|\hw_1(\xi)\varphi'(|\xi|x_N)\) d\xi.
\end{aligned} \end{equation}
Thus by employing \eqref{eq-W-rel-9} and
\begin{multline*}
\Delta_{\xi}\(|\xi|\hw_1(\xi)\varphi'(|\xi|x_N)\) = (n+2\gamma) \hw_1(\rho) \varphi(\rho x_N)x_N + (n+2\gamma-2) \hw_1'(\rho)\varphi'(\rho x_N) \\
+ 2\rho\hw_1'(\rho)\varphi(\rho x_N)x_N + 2\gamma \({n+2\gamma-2 \over \rho}\) \hw_1(\rho) \varphi'(\rho x_N) + \rho \hw_1(\rho) \varphi'(\rho x_N) \(1 + x_N^2\),
\end{multline*}
we can find the value of $F_{3,n,\gamma}(3,2)$.

This completes the proof.
\end{proof}

\subsection{Reduction of $\(J_1^{\gamma} + J_2^{\gamma}\)(\cdot,0)$ into a polynomial $P$}\label{subsec-exp-3}
Lemma \ref{lemma-W-rel} allows us to obtain the following proposition.
\begin{prop}\label{prop-poly}
Assume that the degree of the polynomial $f$ in \eqref{eq-h} is $d_0 = 1$ so that it is written as $f(r) = a_0 + a_1r$ where $a_0$ and $a_1$ are arbitrarily chosen and fixed.
Also, we denote
\[F_{4,n,\gamma}(\alpha, \beta) = F_{2,n,\gamma}(\alpha, \beta) + F_{3,n,\gamma}(\alpha, \beta) \quad \text{for } (\alpha, \beta) \in (2\mn+1) \times (2\mn)\]
and set polynomials
\begin{equation}\label{eq-poly-2}
\begin{aligned}
P_1(t) &= {1 \over n(n+2)} \left[a_1^2(n+8) F_1(1,6) t^4 + 2a_0a_1(n+4) F_1(1,4) t^3 + a_0^2(n+2) F_1(1,2) t^2 \right],\\
P_{31}(t) &= {1 \over n(n+2)} \left[ 6a_1^2(n+4)(n+8) F_1(3,4) t^4 + 8a_0a_1(n+2)(n+4) F_1(3,2) t^3 \right.\\
&\hspace{50pt} \left. + 2a_0^2n(n+2) F_1(3,0) t^2 \right],\\
P_{32}(t) &= {1 \over n} \left[ 24a_1^2(n+4)(n+8) F_1(5,2) t^4 + 16a_0a_1(n+4) F_1(5,0) t^3 \right],\\
P_{33}(t) &= 48a_1^2(n+4)(n+8) F_1(7,0) t^4
\end{aligned}
\end{equation}
where $F_1 = F_{1,n,\gamma}$ (refer to \eqref{eq-F}).
We further define polynomials $P_{21}$, $P_{22}$, $P_{23}$ and $P_{24}$ by substituting each $F_{1,n,\gamma}(\alpha,\beta)$ appearing in $P_1$, $P_{31}$, $P_{32}$ and $P_{33}$ by $F_{4,n,\gamma}(\alpha+2,\beta)$, respectively. If the polynomial $P$ is given by
\begin{equation}\label{eq-poly-3}
\begin{aligned}
&\ -24n(n-1) \cdot P(t)\\
&= {3 \over 2} \((n-1)^2-(1-2\gamma)^2\) P_1(t)\\
&\quad + (1-2\gamma) \sum_{m=1}^{2d_0+2} \left[\prod_{\tm=1}^{m-1} {1 \over (2\tm+3)(N-2(\tm+1))}\right] P_{2m}(t) \\
&\quad + \left\{(n+1)\(\gamma-{1 \over 2}\) - 2\(\gamma^2-{1 \over 4}\) \right\} \sum_{m=1}^{2d_0+1} \left[\prod_{\tm=1}^m {1 \over (2\tm+3)(N-2(\tm+1))}\right] (2m+3)P_{3m}(t)
\end{aligned}
\end{equation}
for $n > 2\gamma + 8 (= 2\gamma + 4(d_0+1))$, where the value in the bracket in front of $P_{21}$ is understood as 1, then it is true that
\begin{equation}\label{eq-poly-0}
\(J_1^{\gamma} + J_2^{\gamma}\)(\dz) = P\big(\delta^2\big).
\end{equation}
\end{prop}
\begin{rmk}\label{rmk-al-br}
Clearly the above result recovers the polynomial found by Almaraz for the case of $\gamma = 1/2$ (up to a multiplicative constant). See (4.6) of \cite{Al}.
Furthermore, putting $\gamma = 1$ and $a_0 = -a_1 = 1$ allows us to regain the polynomial of Brendle in \cite[Proposition 19]{Br}.
\end{rmk}

It is notable that $\Psi^A_{\delta,0} = 0$ holds since $\sum_{i,j=1}^n H_{ij}\pa_{ij} W_{\delta} = 0$ in $\mr^N_+$.
Furthermore, it is straightforward to check
\begin{align*}
&\ \int_{\mr_+^N} x_N^{1-2\gamma} \(\oh_{il}\oh_{jl}\)(\bx) \(\pa_i W_{\delta}\pa_j W_{\delta}\)(\bx,x_N) dx \\
&= \int_0^{\infty} x_N^{1-2\gamma} \int_0^{\infty} f(r^2) \({\pa_rW_{\delta} \over r}\)^2 \(W_{ikls}W_{jplq} \int_{S^{n-1}(0,r)} x^ix^jx^kx^px^qx^s dS_r\) dr dx_N = 0
\end{align*}
by using the contraction and anti-symmetry properties of the tensor $W$.
Hence $J_1^{\gamma}$ and the first term of $J_2^{\gamma}$ vanishes if $\tau = 0$.
Besides, it can be easily seen that $W_{\delta}(x) = W_{\delta}(|\bx|,x_N)$ for any $x = (\bx, x_N) \in \mr_+^N$ owing to the representation formula \eqref{eq-Poisson} of $W_{\delta}$.
Therefore, in view of \eqref{eq-C} and \eqref{eq-wjg}, the proof of the proposition is reduced to computing
\begin{align*}
&\ \int_{\mr^N_+} x_N^{1+2m-2\gamma} \sum_{i,j,k=1}^n \Delta^m\(\pa_k \oh_{ij}\)^2(\bx) W_{\delta}^2(\bx,x_N)dx\\
&= \int_0^{\infty} \int_0^{\infty} \left[\int_{S^{n-1}(0,r)} \sum_{i,j,k=1}^n \Delta^m\(\pa_k \oh_{ij}\)^2(\bx) dS_r\right] x_N^{1+2m-2\gamma} W_{\delta}^2 (r,x_N) drdx_N
\end{align*}
for $m = 0, \cdots, 2d_0+1$ and
\begin{align*}
&\ \int_{\mr^N_+} x_N^{1+2m-2\gamma} \sum_{i,j,k=1}^n \Delta^{m-1}\(\pa_k \oh_{ij}(\bx)\)^2 |\nabla W_{\delta}|^2(\bx,x_N)dx \\
&= \int_0^{\infty} \int_0^{\infty} \left[\int_{S^{n-1}(0,r)} \sum_{i,j,k=1}^n \Delta^{m-1}\(\pa_k \oh_{ij}\)^2(\bx) dS_r \right] x_N^{1+2m-2\gamma} \(|\pa_r W_{\delta}|^2 + |\pa_{x_N} W_{\delta}|^2\) drdx_N
\end{align*}
for $m = 1, \cdots, 2d_0+2$ where $r = |\bx|$.
The most crucial part is to obtain the value of the integrals over the spheres $S^{n-1}(0,r)$.
To do so, it is necessary to look at how the terms $\sum_{i,j,k=1}^n \Delta^m\(\pa_k \oh_{ij}\)^2$ look like.
\begin{lemma}\label{lemma-lap-paH}
For any $m = 0, \cdots, 2d_0+1$, there are radial functions $G_{1,m}$, $G_{2,m}$ and $G_{3,m}$ in $\mr^n$ such that
\begin{equation}\label{eq-lap-paH}
\sum_{i,j,k=1}^n \Delta^m\(\pa_k \oh_{ij}\)^2(\bx) = G_{1,m}(r) \sum_{i,j,k=1}^n \(\pa_kH_{ij}\)^2(\bx) + G_{2,m}(r) \sum_{i,j=1}^n H_{ij}^2(\bx) + G_{3,m}(r)|W|^2
\end{equation}
where $r = |\bx|$.
\end{lemma}
\begin{proof}
We will use mathematical induction to justify the statement.
By the definition of the tensor $\oh_{ij}$, we know
\begin{equation}\label{eq-paH}
\sum_{i,j,k=1}^n \(\pa_k \oh_{ij}\)^2(\bx) = f(r^2)^2 \sum_{i,j,k=1}^n \(\pa_kH_{ij}\)^2(\bx) + \left[8f(r^2)f'(r^2) + 4r^2f'(r^2)^2 \right] \sum_{i,j=1}^n H_{ij}^2(\bx)
\end{equation}
so that \eqref{eq-lap-paH} is valid for $m = 0$.
See the proof of \cite[Proposition 15]{BM} for its detailed derivation.
Suppose that \eqref{eq-lap-paH} holds for $m = \tm$.
Then we verify by direct computations utilizing $x^k\pa_kH_{ij}(\bx) = 2H_{ij}(\bx)$, $x^l\pa_{kl}H_{ij}(\bx) = \pa_kH_{ij}(\bx)$, $\sum_{i,j,k,l=1}^n(\pa_{kl}H_{ij}(\bx))^2 = |W|^2$ and $\sum_{k=1}^n \pa_{kk}H_{ij}(\bx) = 0$ that
\begin{align*}
\sum_{i,j,k=1}^n \Delta^{\tm+1}\(\pa_k \oh_{ij}\)^2(\bx)
&= \left[G_{1,\tm}''(r) + \({n+3 \over r}\)G_{1,\tm}'(r) + 2G_{2,\tm}(r) \right] \sum_{i,j,k=1}^n \(\pa_kH_{ij}\)^2(\bx) \\
&\ + \left[G_{2,\tm}''(r) + \({n+7 \over r}\)G_{2,\tm}'(r)\right] \sum_{i,j=1}^n H_{ij}^2(\bx) \\
&\ + \left[G_{3,\tm}''(r) + \({n-1 \over r}\)G_{3,\tm}'(r) + 2G_{1,\tm}(r) \right] |W|^2.
\end{align*}
Here $G'(r)$ represents the differentiation of $G(r)$ with respect to the radial variable $r$.
Thus \eqref{eq-lap-paH} holds for $m = \tm+1$ as well.
The proof is finished.
\end{proof}

\noindent By the previous lemma, the desired integrals will be evaluated once we get
\begin{lemma}\label{lemma-J-zero-a}
It holds that
\[\sum_{i,j=1}^n \int_{S^{n-1}} H_{ij}^2(\bx) dS
= {1 \over 2(n+2)} \sum_{i,j,k=1}^n \int_{S^{n-1}} \(\pa_kH_{ij}\)^2(\bx) dS
= {\left|S^{n-1}\right| \over 2n(n+2)} |W|^2.\]
\end{lemma}
\begin{proof}
We deduce it by adapting the proof of \cite[Proposition 16]{Br}.
\end{proof}

\noindent Combining all results of this subsection, we are able to complete the proof of Proposition \ref{prop-poly}.
Actually the explicit expression of $G_{1,m}$, $G_{2,m}$ and $G_{3,m}$ is also necessary, but it can be derived from the proof of Lemma \ref{lemma-lap-paH}.
Observe that the definition of the polynomials in \eqref{eq-poly-2} are motivated from the value of
\[{1 \over \left|S^{n-1}\right||W|^2 r^{n-1}} \int_{S^{n-1}(0,r)} \sum_{i,j,k=1}^n \Delta^m\(\pa_k \oh_{ij}\)^2(\bx) dS_r \quad \text{for } m = 0, \cdots, 3 (= 2d_0+1).\]
We leave the details to the reader.

\subsection{The second derivative of $\(J_1^{\gamma} + J_2^{\gamma}\)(\dt)$ at $(\dt) = (\dz)$} \label{subsec-exp-4}
Our goal in this subsection is to calculate the function $\pa_{\tau_i\tau_j} \(J_1^{\gamma} + J_2^{\gamma}\)(\cdot, 0)$ for each fixed $i, j = 1, \cdots, n$.
This observation will be used in Section \ref{sec-conc} on finding a local minimizer of $J_1^{\gamma} + J_2^{\gamma}$ (see (C3) below).

We start this subsection by establishing variants of Lemmas \ref{lemma-lap-paH} and \ref{lemma-J-zero-a}. Set a symmetric two-tensor
\begin{equation}\label{eq-wtw}
\wtw_{ij} = \sum_{k,p,q=1}^n (W_{ikpq}+W_{kqip})(W_{jkpq}+W_{kqjp}).
\end{equation}

\begin{lemma}\label{lemma-J-sec-1}
For any $i, j \in \{1, \cdots, n\}$ and $m = 0, \cdots, 2d_0+1$, there are radial functions $\wtg_{1,m}, \cdots, \wtg_{11,m}$ in $\mr^n$ such that
\begin{align*}
&\ \sum_{k,p,q=1}^n \Delta^m \left[\pa_{ij}\(\pa_k \oh_{pq}\)^2\right](\bx) \\
&= \wtg_{1,m}(r) \delta_{ij} \sum_{k,p,q=1}^n \(\pa_kH_{pq}\)^2(\bx) + \wtg_{2,m}(r) \delta_{ij} \sum_{p,q=1}^n H_{pq}^2(\bx) + \wtg_{3,m}(r) \delta_{ij} |W|^2 \\
&\ + \wtg_{4,m}(r) \sum_{p,q=1}^n \left[x_i \(H_{pq} \pa_j H_{pq}\)(\bx) + x_j \(H_{pq} \pa_i H_{pq}\)(\bx)\right] + \wtg_{5,m}(r) \sum_{p,q=1}^n \(\pa_i H_{pq} \pa_j H_{pq}\)(\bx) \\
&\ + \wtg_{6,m}(r) \sum_{k,p,q=1}^n \left[x_i \(\pa_kH_{pq} \pa_{jk}H_{pq}\)(\bx) + x_j \(\pa_kH_{pq} \pa_{ik}H_{pq}\)(\bx)\right] + \wtg_{7,m}(r) \wtw_{ij}
\stepcounter{equation}\tag{\theequation}\label{eq-lap-paH-2}\\
&\ + \wtg_{8,m}(r) x_ix_j \sum_{k,p,q=1}^n \(\pa_k H_{pq}\)^2(\bx) + \wtg_{9,m}(r) x_ix_j \sum_{p,q=1}^n H_{pq}^2(\bx) \\
&\ + \wtg_{10,m}(r) x_ix_j|W|^2 + \wtg_{11,m}(r) \sum_{p,q=1}^n \(H_{pq} \pa_{ij} H_{pq}\)(\bx)
\end{align*}
where $r = |\bx|$.
\end{lemma}
\begin{proof}
It is plain to check that \eqref{eq-lap-paH-2} is correct for $m = 0$ by employing \eqref{eq-paH}.
Now apply mathematical induction on $m$, referring to the proof of Lemma \ref{lemma-lap-paH}.
The explicit values of $\wtg_{1,m}, \cdots, \wtg_{11,m}$ can be found in \cite{KMW}.
\end{proof}

\begin{lemma}\label{lemma-J-sec-2}
It is valid that
\begin{align*}
\sum_{p,q=1}^n \int_{S^{n-1}} x_ix_j H_{pq}^2(\bx) dS
&= {\left|S^{n-1}\right| \over 2n(n+2)(n+4)} |W|^2 \delta_{ij} + {2 \left|S^{n-1}\right| \over n(n+2)(n+4)} \wtw_{ij},\\
\sum_{k,p,q=1}^n \int_{S^{n-1}} x_ix_j \(\pa_kH_{pq}\)^2(\bx) dS
&= {\left|S^{n-1}\right| \over n(n+2)} |W|^2 \delta_{ij} + {2 \left|S^{n-1}\right| \over n(n+2)} \wtw_{ij},\\
\sum_{p,q=1}^n \int_{S^{n-1}} x_i\(H_{pq}\pa_jH_{pq}\)(\bx) dS
&= {1 \over n+2} \sum_{p,q=1}^n \int_{S^{n-1}} \(\pa_iH_{pq}\pa_jH_{pq}\)(\bx) dS \\
&= {1 \over n+2} \sum_{k,p,q=1}^n \int_{S^{n-1}} x_i \(\pa_kH_{pq}\pa_{jk}H_{pq}\)(\bx) dS
= {\left|S^{n-1}\right| \over n(n+2)} \wtw_{ij},
\end{align*}
and
\[\sum_{p,q=1}^n \int_{S^{n-1}} \(H_{pq}\pa_{ij}H_{pq}\)(\bx) dS = 0\]
for each $i, j \in \{1, \cdots, n\}$.
\end{lemma}
\begin{proof}
The first and second identities in the statement are precisely the ones examined in \cite[Proposition 16]{Br}.
We can deduct the other identities by arguing as its proof.
\end{proof}

\noindent By the previous lemmas, we discover
\begin{prop}\label{prop-J-sec}
Assume that $d_0 = 1$ and $n > 2\gamma+8$, and define
\begin{align*}
\wtp_{1;0}(t) &= {1 \over 2n(n+2)}\left[a_1^2F_2(1,6)t^3 + 2a_0a_1F_2(1,4)t^2 + a_0^2F_2(1,2)t\right],\\
\wtp_{1;1}(t) &= {1 \over n(n+2)}\left[2a_1^2(n+6)(n+16)F_1(1,4)t^3 + 4a_0a_1(n+2)(n+8)F_1(1,2)t^2 \right. \\
&\hspace{50pt} \left.+ 2a_0^2n(n+2)F_1(1,0)t\right],\\
\wtp_{1;31}(t) &= {1 \over n}\left[8a_1^2(n+6)(n+16)F_1(3,2)t^3 + 8a_0a_1n(n+8)F_1(3,0)t^2 \right],\\
\wtp_{1;32}(t) &= 16a_1^2(n+6)(n+16)F_1(5,0)t^3,\\
\wtp_{2;0}(t) &= 0,\\
\wtp_{2;1}(t) &= {1 \over n(n+2)}\left[4a_1^2(n+7)F_1(1,4)t^3 + 4a_0a_1(n+2)F_1(1,2)t^2 \right],\\
\wtp_{2;31}(t) &= {1 \over n}\left[16a_1^2(n+7)F_1(3,2)t^3 + 8nF_1(3,0)t^2 \right],\\
\wtp_{2;32}(t) &= 16a_1^2(2n+13)F_1(5,0)t^3,
\end{align*}
where $F_1 = F_{1,n,\gamma}$ and $F_2 = F_{2,n,\gamma}$ are given in \eqref{eq-F}.
Also, for each $\hat{m} = 1$ and 2, we set the polynomials $\wtp_{\hat{m};21}, \wtp_{\hat{m};22}$ and $\wtp_{\hat{m};23}$
by replacing each $F_{1,n,\gamma}(\alpha,\beta)$ in $\wtp_{\hat{m};0}, \wtp_{\hat{m};31}$ and $\wtp_{\hat{m};32}$ with $F_{4,n,\gamma}(\alpha+2,\beta)$. If we put for $\hat{m} = 1$ or 2,
\begin{align*}
&\ -24n(n-1) \cdot \wtp_{\hat{m}}(t)\\
&= -24n(n-1) \wtp_{\hat{m};0}(t) + {3 \over 2} \((n-1)^2-(1-2\gamma)^2\) \wtp_{\hat{m};1}(t)\\
&\quad + (1-2\gamma) \sum_{m=1}^{2d_0+1} \left[\prod_{\tm=1}^{m-1} {1 \over (2\tm+3)(N-2(\tm+1))}\right] \wtp_{\hat{m};2m}(t) \\
&\quad + \left\{(n+1)\(\gamma-{1 \over 2}\) - 2\(\gamma^2-{1 \over 4}\) \right\} \sum_{m=1}^{2d_0} \left[\prod_{\tm=1}^m {1 \over (2\tm+3)(N-2(\tm+1))}\right] (2m+3)\wtp_{\hat{m};3m}(t)
\end{align*}
where the value in the bracket in front of $\wtp_{\hat{m};21}$ is regarded as 1, then
\begin{equation}\label{eq-poly-1}
{\pa^2 \(J_1^{\gamma} + J_2^{\gamma}\) \over \pa\tau_i \pa\tau_j}(\delta,0) = \wtp_1(\delta^2) \wtw_{ij} + \wtp_2(\delta^2) \delta_{ij}|W|^2.
\end{equation}
\end{prop}
\begin{proof}
\noindent \textsc{Step 1.}
We start the proof by showing that
\[{\pa^2 J_1^{\gamma} \over \pa\tau_i \pa\tau_j}(\delta,0)
= \sum_{k,l=1}^n \int_{\mr^N_+} x_N^{1-2\gamma} \pa_{kl}W_{\delta}(x)
\cdot \pa_{\tau_i\tau_j} \left. \( \oh_{kl}(\bx + \tau) \Psi_{\dt}^A (\bx + \tau) \)\right|_{\tau = 0} dx = 0.\]
Indeed, since $\Psi^A_{\dt}$ and $\pa_{\tau_i} \Psi^A_{\dt}$ are smooth in $\bx$ and $W_{\delta}(\bx,x_N) = W_{\delta}(r,x_N)$ where $r = |\bx|$, we have
\[\sum_{k,l=1}^n \pa_{kl}W_{\delta}(x) \cdot \pa_{\tau_i\tau_j} \left. \( \oh_{kl}(\bx + \tau) \Psi_{\dt}^A (\bx + \tau) \)\right|_{\tau = 0}
= \sum_{k,l=1}^n \pa_{kl}W_{\delta}(x) \cdot \pa_{ij} \oh_{kl}(\bx) \Psi_{\delta, 0}^A(x) = 0 \quad \text{for } x \in \mr^N_+.\]
Therefore the assertion is true.

\noindent \textsc{Step 2.}
We next treat the derivative of the first term $J_{20}^{\gamma}$ in $J_2^{\gamma}$ (defined in \eqref{eq-wjg}).
Because of the observation
\begin{align*}
\left. \sum_{l=1}^n \pa_{\tau_i\tau_j} \left[\(\oh_{pl}\oh_{ql}\)(\bx + \tau)\right] \right|_{\tau = 0} x^px^q
&= \left. \sum_{l=1}^n \pa_{\tau_i\tau_j} \left[\tau^k\tau^{\tilde{k}} f\(|\bx+\tau|^2\)^2 W_{pkls}W_{q\tilde{k}l\tilde{s}} x^sx^{\tilde{s}}x^px^q\right] \right|_{\tau = 0}\\
&= \left. \sum_{l=1}^n \pa_{\tau_i\tau_j} \left[\tau^k\tau^{\tilde{k}} f\(|\bx+\tau|^2\)^2 \(H_{kl}H_{\tilde{k}l}\)(\bx)\right] \right|_{\tau = 0} \\
&= 2\sum_{l=1}^n\(\oh_{il}\oh_{jl}\)(\bx)
\end{align*}
which is true for any fixed $x = (x^1, \cdots, x^n)$ and $\tau \in \mr^n$, we obtain
\begin{align*}
{\pa^2 J_{20}^{\gamma} \over \pa\tau_i \pa\tau_j}(\delta,0)
&= {1 \over 2} \int_{\mr_+^N} x_N^{1-2\gamma} \left. \sum_{l=1}^n \pa_{\tau_i\tau_j} \left[\(\oh_{pl}\oh_{ql}\)(\bx + \tau)\right] \right|_{\tau = 0} x^px^q {(\pa_rW_{\delta})^2(\bx,x_N) \over r^2} dx \\
&= \int_{\mr_+^N} x_N^{1-2\gamma} \sum_{l=1}^n \(\oh_{il}\oh_{jl}\)(\bx) |\bx|^{-2} |\nabla_{\bx} W_{\delta} (\bx, x_N)|^2 dx.
\end{align*}
Thus, after carrying out computations as in Subsection \ref{subsec-exp-2}
and in particular applying the first identity in the proof of \cite[Proposition 20]{BM}, we get
\begin{align*}
{\pa^2 J_{20}^{\gamma} \over \pa\tau_i \pa\tau_j}(\delta,0)
&= \int_0^{\infty} \int_0^{\infty} \left[\int_{S^{n-1}(0,r)} \sum_{l=1}^n \(\oh_{il}\oh_{jl}\)(\bx) dS_r \right] x_N^{1-2\gamma} r^{-2} |\nabla_{\bx} W_{\delta} (\bx,x_N)|^2 drdx_N\\
&= {\delta^2 \over 2n(n+2)} \wtw_{ij} \(a_1^2 \delta^4 F_2(1,6) + 2a_0a_1 \delta^2 F_2(1,4) + a_0^2 F_2(1,2)\)
\end{align*}
where $F_2 = F_{2,n,\gamma}$ is set in \eqref{eq-F}.

\medskip \noindent \textsc{Step 3.}
For $\gamma = 1$, one can compute the second derivatives $\pa_{\tau_i\tau_j} J_2^{\gamma} (\delta,0)$ as in \cite[Proposition 21]{Br}.
However, since the explicit formula for the bubble $W_{\delta}$ is unknown in our case except when $\gamma = 1/2$, we cannot follow it and need to devise an alternative approach.

As a matter of the fact, as we can expect from the previous step, it suffices to calculate the values
\[\left.{\pa^2 \over \pa\tau_i \pa\tau_j} \left[\int_{S^{n-1}(0,r)} \sum_{k,p,q=1}^n \Delta^m\(\pa_k \oh_{pq}\)^2(\bx) dS_r \right] \right|_{\tau = 0}
= \int_{S^{n-1}(0,r)} \sum_{k,p,q=1}^n \Delta^m \left[\pa_{ij}\(\pa_k \oh_{pq}\)^2\right](\bx) dS_r\]
for $m = 0, \cdots, 2d_0+1$.
Therefore we can achieve the result by applying Lemmas \ref{lemma-J-zero-a}, \ref{lemma-J-sec-1} and \ref{lemma-J-sec-2}.
The proof is concluded.
\end{proof}

\section{Search for a critical point of the polynomial $P$ and conclusion of the proof of Theorem \ref{thm-main}}\label{sec-conc}
\subsection{A positive local minimizer of the polynomial $P$}\label{subsec-high}
We now choose appropriate coefficients $a_0$ and $a_1$ of the polynomial $f(t) = a_0 + a_1t$ in \eqref{eq-h}
so that the function $J_1^{\gamma} + J_2^{\gamma}$ introduced in \eqref{eq-J1} and \eqref{eq-wjg} has a strict local minimum at $(1,0)$, provided that the dimension $n$ is sufficiently large.
By \eqref{eq-J-diff} and \eqref{eq-poly-0}, it suffices to confirm three conditions

\medskip \noindent (C1) $\pa_{\delta} (J_1^{\gamma} + J_2^{\gamma}) (1,0) = 2P'(1) = 0$;

\noindent (C2) $\pa_{\delta\delta} (J_1^{\gamma} + J_2^{\gamma}) (1,0) = 4P''(1) > 0$;

\noindent (C3) The matrix $\(\pa_{\tau_i\tau_j} \(J_1^{\gamma} + J_2^{\gamma}\)(1,0)\)_{i,j=1}^n$ is positive definite;

\medskip \noindent to guarantee that $(1,0)$ is a strict minimizer of $J_1^{\gamma} + J_2^{\gamma}$.

\medskip
As in Subsection 4.1 of \cite{Al}, we put $a_1 = -1$ and then denote
\begin{equation}\label{eq-P-der}
P'(1) = Q(a_0)
\end{equation}
where $P$ is the polynomial defined in \eqref{eq-poly-3}.
Then $Q(t) = b_0 + b_1t + b_2t^2$ ($b_2 < 0$) is a quadratic polynomial in $t \in \mr$ whose exact definition is described in \cite{KMW}.
Let $\text{disc}(Q) = \text{disc}(Q)(n,\gamma) = b_1^2 - 4b_0b_2$ be the discriminant of $Q$, which is a function of $n \in \mn$ and $\gamma \in (0,1)$.
Then we discover that it is positive for all $\gamma \in (0,1)$ whenever $n \ge 52$.
To check this fact, we observe that $\text{disc}(Q)(n,\gamma) = C(n,\gamma)R(n,\gamma)$ where
\[R(n,\gamma) \simeq 33075n^{17} - 3307500n^{16} + 132300 \(893 - 3\gamma^2\)n^{15}\]
is the 17th order polynomial in $n$ and $C(n,\gamma) > 0$ for every $n \ge 52$ and $\gamma \in (0,1)$.
After expanding $R(n,\gamma)$ in terms of $n$ and $\gamma$, we put $\gamma = 1$ (0, respectively) into each term whose coefficient is negative (nonnegative, respectively).
Then we get $\tilde{R}(n) \simeq 33075n^{17} - 3307500n^{16} + 117747000n^{15}$, which is obviously a lower bound of $R(n,\gamma)$ for any $\gamma \in (0,1)$.
Since the largest real solution of $\tilde{R}$ is $n \simeq 52.2022$, we conclude that $\tilde{R}(n) > 0$, hence $R(n,\gamma) > 0$ for each $n \ge 53$.
Also it can be directly checked that $R(52,\gamma) > 0$ for all $\gamma \in (0,1)$.
(The precise expression of $R(n,\gamma)$, $C(n,\gamma)$ and $\tilde{R}(n)$ can be found in \cite{KMW}.)

\medskip
Let us set
\begin{equation}\label{eq-a0} 
a_0 = {- b_1 - \sqrt{\text{disc}(Q)} \over 2b_2}
\end{equation}
so that $P'(1) = 0$. We now claim
\begin{prop}\label{prop-loc-min}
Fix $\gamma \in (0,1)$ and assume that $n \ge 52$.
If the coefficients $a_0$ and $a_1$ of the polynomial $f$ in \eqref{eq-h} are selected by \eqref{eq-a0} and $a_1 = -1$,
then $(1,0)$ is the strict local minimizer of the localized energy $J_1^{\gamma} + J_2^{\gamma}$.
\end{prop}
\begin{proof}
Denote by $\tilde{a}_0$ the number which we chose for the coefficient $a_0$, namely, the right-hand side of \eqref{eq-a0}.
By the above discussion and \eqref{eq-J-diff}, we see that $(1,0)$ is the critical point of $J_1^{\gamma} + J_2^{\gamma}$.
We need to be verify that it satisfies conditions (C2) and (C3).

\medskip \noindent
\textsc{Step 1.} Let us establish (C2) first. Thanks to \eqref{eq-poly-0}, we have $\pa_{\delta\delta}(J_1^{\gamma} + J_2^{\gamma})(1,0) = 4P''(1)|_{a_0 = \tilde{a}_0}$.
Since $Q(99/50) > 0$ for all $\gamma \in (0,1)$ and the leading coefficient $b_2$ of $Q$ is negative, we have $\tilde{a}_0 > 99/50$.
On the other hand, one can check that $P''(1) - P'(1)$ is an increasing linear function in $a_0$, say $\widetilde{Q}(a_0)$, whose coefficients depend on $n$ and $\gamma$.
As a result, $P''(1) = \widetilde{Q}(\tilde{a}_0) > \widetilde{Q}(99/50) > 0$.

\medskip \noindent
\textsc{Step 2.} We check (C3).
This will be followed by Proposition \ref{prop-J-sec} and our assumption $|W| > 0$ once we derive that $\wtp_1(1) > 0$ and $\wtp_2(1) > 0$.
If we regard the functions $\wtp_1$ and $\wtp_2$ as a polynomial in $a_0$, then clearly their degrees are at most 2.
In fact, further computation shows that they are increasing linear functions in $a_0$ for any $n \ge 52$ and $\gamma \in (0,1)$.
From this fact, we get $\wtp_m(1)|_{a_0 = \tilde{a}_0} > \wtp_m(1)|_{a_0 = 99/50} > 0$ for $m = 1, 2$.
\end{proof}

\subsection{The lower dimensions} \label{subsec-low}
For lower dimensional case, we make the reduced energy functional $J_0^{\gamma}$ to have a local minimizer by inserting a polynomial $f$ of higher degree ($d_0 \ge 2$) in the definition of the tensor $h$ in \eqref{eq-h}.
This approach is pursued in the local cases by Brendle-Marques \cite{BM} ($\gamma = 1$), Almaraz \cite{Al} ($\gamma = 1/2$) and Wei-Zhao \cite{WZ} ($\gamma = 2$).
Here we will select a quartic polynomial $f(t) = \sum_{i=0}^4 a_it^i$ as in \cite{Al} and \cite{WZ}.
In \cite{BM}, the cubic polynomial was chosen.

\medskip
By using the computations in Subsections \ref{subsec-exp-3} and \ref{subsec-exp-4} again, we extend Propositions \ref{prop-poly} and \ref{prop-J-sec}.
\begin{prop}
Assume that $n > 2\gamma + 20$ and the degree of the polynomial $f$ is $d_0 = 4$.
Then \eqref{eq-poly-0} and \eqref{eq-poly-1} hold for some polynomials $P$ of degree 10, and $\wtp_1$ and $\wtp_2$ of degree 9, respectively.
The coefficients of $P$, $\wtp_1$ and $\wtp_2$ depends on $a_0, \cdots, a_4$.
(The full details can be found in \cite{KMW}.)
\end{prop}
\begin{rmk}
As in the higher dimensional case, we obtain the polynomial of Almaraz \cite{Al} from $P$ when $\gamma = 1/2$.
Furthermore, if we take $f(s) = \tau + 5s - s^2 + {1 \over 20}s^3$ and $\gamma = 1$, then we again attain the function $I(s)$ defined in \cite[Proposition 18]{BM} as a factor of $P$.
\end{rmk}

Now, we set
\begin{equation}\label{eq-f-lower}
f(t) = t^4 - {882178 \over 10000} t^3 + {146178 \over 100} t^2 - {713925 \over 100} t + a_0
\end{equation}
leaving $a_0$ undetermined for a minute.
Defining the polynomial $Q$ as in \eqref{eq-P-der}, we again find that it is a quadratic polynomial.
Like above, let us write $Q(t) = b_0 + b_1t + b_2t^2$. We also deduce
\begin{enumerate}
\item $\text{disc}(Q)(n,\gamma) > 0$ for all $\gamma \in (0,1)$ whenever $25 \le n \le 51$;
\item the function $\gamma \mapsto \text{disc}(Q)(24,\gamma)$ is positive if $\gamma \in (0,\gamma^*)$ and negative if $\gamma \in (\gamma^*,1)$ where $\gamma^* \simeq 0.940197$.
\begin{figure}
\centering
\includegraphics[width=8cm]{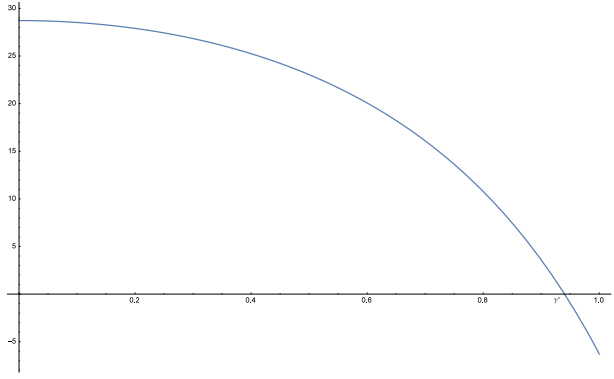}
\caption{\small The graph of $\gamma \mapsto \text{disc}(Q)(24,\gamma)$ ($0 < \gamma < 1$)}
\end{figure}
    We chose $f$ so that $(0,\gamma^*)$ well approximates the longest interval where the blow-up phenomenon occurs.
\item $\gamma \mapsto \text{disc}(Q)(23,\gamma) < 0$ for all $\gamma \in (0,1)$.
    As a matter of fact, we could not find any quartic polynomial $f$ which leads the positive discriminant of $Q$ for some $0 < \gamma < 1$ provided that $n = 23$.
\end{enumerate}
If we denote
\[n(\gamma) = \min\{n_0 \in \mn: \text{disc}(Q)(n_0,\gamma) > 0 \text{ for } n_0 \le n \le 51\}\]
and take $a_0$ as in \eqref{eq-a0},
the following assertion is valid. This is an analogue of Proposition \ref{prop-loc-min} for lower dimensions.
\begin{prop}\label{prop-loc-min2}
Fix $\gamma \in (0,1)$ and assume that $n(\gamma) \le n \le 51$.
If the polynomial $f$ is given by \eqref{eq-f-lower} with \eqref{eq-a0},
then $(1,0)$ is the strict local minimizer of the localized energy $J_1^{\gamma} + J_2^{\gamma}$.
\end{prop}
\begin{proof}
As in the proof of Proposition \ref{prop-loc-min}, the assertion is justified if we check that conditions (C2) and (C3) are true.
Their verification can be done for each $n(\gamma) \le n \le 51$.
\end{proof}

\begin{rmk}
For a sufficiently small $\gamma > 0$, the best $n(\gamma)$ one can get with a cubic (a quadratic, respectively) polynomial $f$ is 25 (29, respectively).
Moreover we need $n > 2\gamma + 24$ when we put a quintic polynomial $f$ into the metric.
This is because the polynomial $P$ (see Proposition \ref{prop-poly}) would contain $F_1(1,22)$ and $F_4(3,22)$ as its coefficients
and they are finite only if the dimensional assumption $n > 2\gamma + 24$ holds.
\end{rmk}

\subsection{Completion of the proof of Theorem \ref{thm-main}}
From what we have obtained so far, we can deduce the following existence result.
\begin{prop}\label{prop-main}
Assume that $n \ge 24$ if $\gamma \in (0, \gamma^*)$ or $n \ge 25$ if $\gamma \in [\gamma^*, 1)$ (refer to Subsection \ref{subsec-low} for the definition of the number $\gamma^*$).
If $\ep > 0$ is a small parameter in \eqref{eq-h}, $\bg$ is the metric tensor and $\rho$ is the boundary defining function chosen in Section \ref{sec-setting},
then Eq. \eqref{eq-main-ext} possesses a positive solution $U_{\ep}$ in $\mr_+^{n+1}$,
whose restriction $u_{\ep}$ on $\mr^n$ satisfies the fractional Yamabe equation \eqref{eq-main} with $c = 1$ and $\|u_{\ep}\|_{L^{\infty}(\mr^n)} \ge C \ep^{-{n-2\gamma \over 2}}$.
\end{prop}
\begin{proof}
For the existence of a positive solution to \eqref{eq-main-ext},
it suffices to search a critical point of the functional $J_0^{\gamma}$ by Lemma \ref{lemma-red}.
Lemma \ref{lemma-rem} and Proposition \ref{prop-expand} ensure that
if one finds a local minimizer of $J_1^{\gamma} + J_2^{\gamma}$ (see \eqref{eq-J1} and \eqref{eq-wjg}) in the admissible set
$\ma = (1-\varepsilon_0, 1+\varepsilon_0) \times B^n(0,\varepsilon_0)$ for some small $\varepsilon_0 > 0$,
then the set $\ma$ must contain also a local minimum $J_0^{\gamma}$.
However, we already know its validity from Propositions \ref{prop-loc-min} and \ref{prop-loc-min2}.
Thus \eqref{eq-main-ext} has a positive solution.
The lower $L^{\infty}(\mr^n)$-bound of the solution comes from \eqref{eq-lower-bd}.
\end{proof}

\noindent We are now ready to finish our proof of the main theorem.
\begin{proof}[Proof of Theorem \ref{thm-main}]
Define a smooth two-tensor $h_{ab}$ in $\mr_+^N$ as
\[h_{ab}(x) = \sum_{m = m_0}^{\infty} \chi\(4m^2|x-x_m|\)2^{-(d_0 + 1/6)m} f\(2^m|x-x_m|^2\) H_{ab}(x-x_m).\]
Here $\chi \in C^{\infty}(\mr)$ is a truncation function such that $\chi(t) = 1$ for $|t| \le 1$ and 0 for $|t| \ge 2$,
$H_{ab}$ is the tensor in \eqref{eq-h0}, $x_m = (m^{-1}, 0, \cdots, 0) \in \omrn$ and
\[d_0 = \begin{cases}
1 &\text{for } n \ge 52 \text{ (see Subsection \ref{subsec-high}),}\\
4 &\text{for } 24 \le n \le 51 \text{ and } \gamma \in (0, \gamma^*], \text{or } 25 \le n \le 51 \text{ and } \gamma \in (\gamma^*, 1) \text{ (see Subsection \ref{subsec-low}).}
\end{cases}\]
If we set $\bg = \exp{(h)}$, then one can construct a metric tensor $g^+$ and a defining function $\rho$ on $\mr_+^N$ as described in the proof of Proposition \ref{prop-rho-est}.
Moreover, \eqref{eq-rho-est} remains valid because the proof requires only the local structure of $\bg$ and the vanishing mean curvature condition $H = 0$ on $\mr^n$.
Therefore, by choosing $m_0 \in \mn$ so large that \eqref{eq-eta_0} holds, we can employ Proposition \ref{prop-main} with $\mu = 2^{-m/6}$, $\ep = 2^{-m/2}$ and $\nu = (4m)^{-1}$, completing our proof of Theorem \ref{thm-main}.
\end{proof}

\appendix
\section{The Loewner-Nirenberg problem}\label{sec-app}
\subsection{Existence and uniqueness of the solution}\label{subsec-app-a}
The existence theorem in Andersson-Chru\'{s}ciel-Friedrich \cite{ACF} to the singular Yamabe problem is presented in the setting of compact Riemannian manifolds.
In this subsection we illustrate how their result can be applied to the problem \eqref{eq-LN} defined in the upper half space.

\medskip
To this end, we define some notations:
Let $B_{1/2} := B^N((0,\cdots,0,-1/2),1/2)$ be the ball of radius $1/2$ centered at $(0,\cdots,0,-1/2) \in \mr^N$ and $y_s = (0, \cdots, 0, -1)$.
Moreover let $\mc: \omrn \to \overline{B_{1/2}} \setminus \{y_s\}$ be a conformal equivalence between the sets $\omrn$ and $\overline{B_{1/2}} \setminus \{y_s\}$, and $\md$ its inverse expressed as
\[\mc(x) = {(\bx, x_N+1) \over |\bx|^2+(x_N+1)^2} + (0, \cdots, 0, -1) \quad \text{and} \quad \md(y) = \mc(y)\]
for $x = (\bx, x_N) \in \omrn$ and $y \in \overline{B_{1/2}} \setminus \{y_s\}$.
Next if we denote
\[\mw_{1/2}(x) = {1 \over \(|\bx|^2+(x_N+1)^2\)^{N-2 \over 2}} \quad \text{for } x = (\bx, x_N) \in \omrn,\]
then it is the standard bubble in $\mr_+^N$ which is the same function as $W_{1,0}$ in \eqref{eq-Poisson} up to a constant multiple provided that $\gamma = 1/2$.
Introduce also pullback metrics
\[\bg_B := \md^*\(\mw_{1/2}^{4 \over N-2}\bg\) \quad \text{in } \overline{B_{1/2}} \setminus \{y_s\} \quad \text{and } \quad \tg_B := \md^*\(\tg\) \quad \text{in } B_{1/2}.\]
We can smoothly extend $\bg_B$ on the whole closed ball $\overline{B_{1/2}}$ by defining $\bg_B(y_s) = \delta_B(y_s)$,
where $\delta_B$ means the canonical metric on the ball $\overline{B_{1/2}}$,
because $\bg$ is equal to the standard metric $g_c$ outside of the half ball $\{|x| \le 1\}$.
Furthermore $\tg_B(y) = \md^*\(x_N^{-2}\bg(x)\) = x_N^{-2}(\mw_{1/2}(x))^{-{4 \over N-2}} \bg_B(y)$ for all $y = \mc(x) \in B_{1/2}$.

\begin{lemma}
Set $\rho_B(y) = (\md(y))_N (\mw_{1/2}(\md(y)))^{2 \over N-2}$ for $y = \mc(x) \in \overline{B_{1/2}} \setminus \{y_s\}$ and $\rho_B(y_s) = 0$.
Then it is a smooth boundary defining function for $B_{1/2}$ satisfying $|d\rho_B|_{\bg_B} = 1$ on $\pa B_{1/2}$.
\end{lemma}
\begin{proof}
Clearly $\rho_B(y) > 0$ in $B_{1/2}$ and $\rho_B(y) = (\md(y))_N = 0$ for $y \in \pa B_{1/2} \setminus \{y_s\}$.
Since the decay of $\mw_{1/2}(x)^{2 \over N-2}$ is $|x|^{-2}$ for $|x|$ large enough, the definition $\rho_B(y_s) = 0$ gives the smooth extension of $\rho_B$ to the singularity.

On the other hand, the condition $|d\rho_B|_{\bg_B} = 1$ on $\pa B_{1/2}$ implies that the sectional curvature of $\tg_B = \rho_B^{-2}\bg_B$ approaches to $-1$ at $\pa B_{1/2}$, and vice versa.
Since $\tg$ is equal to the standard hyperbolic metric in $\{|x| \ge 1\}$, the sectional curvature of $\tg_B$ is precisely $-1$ in the neighborhood of $y_s$.
Moreover, we have $|dx_N|_{\bg} = 1$ on $\omrn$, which means that the sectional curvature of $\tg(x)$ goes to $-1$ as $x$ tends to $\mr^n$, so does the sectional curvature of $\tg_B(y)$ as $y$ converges to a point in $\pa B_{1/2} \setminus \{y_s\}$.
\end{proof}
\noindent In summary, $(B_{1/2}, \bg_B)$ is a compact manifold, $\rho_B$ is a smooth defining function for its boundary $\pa B_{1/2}$ and $\tg_B = \rho_B^{-2}\bg_B$ in $B_{1/2}$.
Therefore, according to \cite{ACF}, there is a unique solution $u_B \in C^{N-1}(\overline{B_{1/2}}) \cap C^{\infty}(B_{1/2})$ of
\begin{equation}\label{eq-LN-ball}
-{4(N-1) \over N-2}\Delta_{\tg_B} u + R_{\tg_B} u + N(N-1)u^{N+2 \over N-2} = 0 \quad \text{in } B_{1/2} \quad \text{and} \quad u = 1 \quad \text{on } \pa B_{1/2}.
\end{equation}
Then $u(x) = u_B(\mc(x))$ $\big($for $x \in \omrn\big)$ satisfies \eqref{eq-LN}.

\subsection{Expansions for the solution near the boundary}\label{subsec-app-b}
This subsection is devoted to give account of the derivation of Proposition \ref{prop-rho-est} under the assumption that $N \ge 22$.
To get information on the lower order terms of the expansion for $\rho$ (or equivalently, $u$) in terms of $x_N$, we will inspect the equation that $z := u-1$ satisfies.
We remark that our proof is inspired by Han-Jiang \cite{HJ}.

\medskip
Introduce a linear operator
\[\ml(\tz) = {4(N-1) \over N-2} \left[\Delta_{\bg}\tz - (N-2) x_N^{-1} \pa_N \tz - N x_N^{-2} \tz \right] - R_{\bg} \tz \quad \text{for } \tz \in C^2(\mr^N_+)\]
and a function $g: (-1, \infty) \to [0, \infty)$ by
\begin{equation}\label{eq-g}
g(t) := (1+t)^{N+2 \over N-2} - \(1 + {N+2 \over N-2}t\).
\end{equation}
Then, by employing the relations
\begin{equation}\label{eq-tg}
\Delta_{\tg} = x_N^2\Delta_{\bg} -(N-2)x_N\pa_N \quad \text{and} \quad R_{\tg} = - N(N-1) + R_{\bg}x_N^2
\end{equation}
which are valid due to the condition $H = 0$, one can deduce from \eqref{eq-LN} that $z$ is a solution of $\mq(\tz) = 0$ where $\mq$ is the operator
\begin{equation}\label{eq-rho-est-1}
\mq(\tz) = \ml(\tz) - N(N-1) x_N^{-2} g(\tz) - R_{\bg}.
\end{equation}

To approximate the function $z$ near the boundary $\mr^n$, let us set a polynomial $z_{d_0}$ in the $x_N$-variable,
\begin{equation}\label{eq-z-star}
z_{d_0}(\bx, x_N) = \sum_{m=1}^{2d_0+2} D_{2m}(\bx) x_N^{2m} \quad (1 \le d_0 \le 4)
\end{equation}
where smooth functions $D_{2m}$ in $\mr^n$ are determined in the next lemma.
We also remind that $\Delta_{\bg} = \Delta_{\hh} + \pa_{NN}$ and $R_{\bg}(\bx, x_N) = R_{\hh}(\bx)$ if $x_N > 0$ is small enough.
Then the main order term of $D_{2m}$ will turn out to be equal to $\Delta_{\hh}^{m-1} R_{\hh}$ up to a constant factor.
\begin{lemma}\label{lemma-Q-est}
Let $\mcr(\vr_1, \vr_2) = B^n(0, \vr_1) \times (0, \vr_2)$ be a cylinder in $\mr^N_+$.
Then, for a fixed small number $\vr_2 \in (0,1)$, there exist a constant $C = C(\vr_2) > 0$ (independent of $\vr_1$) and functions $D_{2m} \in C^{\infty}(\mr_+^N)$ for $m = 1, \cdots, 2d_0+2$ satisfying
\[\left|\mq(z_{d_0})(x)\right| \le C x_N^{2(2d_0+2)} \quad \text{for all } x = (\bx, x_N) \in \mcr(\vr_1, \vr_2).\]
\end{lemma}
\begin{proof}
By putting the polynomial $z_{d_0}$ given \eqref{eq-z-star} into \eqref{eq-rho-est-1}, we observe that
\begin{align*}
\mq(z_{d_0}) &= \(-12(N-1)D_2-R_{\hh}\) \\
&\ + \sum_{m=1}^{2d_0+1} \left[{4(N-1) \over N-2} \(\Delta_{\hh}D_{2m} - (2m+3)(N-2(m+1)) D_{2(m+1)}\) - D_{2m} R_{\hh} \right] x_N^{2m}\\
&\ + \({4(N-1) \over N-2} \Delta_{\hh} D_{2(2d_0+2)} - D_{2(2d_0+2)} R_{\hh}\) x_N^{2(2d_0+2)} - N(N-1) x_N^{-2} g\(z_{d_0}\)
\end{align*}
where
\[g\(z_{d_0}\) = {{N+2 \over N-2} \choose 2} D_2^2 x_N^4 + {{N+2 \over N-2} \choose 3} \(2D_2D_4 + D_2^3\) x_N^6 + \cdots\]
is a power series whose coefficients are sums of products of two or more $D_{2m}$'s.
Expanding $\mq(z_{d_0})$ in ascending power of $x_N$ up to the $2(2d_0+2)$-th order yields
\begin{align*}
\mq(z_{d_0}) &= \sum_{m=0}^{2d_0+1} \mg_m\(D_2, \cdots, D_{2m}, \Delta_{\hh}D_{2m}, D_{2(m+1)}, R_{\hh}\) x_N^{2m} \\
&\ + \mg_{2d_0+2} \(D_2, \cdots, D_{2(2d_0+2)}, \Delta_{\hh}D_{2(2d_0+2)}, R_{\hh}\) x_N^{2(2d_0+2)} + O\(x_N^{2(2d_0+3)}\)
\end{align*}
where $\mg_m$ is a function which can be explicitly written
(setting $D_0 = 0$). For example, we have
\[\mg_0 = -12(N-1)D_2-R_{\hh} \quad \text{and} \quad \mg_1 = {4(N-1) \over N-2} \(\Delta_{\hh}D_2 - 5(N-4) D_4\) - D_2 R_{\hh} - N(N-1)D_2^2.\]
Solving the equations $\mg_0 = \cdots = \mg_{2d_0+1} = 0$ inductively, we obtain
\begin{equation}\label{eq-D}
D_2 = - {R_{\hh} \over 12(N-1)} \quad \text{and} \quad D_{2(m+1)} = {\Delta_{\hh}D_{2m} \over (2m+3)(N-2(m+1))} + \mathcal{R}_m\(D_2, \cdots, D_{2m}, R_{\hh}\)
\end{equation}
for $m = 1, \cdots, 2d_0+1$, where the remainder $\mathcal{R}_m$ is a sum of products of two or more its arguments.
By $\eqref{eq-eta_0}$ and $\eqref{eq-R-bg-est-0}$, $\|\mg_{2d_0+2}\|_{L^{\infty}(\mr^n)}$ is controlled by the small number $\eta_0$.
Hence the proof is completed.
\end{proof}
\noindent As a result of the previous lemma, one gets
\begin{equation}\label{eq-L-z-zd}
\ml(z-z_{d_0}) = O\(x_N^{2(2d_0+2)}\) + N(N-1)x_N^{-2}\ell(x)(z-z_{d_0}) \quad \text{where } \ell(x) = {g(z(x))-g(z_{d_0}(x)) \over z(x)-z_{d_0}(x)}.
\end{equation}
Furthermore, given any fixed $\eta_1 > 0$, we may assume that $z(x),\ z_{d_0}(x) \ge -\eta_1$ for all $x \in \mcr(\vr_1, \vr_2)$
by decreasing $\vr_2 > 0$ in the statement of Lemma \ref{lemma-Q-est} and $\eta_0 > 0$ in \eqref{eq-eta_0} if necessary. 
Therefore we have $\ell(x) \ge - C\eta_1$ for $x \in \mcr(\vr_1, \vr_2)$, which makes possible to deduce the comparison principle to the operator
\begin{equation}\label{eq-L1}
\ml_1(\tz) = \ml(\tz) - N(N-1)x_N^{-2} \ell(x) \tz \quad \text{defined for } \tz \in C^2(\mcr(\vr_1, \vr_2)).
\end{equation}
\begin{lemma}\label{lemma-L1-mp}
Choose a small number $\eta_2 > 0$ such that $|R_{\bg}| \le \eta_2$ in $\mr^N_+$, which is possible due to \eqref{eq-eta_0} and \eqref{eq-R-bg-est-0}.
In addition, let $\vr_1 > 0$ be any number and $\vr_2 \in (0,1)$ sufficiently small so that $\ell(x) \ge - C\eta_1$ for every $x \in \mcr(\vr_1, \vr_2)$.
If $\ml_1(\tz_1) \ge \ml_1(\tz_2)$ in the set $\mcr(\vr_1, \vr_2)$ and $\tz_1 \le \tz_2$ on its boundary $\pa \mcr(\vr_1, \vr_2)$,
then $\tz_1 \le \tz_2$ in $\mcr(\vr_1, \vr_2)$.
\end{lemma}
\begin{proof}
Suppose not. Then we can choose a point $x_0 = (\bx_0, (x_0)_N) \in \mcr(\vr_1, \vr_2)$
satisfying $(x_0)_N > 0$, $(\tz_1-\tz_2)(x_0) > 0$, $\nabla(\tz_1-\tz_2)(x_0) = 0$ and $(\Delta_{\hh}(\tz_1-\tz_2) + \pa_{NN}(\tz_1-\tz_2))(x_0) \le 0$. Thus
\begin{align*}
0 &\ge {4(N-1) \over N-2} \left[\Delta_{\hh}(\tz_1-\tz_2) + \pa_{NN}(\tz_1-\tz_2)\right](x_0) \\
&\ge \left[{4(N-1)N \over N-2} \cdot (x_0)_N^{-2} + N(N-1) (x_0)_N^{-2}\ell(x_0) + R_{\hh}(x_0) \right] \cdot \(\tz_1-\tz_2\)(x_0)\\
&\ge \left[N(N-1) \left\{{4 \over N-2} - C\eta_1\right\} (\vr_2)^{-2} - \eta_2 \right] \cdot \(\tz_1-\tz_2\)(x_0) > 0
\end{align*}
provided that $\vr_2^2 \in \(0, \min\left\{1, N(N-1)\left\{{4 \over N-2}-C\eta_1\right\}\eta_2^{-1}\right\}\)$.
Accordingly we reach the contradiction. The lemma should hold.
\end{proof}
\noindent Together with Lemmas \ref{lemma-Q-est} and \ref{lemma-L1-mp}, we are able to estimate the difference between $z$ and its approximation $z_0$.
\begin{lemma}
Fix any $\eta_3 > 0$ and small $\nu > 0$. Then it holds that
\begin{equation}\label{eq-z-error}
\left|z(x) - z_{d_0}(x)\right| \le Cx_N^{2(2d_0+3)-\eta_3} \quad \text{for all } x \in \mcr(\nu, \nu).
\end{equation}
\end{lemma}
\begin{proof}
Its proof will be carried out in three steps.

\medskip \noindent \textsc{Step 1.}
Define
\[z^*(\bx, x_N) = C_1^*\left[\sin\({\pi \over 2\vr_1^2}|\bx|^2\) + {\pi \over 2\vr_1^2} \cos\({\pi \over 2\vr_1^2}|\bx|^2\)\right]+ C_2^* x_N^{2(2d_0+3)-\eta_3}\]
with $C_1^*, C_2^* > 0$ to be determined soon.
We claim that there is $C = C(\vr_2) > 0$ such that
\begin{equation}\label{eq-L1-zstar}
\ml_1(z^*) \le -Cx_N^{2(2d_0+2)-\eta_3} \quad \text{in } \mcr(\vr_1, \vr_2)
\end{equation}
where $(\vr_1, \vr_2)$ is the pair for which Lemma \ref{lemma-L1-mp} is true.

Write $\vr_1^* = 2\vr_1^2/\pi$ for simplicity.
Since $|h_{ab}(x)| = 0$ for $|x| \le 1$ (see \eqref{eq-h}), we see that
\begin{align*}
&\ \Delta_{\hh}\(\sin\(r^2/\vr_1^*\) + \cos\(r^2/\vr_1^*\)/\vr_1^*\) \\
&\le 2\cos\(r^2/\vr_1^*\)/\vr_1^*
+ O\(|h_{ab}| + |Dh_{ab}|\)\(\sin\(r^2/\vr_1^*\)/\(\vr_1^*\)^2 + \cos\(r^2/\vr_1^*\)/\vr_1^*\).
\end{align*}
Therefore we have
\begin{equation}\label{eq-L1-zstar-1}
\begin{aligned}
&\ \ml_1\(\sin\(r^2/\vr_1^*\) + \cos\(r^2/\vr_1^*\)/\vr_1^*\) \\
&\le {4(N-1) \over N-2} \left[\Delta_{\hh} - N\vr_2^{-2}\(1-{N-2 \over 4} C\eta_1\) + {N-2 \over 4(N-1)} \eta_2 \right]\(\sin\(r^2/\vr_1^*\) + \cos\(r^2/\vr_1^*\)/\vr_1^*\) \le 0
\end{aligned}
\end{equation}
where $r = |\bx|$.

Moreover the polynomial $\alpha \in \mr \mapsto \alpha^2 - (N-1)\alpha -N$ has $N$ and $-1$ as its zeros.
Hence given that $N \ge 22$, we compute
\begin{equation}\label{eq-L1-zstar-2}
\ml_1\(x_N^{\alpha_0}\) \le {4(N-1) \over N-2} \left[\(\alpha_0^2 - (N-1)\alpha_0 -N\) + {N-2 \over 4(N-1)} \eta_2\vr_2^2 + {N(N-2) \over 4} C\eta_1 \right] x_N^{\alpha_0-2} \le -\widetilde{C}x_N^{\alpha_0-2}
\end{equation}
in $\mcr(\vr_1, \vr_2)$ where $\alpha_0 = 2(2d_0+3)-\eta_3$.

Consequently \eqref{eq-L1-zstar} follows from \eqref{eq-L1-zstar-1} and \eqref{eq-L1-zstar-2}.

\medskip \noindent \textsc{Step 2.} Combining \eqref{eq-L-z-zd} and \eqref{eq-L1-zstar}, we obtain
\[\ml_1(z^*) - \ml_1\(z-z_{d_0}\) \le -Cx_N^{2(2d_0+2)-\eta_3}\(1-\widetilde{C}x_N^{\eta_3}\) \le 0 \quad \text{in } \mcr(\vr_1, \vr_2)\]
for some $C,\ \widetilde{C} > 0$.
Moreover $\|z-z_{d_0}\|_{L^{\infty}(\{x \in \mr^N_+: x_N \le \vr_2\})}$ is bounded, so we can choose $C_1^*,\ C_2^* > 0$ so large that $z - z_{d_0} \le z^*$ on $\pa \mcr(\vr_1, \vr_2)$.
Thus we infer from the maximum principle in Lemma \ref{lemma-L1-mp} that $z - z_{d_0} \le z^*$ holds in $\mcr(\vr_1, \vr_2)$.
By taking $\vr_1 \to 0$ and $\vr_2 = \nu$, we observe that
$(z - z_{d_0})(x) \le C_2^* x_N^{2(2d_0+3)-\eta_3}$ for $x \in \mcr(\nu,\nu)$.

\medskip \noindent \textsc{Step 3.}
Similarly we have $(z - z_{d_0})(x) \ge -z^*(x)$ for all $x \in \mcr(\vr_1, \vr_2)$.
By letting $\vr_1 \to 0$ again in the square $\mcr(\nu,\nu)$, we conclude that \eqref{eq-z-error} is true.
\end{proof}

\noindent By elliptic regularity, we also obtain decay estimates for the first and second derivatives of $z-z_{d_0}$ (cf. \cite{Lin, HJ}).
\begin{lemma}\label{lemma-z-reg}
There exists a constant $C > 0$ such that
\begin{align*}
\left|D_{\bx}(z(x) - z_{d_0}(x))\right| + \left|D_{\bx}^2(z(x) - z_{d_0}(x))\right| &\le Cx_N^{2(2d_0+3)-\eta_3},\\
\left|\pa_{x_N}(z(x) - z_{d_0}(x))\right| &\le Cx_N^{2(2d_0+2)-\eta_3},\\
\left|\pa_{x_Nx_N}(z(x) - z_{d_0}(x))\right| &\le Cx_N^{2(2d_0+1)-\eta_3}
\end{align*}
for every $x \in \mcr(\nu, \nu)$.
Here $D_{\bx}$ implies the derivative with respect to the $\bx$-variable and so forth.
\end{lemma}

We are now ready to conclude the proof of Proposition \ref{prop-rho-est}.
However, before initiating the proof, it may as well note that the rescaled function $(\Delta_{\hh}^{m-1} R_{\hh})(\ep \bx)$ of the main term of $D_{2m}$
is comparable to $\mu^2 \ep^{4d_0+4-2m} \Delta_{\hh}^{m-1}(\pa_k \oh_{ij})^2(\bx)$ in the set $\{|\bx| \le \nu/\ep\}$ where $\oh_{ij}(\bx) = f(|\bx|^2)H_{ij}(\bx)$.
This is because we have by \eqref{eq-h} and \eqref{eq-R-bg-est} that
\begin{equation}\label{eq-R-bg-est-2}
R_{\hh}(\ep \bx) = -{1 \over 4} \mu^2 \ep^{4d_0+2} \sum_{i,j,k=1}^n \(\pa_k \oh_{ij}(\bx)\)^2 + O\(\mu^3\ep^{6d_0+4}|\bx|^4\(1+|\bx|^{6d_0}\)\) \quad \text{in } C^{\infty}(\{|\bx| \le \nu/\ep\}).
\end{equation}
Since $\pa_k \oh_{ij}$ is a polynomial of degree $2d_0+1$, we have also that $\Delta_{\hh}^{(2d_0+3)-1}(\pa_k \oh_{ij})^2 = 0$.
\begin{proof}[Proof of Proposition \ref{prop-rho-est}]
A combination of \eqref{eq-z-star} and \eqref{eq-z-error} implies
\begin{equation}\label{eq-z-1}
z(\ep x) = \sum_{m=1}^{2d_0+2} D_{2m}(\ep\bx) (\ep x_N)^{2m} + O\((\ep x_N)^{2(2d_0+3)-\eta_3}\) \quad \text{in } B_+^N(0,\nu/\ep).
\end{equation}
Moreover, by virtue of \eqref{eq-hh-inv-est}, it holds
\begin{align*}
\Delta_{\hh(\ep\cdot)}v &= \pa_i\(\hh^{ij}(\ep\cdot) \pa_j v\) = \hh^{ij}(\ep\cdot)\pa_{ij} v + O\(|h(\ep\cdot)||D(h(\ep\cdot))||Dv|\)\\
&= \Delta v + O\(\mu\ep^{2(d_0+1)}|\bx|^2\(1+|\bx|^{2d_0}\)|D^2v|\) + O\(\mu^2\ep^{4(d_0+1)}|\bx|^3\(1+|\bx|^{4d_0}\)|Dv|\)
\end{align*}
for any $v \in C^{\infty}(\mr^n)$, so we get from \eqref{eq-D} and \eqref{eq-R-bg-est-2} that
\begin{equation}\label{eq-z-2}
D_2(\ep\bx) (\ep x_N)^2 = {\mu^2\ep^{4(d_0+1)} \over 48(N-1)} \sum_{i,j,k=1}^n \(\pa_k \oh_{ij}(\bx)\)^2 x_N^2 + O\(\mu \nu^{2(d_0+1)} \cdot \mu^2 \ep^{4(d_0+1)} |\bx|^2\(1+|\bx|^{4d_0}\) x_N^2\)
\end{equation}
and
\begin{equation}\label{eq-z-3}
\begin{aligned}
&\ D_{2(m+1)}(\ep\bx) (\ep x_N)^{2(m+1)} \\
&= \left[ {\ep^{2m} \Delta \(D_{2m}(\ep\bx)\) \over (2m+3)(N-2(m+1))} + O\(\mu^3\ep^{6(d_0+1)} |\bx|^2\(1+|\bx|^{6d_0+2-2m}\)\) \right] x_N^{2(m+1)} \\
&\ + O\(|D_{2m}(\ep\cdot)|^2 \ep^{2(m+1)}x_N^{2(m+1)}\) \\
&= {\mu^2\ep^{4(d_0+1)} \over 48(N-1)} \left[\prod_{\tm=1}^m {1 \over (2\tm+3)(N-2(\tm+1))} \right] \sum_{i,j,k=1}^n \Delta^m\(\pa_k \oh_{ij}(\bx)\)^2 x_N^{2(m+1)} \\
&\ + O\(\(\mu \nu^{2(d_0+1)} + \mu^2 \nu^{4d_0+6-2m}\) \cdot \mu^2\ep^{4(d_0+1)}\(1+|\bx|^{4d_0+2-2m}\)x_N^{2(m+1)}\)
\end{aligned}
\end{equation}
in $B_+^N(0,\nu/\ep)$ for each $m = 1, \cdots, 2d_0+1$.
Subsequently we get from \eqref{eq-z-1}-\eqref{eq-z-3} that
\begin{equation}\label{eq-z-4}
\begin{aligned}
z(\ep x) &= {\mu^2\ep^{4(d_0+1)} \over 48(N-1)} \sum_{m=1}^{2d_0+2} \(\left[\prod_{\tm=1}^{m-1} {1 \over (2\tm+3)(N-2(\tm+1))} \right] \sum_{i,j,k=1}^n \Delta^{m-1} \(\pa_k \oh_{ij}(\bx)\)^2\) x_N^{2m} \\
&\ + O\(\mu^3 \ep^{4(d_0+1)}|x|^2\(1+|x|^{4d_0}\)x_N^2\) + O\((\ep x_N)^{2(2d_0+3)-\eta_3}\)
\end{aligned}
\end{equation}
in $B_+^N(0,\nu/\ep)$.
Since the magnitude of $z(\ep \cdot)$ is $O(\mu^2 \ep^{4(d_0+1)} |\cdot|^{4(d_0+1)})$, the $m$-th power of $z$ for $m \ge 2$ can be ignored.
Accordingly \eqref{eq-rho-est} follows from \eqref{eq-z-4} and $\rho = (1+z)^{-{2 \over N-2}}x_N$.
Lemma \ref{lemma-z-reg} guarantees the $C^2$-validity of \eqref{eq-rho-est}, establishing the proposition.
\end{proof}

\subsection{Global behavior of the solution}\label{subsec-app-c}
Here we investigate the behavior of $z = u-1$ in the whole space $\mr^N_+$ $(N \ge 3)$ where $u$ is the solution of \eqref{eq-LN}.
It is one of the key parts in the proof of Proposition \ref{prop-lin}.

\begin{lemma}\label{lemma-z-est-RN}
Let $\eta_0 > 0$ be a fixed number in \eqref{eq-eta_0}, which can be reduced if needed. 
Then there is a constant $C > 0$ relying only on $N$ such that
\begin{equation}\label{eq-z-ann}
\left|\nabla_{\bx}^m z(x)\right| \le {C\eta_0 x_N^2 \over 1+|x|^{4+m}} \quad \text{and} \quad
\left|\pa_{x_N}^m z(x)\right| \le {C\eta_0 x_N^{2-m} \over 1+|x|^4}
\quad \text{for any } x \in \mr^N_+ \text{ and } m = 0, 1, 2.
\end{equation}
\end{lemma}
\begin{proof}
The notations in Appendix \ref{subsec-app-a} will be used again in this proof.

By \eqref{eq-eta_0}, \eqref{eq-R-bg-est-0} and \eqref{eq-tg}, one can pick a number $C > 0$ (depending only on $N$) so large that the function $\bar{u} := 1 + C \eta_0 x_N^2 \mw_{1/2}^{4 \over N-2}(x)$ satisfy
\[-{4(N-1) \over N-2} \Delta_{\tg} u = {4(N-1) \over N-2} C\eta_0x_N^2 \left[ {2\left\{(N-3)\(1+|\bx|^2\) + 6x_N + (N-3)x_N^2\right\} \over \(|\bx|^2+(x_N+1)^2\)^3} + O(|h|)\right] \ge 0 \]
and
\[R_{\tg} u + N(N-1)u^{N+2 \over N-2} = \left[N(N-1)\(u^{4 \over N-2}-1\) + O\(|Dh|^2\) x_N^2\right] u \ge C \eta_0 x_N^2 \mw_{1/2}^{4 \over N-2}(x) \cdot u \ge 0\]
in $\mr^N_+$. This implies that $\bar{u}$ is a global upper solution to \eqref{eq-LN}.
Similarly, one sees that $\underline{u} := 1 - C \eta_0 x_N^2 \mw_{1/2}^{4 \over N-2}(x)$ is a lower solution to \eqref{eq-LN}.
Because of the conformal equivalence between $\mr^N_+$ and $B_{1/2}$, it follows that the functions $\bar{v} := 1 + C \eta_0 \rho_B^2(y)$
and $\underline{v} := 1 - C \eta_0 \rho_B^2(y)$ in $B_{1/2}$ are super- and sub-solutions of \eqref{eq-LN-ball}.
Therefore the standard monotone iteration scheme produces a solution $v$ of \eqref{eq-LN-ball} such that $\underline{v} \le v \le \bar{v}$ (adapt the proof of Proposition 2.1 in \cite{AM2}).
Since $v = u_B$ by the uniqueness, we get Eq. \eqref{eq-z-ann} with $m = 0$.

The higher regularity results (\eqref{eq-z-ann} with $m = 1,2$) follows from the scaling property of equation $\mq(z) = 0$ in $\mr^N_+$ (cf. \cite[Proposition 3.1]{LZ2}).
\end{proof}

\begin{rmk} 
By applying the maximum principle, it is possible to improve the decay estimate of $z$.
Especially, we observe that $z$ decays faster as the dimension $N$ gets higher: It holds that
\begin{equation}\label{eq-z-est-decay}
|z(x)| \le C \eta_0 x_N^2|x|^{-\(2+\sqrt{3N-2}\)} \quad \text{in } \left\{x \in \mr^N_+: |x| > 1\right\}
\end{equation}
for $C > 0$ depending only on $N$.
\begin{proof}
Set $\bar{z}(x) = \|z\|_{C^2\big(\mr_+^N\big)} x_N^2|x|^{-\alpha}$
in $\{|x| > 1\}$ with any fixed $\alpha \ge 4$.
Since $\bg_{ab} = \delta_{ab}$ if $|x| > 1$ and $g(t) \ge 0$ for any $t \ge -1$ (see \eqref{eq-g} for the definition of $g$), we see
\begin{align*}
\mq(\bar{z}) &\le {4(N-1) \over N-2} \left[\Delta\bar{z} - (N-2) x_N^{-1} \pa_N \bar{z} - N x_N^{-2} \bar{z} \right] \\
&= {4(N-1) \over N-2} \|z\|_{C^2\big(\mr_+^N\big)} |x|^{-\alpha} \left[-3(N-2) + \alpha(\alpha-4) x_N^2|x|^{-2} \right] \le 0 = \mq(z)
\end{align*}
provided that $4 \le \alpha \le 2+\sqrt{3N-2}$.
Moreover, it holds that $z = 0 \le \bar{z}$ in $\{|x| > 1 \text{ and } x_N = 0\}$
and
$|z(x)| \le \|\pa_{NN}z\|_{L^{\infty}(\{x_N < 1\})} x_N^2/2 \le \bar{z}(x)$ in $\{|x| = 1 \text{ and } x_N > 0\}$.

By Lemma \ref{lemma-z-est-RN}, $\|z\|_{C^2\big(\mr_+^N\big)} \le C\eta_0$.
Besides $\bar{z}(x), z(x) \to 0$ uniformly as $|x| \to \infty$,
so we can apply the argument in the proof of Lemma \ref{lemma-L1-mp} to derive that $z(x) \le \bar{z}(x)$ for $|x| > 1$.
Analogously, $z(x) \ge -\bar{z}(x)$ is true for $|x| > 1$, validating \eqref{eq-z-est-decay}.
\end{proof}
\end{rmk}

\begin{proof}[Proof of \eqref{eq-eig}]
Denote by $g^+_{\text{h}}$ by the hyperbolic metric in $\mr^N_+$, i.e., $g^+_{\text{h}} = x_N^{-2}(d\bx^2+dx_N^2)$.
It is well known that
\[\lambda_1\(-\Delta_{g^+_{\text{h}}}\)
= \inf_{V \in C^{\infty}_c(\mr^N_+) \setminus \{0\}} {\int_{\mr^N_+} \(g^+_{\text{h}}\)^{ab} \pa_aV \pa_bV \sqrt{\left|g^+_{\text{h}}\right|} dx \over \int_{\mr^N_+} V^2 \sqrt{\left|g^+_{\text{h}}\right|} dx}
= \inf_{V \in C^{\infty}_c(\mr^N_+) \setminus \{0\}} {\int_{\mr^N_+} x_N^{2-N} |\nabla V|^2 dx \over \int_{\mr^N_+} x_N^{-N} V^2 dx} = {(N-1)^2 \over 4}\]
(see \cite{Mc}).
By \eqref{eq-eta_0} and Lemma \ref{lemma-z-est-RN}, there are a bounded function $\hat{z}$ and a two-tensor $\bar{h}_{ab}$ in $\mr^N_+$
such that $|\bar{h}_{ab}|$ is uniformly bounded, $u^{4 \over N-2} = 1 + \eta_0 \hat{z}$ and $\bg_{ab} = \delta_{ab} + \eta_0 \bar{h}_{ab}$ in $\mr^N_+$.
Hence it follows from the definition of $g^+$ that
\[g^+_{ab} = x_N^{-2}(1 + \eta_0 \hat{z})(\delta_{ab} + \eta_0 \bar{h}_{ab}) := (g^+_{\text{h}})_{ab} + x_N^{-2} \eta_0 \bar{h}'_{ab}.\]
This implies that $\sqrt{|g^+|} = x_N^{-N}(1+O(\eta_0))$ and $(g^+-g^+_{\text{h}})^{ab} = x_N^2 \eta_0 \bar{h}''_{ab}$ for some tensor $\bar{h}''$ in $\mr^N_+$ having the bounded norm.
We obtain accordingly
\[\lambda_1\(-\Delta_{g^+}\) = \inf_{V \in C^{\infty}_c(\mr^N_+) \setminus \{0\}} {(1+O(\eta_0)) \(\int_{\mr^N_+} x_N^{2-N} |\nabla V|^2 dx\) \over (1+O(\eta_0))\(\int_{\mr^N_+} x_N^{-N} V^2 dx\)}
\ge (1-C\eta_0) \cdot \lambda_1\(-\Delta_{g^+_{\text{h}}}\) > {n^2 \over 4} - \gamma^2\]
by choosing $\eta_0 > 0$ small.
\end{proof}

\section{The values of integrals $F_{1,n,\gamma}$, $F_{2,n,\gamma}$ and $F_{3,n,\gamma}$} \label{sec-app-b}
The next lemma enumerate some values of integrals $F_{1,n,\gamma}$, $F_{2,n,\gamma}$ and $F_{3,n,\gamma}$ defined in \eqref{eq-F} which are necessary to calculate the function $J_2^{\gamma}$ in \eqref{eq-wjg} for the case $d_0 = 1$.
It can be derived in a similar manner to Lemma \ref{lemma-W-rel}
However, the computation becomes much more involved, so we carried out it by using Mathematica.
More values required to deal with the case $d_0 = 4$ can be found in the supplement \cite{KMW}.

\begin{lemma} \label{lemma-W-rel-2}
We have
\begin{align*}
F_{1,n,\gamma}(3,0) &= \left|S^{n-1}\right|\left[{8(n-3)\(1-\gamma^2\) \over 3(n-4)(n-2\gamma-4)(n-2\gamma+4)}\right]A_1B_2,\\
F_{1,n,\gamma}(3,2) &= \left|S^{n-1}\right|\left[{8(n-3)n\(1-\gamma^2\)(5(n-3)(n-5)+(1-2\gamma)(1+2\gamma)) \over 15(n-4)(n-6)(n-2\gamma-4)(n-2\gamma+4)(n-2\gamma-6)(n+2\gamma-6)}\right]A_1B_2,\\
F_{1,n,\gamma}(3,4) &= \left|S^{n-1}\right|\left[\tfrac{8(n-3)n(n+2)\(1-\gamma^2\)\(35(n-3)(n-5)^2(n-7)+R_{1,n,\gamma}(3,4)\)} {105(n-4)(n-6)(n-8)(n-2\gamma-4)(n-2\gamma+4)(n-2\gamma-6)(n+2\gamma-6)(n-2\gamma-8)(n+2\gamma-8)}\right]A_1B_2,\\
F_{1,n,\gamma}(5,0) &= \left|S^{n-1}\right|\left[{128(n-5)(n-3)\(4-\gamma^2\)\(1-\gamma^2\) \over 15(n-4)(n-6)(n-2\gamma-4)(n+2\gamma-4)
(n-2\gamma-6)(n+2\gamma-6)}\right]A_1B_2,\\
F_{1,n,\gamma}(5,2) &= \left|S^{n-1}\right|\left[\tfrac{128(n-5)(n-3)n\(4-\gamma^2\)\(1-\gamma^2\)(7(n-3)(n-7)+(1-2\gamma)(1+2\gamma))} {105(n-4)(n-6)(n-8)(n-2\gamma-4)(n-2\gamma+4)(n-2\gamma-6)(n+2\gamma-6)(n-2\gamma-8)(n+2\gamma-8) }\right]A_1B_2,\\
F_{1,n,\gamma}(7,0) &= \left|S^{n-1}\right|\left[\tfrac{1024(n-7)(n-5)(n-3)\(9-\gamma^2\)\(4-\gamma^2\)\(1-\gamma^2\)} {35(n-4)(n-6)(n-8)(n-2\gamma-4)(n-2\gamma+4)(n-2\gamma-6)(n+2\gamma-6)(n-2\gamma-8)(n+2\gamma-8) }\right]A_1B_2,\\
F_{2,n,\gamma}(1,2) &= \left|S^{n-1}\right| \left[{(n+2) \(3(n-1)^2 + \(1-4\gamma^2\)\) \over 12(n-1)}\right] A_1B_2, \\
F_{2,n,\gamma}(1,4) &= \left|S^{n-1}\right| \left[{(n+2)(n+4) \(15(n-1)^2(n-3)^2 + R_{2,n,\gamma}(1,4)\) \over 60(n-1)(n-4)(n-2\gamma-4)(n+2\gamma-4)}\right] A_1B_2, \\ F_{2,n,\gamma}(1,6) &= \left|S^{n-1}\right| \left[ \tfrac{(n+2)(n+4)(n+6) \(35(n-1)^2(n-3)^2(n-5)^2 + R_{2,n,\gamma}(1,6)\)} {140(n-1)(n-4)(n-6)(n-2\gamma-4)(n+2\gamma-4)(n-2\gamma-6)(n+2\gamma-6)}\right] A_1B_2,\\
F_{2,n,\gamma}(3,4) &= \left|S^{n-1}\right| \left[{2(n+2)(n+4)\(1-\gamma^2\)\(35(n-1)(n-3)^2(n-5) + R_{2,n,\gamma}(3,4)\) \over 105(n-4)(n-6)(n-2\gamma-4)(n+2\gamma-4)
(n-2\gamma-6)(n+2\gamma-6)}\right] A_1B_2,\\
F_{2,n,\gamma}(3,6) &= \left|S^{n-1}\right| \left[ \tfrac{2(n+2)(n+4)(n+6)\(1-\gamma^2\)\(105(n-1)(n-3)^2(n-5)^2(n-7) +R_{2,n,\gamma}(3,6)\)}{315(n-4)(n-6)(n-8)(n-2\gamma-4)(n+2\gamma-4)
(n-2\gamma-6)(n+2\gamma-6)(n-2\gamma-8)(n+2\gamma-8)}\right] A_1B_2,\\
F_{2,n,\gamma}(5,0) &= \left|S^{n-1}\right|\left[{32(n-3)\(4-\gamma^2\)\(1-\gamma^2\) \over 15(n-4)(n-2\gamma-4)(n+2\gamma-4)}\right]A_1B_2,\\
F_{2,n,\gamma}(5,2) &= \left|S^{n-1}\right|\left[\tfrac{32(n-3)(n+2) \(4-\gamma^2\)\(1-\gamma^2\)\(7(n-1)(n-5)+(1-2\gamma)(1+2\gamma) \)} {105(n-4)(n-6)(n-2\gamma-4)(n+2\gamma-4)(n-2\gamma-6)(n+2\gamma-6)}\right]A_1B_2,\\
F_{2,n,\gamma}(5,4) &= \left|S^{n-1}\right|\left[\tfrac{32(n-3)(n+2)(n+4) \(4-\gamma^2\)\(1-\gamma^2\) \(21(n-1)(n-3)(n-5)(n-7) + R_{2,n,\gamma}(5,4)\)} {315(n-4)(n-6)(n-8)(n-2\gamma-4)(n+2\gamma-4)
(n-2\gamma-6)(n+2\gamma-6)(n-2\gamma-8)(n+2\gamma-8)}\right]A_1B_2,\\
F_{2,n,\gamma}(7,0) &= \left|S^{n-1}\right|\left[{256(n-5)(n-3) \(9-\gamma^2\)\(4-\gamma^2\)\(1-\gamma^2\) \over 35(n-4)(n-6)(n-2\gamma-4)(n+2\gamma-4)(n-2\gamma-6)(n+2\gamma-6)} \right]A_1B_2,\\
F_{2,n,\gamma}(7,2) &= \left|S^{n-1}\right|\left[\tfrac{256(n-5)(n-3)(n+2) \(9-\gamma^2\)\(4-\gamma^2\)\(1-\gamma^2\) \(9(n-1)(n-7)+(1-2\gamma)(1+2\gamma)\)} {315(n-4)(n-6)(n-8)(n-2\gamma-4)(n+2\gamma-4)
(n-2\gamma-6)(n+2\gamma-6)(n-2\gamma-8)(n+2\gamma-8)}\right]A_1B_2,\\
F_{2,n,\gamma}(9,0) &= \left|S^{n-1}\right|\left[\tfrac{8192(n-7)(n-5)(n-3) \(16-\gamma^2\)\(9-\gamma^2\)\(4-\gamma^2\)\(1-\gamma^2\)} {315(n-4)(n-6)(n-8)(n-2\gamma-4)(n+2\gamma-4)
(n-2\gamma-6)(n+2\gamma-6)(n-2\gamma-8)(n+2\gamma-8)}\right]A_1B_2,\\
F_{3,n,\gamma}(3,4) &= \left|S^{n-1}\right| \left[{2(n+2)(2-\gamma)(1-\gamma)\(35(n-1)(n-3)^2(n-4)(n-5)
-R_{3,n,\gamma}(3,4)\) \over 105(n-4)(n-6)(n-2\gamma-4)(n+2\gamma-4)
(n-2\gamma-6)(n+2\gamma-6)}\right] A_1B_2,\\
F_{3,n,\gamma}(3,6) &= \left|S^{n-1}\right| \left[\tfrac{ 2(n+2)(n+4)(2-\gamma)(1-\gamma)\(105(n-1)(n-3)^2(n-5)^2(n-6)(n-7)-R_{3,n,\gamma}(3,6)\)} {315(n-4)(n-6)(n-8)(n-2\gamma-4)(n+2\gamma-4) (n-2\gamma-6)(n+2\gamma-6)(n-2\gamma-8)(n+2\gamma-8)} \right] A_1B_2,\\
F_{3,n,\gamma}(5,0) &= \left|S^{n-1}\right|\left[{32(n-3)(3-\gamma)(2-\gamma)\(1-\gamma^2\) \over 15(n-4)(n-2\gamma-4)(n+2\gamma-4)}\right]A_1B_2,\\
F_{3,n,\gamma}(5,2) &= \left|S^{n-1}\right|\left[\tfrac{32(n-3)(3-\gamma)(2-\gamma) \(1-\gamma^2\)\(7(n-1)(n-2)(n-5) - R_{3,n,\gamma}(5,2)\)} {105(n-4)(n-6)(n-2\gamma-4)(n+2\gamma-4)
(n-2\gamma-6)(n+2\gamma-6)} \right]A_1B_2,\\
F_{3,n,\gamma}(5,4) &= \left|S^{n-1}\right|\left[\tfrac{32(n-3)(n+2)(3-\gamma)(2-\gamma) \(1-\gamma^2\)\(21(n-7)(n-5)(n-4)(n-3)(n-1) - R_{3,n,\gamma}(5,4)\)} {315(n-4)(n-6)(n-8)(n-2\gamma-4)(n+2\gamma-4)
(n-2\gamma-6)(n+2\gamma-6)(n-2\gamma-8)(n+2\gamma-8)}\right]A_1B_2,\\
F_{3,n,\gamma}(7,0) &= \left|S^{n-1}\right|\left[{256(n-5)(n-3) (4-\gamma)(3-\gamma)\(4-\gamma^2\)\(1-\gamma^2\) \over 35(n-4)(n-6)(n-2\gamma-4)(n+2\gamma-4)(n-2\gamma-6)(n+2\gamma-6)} \right]A_1B_2,\\
F_{3,n,\gamma}(7,2) &= \left|S^{n-1}\right|\left[\tfrac{256(n-5)(n-3) (4-\gamma)(3-\gamma)\(4-\gamma^2\)\(1-\gamma^2\)\(9(n-7)(n-2)(n-1) - R_{3,n,\gamma}(7,2)\)} {315(n-4)(n-6)(n-8)(n-2\gamma-4)(n+2\gamma-4) (n-2\gamma-6)(n+2\gamma-6)(n-2\gamma-8)(n+2\gamma-8)} \right]A_1B_2,\\
F_{3,n,\gamma}(9,0) &= \left|S^{n-1}\right|\left[\tfrac{8192(n-7)(n-5)(n-3) (5-\gamma)(4-\gamma)\(9-\gamma^2\)\(4-\gamma^2\)\(1-\gamma^2\)} {315(n-4)(n-6)(n-8)(n-2\gamma-4)(n+2\gamma-4) (n-2\gamma-6)(n+2\gamma-6)(n-2\gamma-8)(n+2\gamma-8)} \right]A_1B_2
\end{align*}
for $n > 2\gamma + 8$ where
\begin{align*}
R_{1,n,\gamma}(3,4) &= \(1-4\gamma^2\)\left[14n^2 - 140n + 377 - 12\gamma^2\right],\\
R_{2,n,\gamma}(1,4) &= \(1-4\gamma^2\)\left[10n^2 - 40n + 57 - 12\gamma^2\right],\\
R_{2,n,\gamma}(1,6) &= \(1-4\gamma^2\)\left[35n^4 - 420n^3 + 1939n^2 - 4074n + 3645 \right. \\
&\hspace{120pt} \left. + 80\gamma^4 - 4\gamma^2\(21n^2-126n+275\)\right],\\
R_{2,n,\gamma}(3,4) &= \(1-4\gamma^2\)\left[14n^2 - 84n + 153 - 12\gamma^2\right],\\
R_{2,n,\gamma}(3,6) &= \(1-4\gamma^2\)\left[9\(7n^4 - 112n^3 + 685n^2 - 1896n + 2105\) \right. \\
&\hspace{120pt} \left. + 80 \gamma^4 - 4\gamma^2 \(27n^2 - 216n + 575\)\right],\\
R_{2,n,\gamma}(5,4) &= \(1-4\gamma^2\)\left[6n^2-48n+99-4\gamma^2\right],\\
R_{3,n,\gamma}(3,4) &= (1-2\gamma) \left[42n^3 - 532n^2 + 2103n - 2844 + 24 \gamma^3 (n+4) \right.\\
&\hspace{120pt} \left. - 84 \gamma^2 (n-4) - 14\gamma \(2n^3-12n^2+15n+36\)\right],\\
R_{3,n,\gamma}(3,6) &= (1-2\gamma) \left[9\(21n^5 - 518n^4 + 4931n^3 - 22922n^2 + 52567n - 48810\) - 160 \gamma^5 (n+6) \right.\\
&\hspace{120pt} + 80 \gamma^4 (11n-42) + 8 \gamma^3 \(27n^3 - 162n^2 - 97n + 2550\) \\
&\hspace{120pt} - \gamma^2 \(756n^3 - 10584n^2 + 50308n - 82200\) \\
&\hspace{120pt} \left.- 18 \gamma \(7n^5 - 126n^4 + 861n^3 - 2534n^2 + 2153n + 2730\)\right],\\
R_{3,n,\gamma}(5,2) &= (1-2\gamma)\left[(3n-22)-2\gamma(n+2)\right],\\
R_{3,n,\gamma}(5,4) &= (1-2\gamma)\left[3\(6n^3-104n^2+543n-892\) + 8\gamma^3(n+4) \right.\\
&\hspace{120pt} \left. - 4\gamma^2\(3n^3+7n-44\) + 2\gamma\(48n^2-83n-116\) \right],\\
R_{3,n,\gamma}(7,2) &= (1-2\gamma)\left[3(n-10)-2\gamma(n+2) \right].
\end{align*}
\end{lemma}

\bigskip \noindent \textbf{Acknowledgement.}
S. Kim is indebted to Professor M. d. M. Gonz\'alez and Dr. W. Choi for their valuable comments.
Also, part of the paper was written when he was visiting the University of British Columbia and Universit\`{a} di Torino.
He appreciates the both institutions, and especially Professor S. Terracini, for their hospitality and financial support.
He is supported by FONDECYT Grant 3140530.
The research of M. Musso has been partly supported by FONDECYT Grant 1120151 and Millennium Nucleus Center for Analysis of PDE, NC130017.
The research of J. Wei is partially supported by NSERC of Canada.

\Addresses
\end{document}